\documentclass[reqno]{amsart}
\usepackage{amsmath}
\usepackage{paralist}
\usepackage{amsfonts}
\usepackage{amssymb}
\usepackage{amsthm,amsmath}
\usepackage{amscd}
\usepackage{amsmath}
\usepackage{float}
\usepackage{tikz}
\usepackage{graphicx}
\usepackage[colorlinks=true]{hyperref}
\hypersetup{urlcolor=blue, citecolor=red}
\usepackage{hyperref}
\usepackage{mathabx}

\usepackage[
top    = 2.5cm,
bottom = 2.5cm,
left   = 2.50cm,
right  = 2.50cm]{geometry}

\newtheorem{theorem}{Theorem}[section]
\newtheorem{corollary}[theorem]{Corollary}

\newtheorem{lemma}[theorem]{Lemma}
\newtheorem{proposition}[theorem]{Proposition}

\theoremstyle{definition}
\newtheorem{definition}[theorem]{Definition}
\newtheorem{remark}[theorem]{Remark}

\newcommand{\R}{\mathbb{R}}
\newcommand{\T}{\mathbb{T}}
\newcommand{\Z}{\mathbb{Z}}
\newcommand{\N}{\mathbb{N}}
\newcommand{\K}{\mathbb{K}}
\newcommand{\D}{\mathbb{D}}
\newcommand{\C}{\mathbb{C}}

\author[Sebastian Herr]{Sebastian Herr}
\address{Fakult\"at f\"ur Mathematik, Universit\"at Bielefeld, Postfach 10 01 31, 33501 Bielefeld, Germany}
\email{herr@math.uni-bielefeld.de}

\author[Akansha Sanwal]{Akansha Sanwal}
\address{Fakult\"at f\"ur Mathematik, Universit\"at Bielefeld, Postfach 10 01 31, 33501 Bielefeld, Germany}
\email{asanwal@math.uni-bielefeld.de}

\author[Robert Schippa]{Robert Schippa*}
\address{Karlsruhe Institute of Technology, Englerstrasse 2, 76131 Karlsruhe, Germany
}
\email{robert.schippa@kit.edu}

\makeatletter
\@namedef{subjclassname@2020}{%
  \textup{2020} Mathematics Subject Classification}
\makeatother

\thanks{*Corresponding author}

\begin{document}
	\title{Low regularity well-posedness of KP-I equations: the three-dimensional case}

	\begin{abstract} In this paper, low regularity local well-posedness results for the Kadomtsev--Petviashvili--I equation posed in spatial dimension $d =3$ are proved.
		Periodic, non-periodic and mixed settings as well as generalized dispersion relations are considered. In the weak dispersion regime, these initial value problems show a quasilinear behavior so that
		bilinear and energy estimates on frequency dependent time scales are used in the analysis.
	\end{abstract}
	
	\maketitle
	\section{Introduction}
	We consider the Cauchy problem for the dispersion generalized Kadomtsev--Petviashvili--I equation
	\begin{equation}
		\label{eq:FKPI}\tag{KP-I}
		\left. \begin{split}
			\partial_t u - \partial_{x} D_x^\alpha u - \partial_{x}^{-1} \Delta_y u &=  \partial_x (u^2),  \\
			u(0) &= \phi \in E^s(\D),
		\end{split} \right\}
	\end{equation}
	for real-valued functions $u:\D \times \R \rightarrow \R$ and dispersion parameter $\alpha \in [2,4]$.
	We treat the case of three-dimensional spatial domains $\D := \K_1 \times \K_2 \times \K_3$, where $\K_i \in \{\R ; \T \}$, for $\T := \R/(2\pi \Z)$. We write $u=u(x,y,t)$, for $t \in \R$, $x\in \K_1$, and $y=(y_1,y_2) \in \K_2 \times \K_3$ so that $ \Delta_y=\partial_{y_1}^2+\partial_{y_2}^2$. $D^{\alpha}_x$ is defined via the Fourier transform
	\begin{equation*}
		(D_x^{\alpha} f) \widehat{\,} (\xi,\eta) = |\xi|^{\alpha} \hat{f}(\xi,\eta),
	\end{equation*}
	If $\D = \T \times \K_2 \times \K_3$, then we additionally assume that
	\begin{equation*}
		\int_{\T} u(x,y) dx = 0
	\end{equation*}
	to make the linear evolution well-defined, and the nonlinear evolution preserves mean zero.
	Formally, real-valued solutions also conserve the $L^2(\D)$-norm and the energy (Hamiltonian)
	\begin{equation}\label{eq:e}\mathbf{e}(u)=
		\int_\D \frac12 |D_x^{\frac{\alpha}{2}} u|^2+\frac13 u^3 + \frac12 |\partial_x^{-1} \nabla_{y}  u|^2 dxdy.
	\end{equation}
	
	In two dimensions, the KP-I equation has been introduced as a model for the propagation of dispersive waves with small amplitude under weakly transverse effects \cite{KP} and later has been found to be completely integrable \cite{D74}. In this paper, we address the local well-posedness of the Cauchy problem in the three-dimensional case. In \cite{CR10}, it was rigorously derived in the long wave transonic limit regime (where the amplitude is close to one) from the Gross-Pitaevskii equation after suitable rescaling, see also \cite{KS12} for further information and references.

	We consider initial data in Sobolev spaces $E^s(\D)$ of real valued functions, which are defined for $s \geq 0$ via
	\begin{equation*}
		\| f \|_{E^{s}(\D)} := \| \langle \xi \rangle^s p(\xi,\eta) \hat{f}(\xi,\eta) \|_{L^2_{\xi,\eta}},
	\end{equation*}
	with respect to the three-dimensional product of Lebesgue or counting measures, see Section \ref{section:Notations} for more details. The symbol $p(\xi,\eta)= 1 + \frac{|\eta|}{|\xi|}$ is motivated by the last term in energy \eqref{eq:e} and crucially used in commutator arguments, see Section \ref{section:EnergyEstimates}.
	
	\medskip
	
	The main result of this paper is
	\begin{theorem}
		\label{thm:LWPAnisotropic}
		\begin{enumerate}[(i)]\item
			Let  $\alpha=2$. The Cauchy problem \eqref{eq:FKPI} is locally well-posed for initial data $\phi
			\in E^{s}(\D)$ provided that  $s>2$.
			\item Let  $\alpha\in (2,4)$ and $\D = \K_1 \times \R^2$.  The Cauchy problem \eqref{eq:FKPI} is locally well-posed for initial data $\phi
			\in E^{s}(\D)$ provided that $s> 3 - \frac{\alpha}{2}$.
		\end{enumerate}
	\end{theorem}
	
	By local well-posedness, we refer to existence, uniqueness, and continuous dependence of solutions in function spaces $F^{s}(T) \hookrightarrow C([-T,T];E^{s}(\D))$.
	In two dimensions, \eqref{eq:FKPI} is known to be quasilinear for $\alpha=2$ in the sense that it cannot be solved by the contraction mapping principle.
	In order to improve the standard energy method, Ionescu--Kenig--Tataru \cite{IKT} introduced an approach based on localizations to frequency dependent time scales.
	Roughly speaking, this is the strategy we will use for the proof of the main result. The use of frequency-dependent time localization is justified in Appendix \ref{sec:IllPosed},
	where we prove on $\D=\R^3$ that for certain $\alpha \geq 2$ \eqref{eq:FKPI} cannot be solved by Picard iteration in anisotropic Sobolev spaces.
	The proof of well-posedness crucially makes use of short-time bilinear estimates,
	which become effective for resonant interactions (which are transverse).
	
	In the proof of Theorem \ref{thm:LWPAnisotropic},
	we use frequency-dependent time localization to overcome the derivative loss in the resonant interaction. For $\alpha \in [2,4)$ we choose
	\begin{equation*}
		T(N) = N^{-(2 - \frac{\alpha}{2})}.
	\end{equation*}
	This is again motivated by the bilinear Strichartz estimate we can prove in the resonant case on domains $\K \times \R^2$. This suffices to ameliorate one derivative in case a high frequency interacts with a low one in the resonant case.
	
	We remark that in the case $\alpha \in (2,4)$ one also obtains the well-posedness result with regularity threshold $s>2$ on general domains. This follows by choosing the time scale $T(N)=N^{-1}$ instead, we omit the details.
	However, it is not clear how to improve the regularity threshold $s>2$ on general domains in the case $\alpha \in (2,4)$.
	
	Since $\Delta_y$ is $\mathrm{O}(2)$-invariant, the result extends to more general $y-$domains with $d_1$ non-periodic and $d_2$ periodic transverse directions, where $d_1 + d_2=2$.
	
	\medskip
	
	The second main result is for the fifth-order KP-I equation, where we can use perturbative arguments. We show a  well-posedness result in anisotropic Sobolev spaces $H^{s_1,s_2}(\D)$, which are defined for $s_1,s_2\geq 0$ via
	\begin{equation*}
		\| \phi \|_{H^{s_1,s_2}(\D)} =\|\langle \xi \rangle^{s_1} \langle \eta \rangle^{s_2}\hat{\phi}(\xi,\eta) \|_{L^2_{\xi,\eta}}.
	\end{equation*}
	\begin{theorem}
		\label{thm:SemilinearLWP}
		Let $\alpha = 4$, $s_1,s_2 > \frac{1}{2}$, and $\D \in \{ \R^3; \T \times \R^2 \}$. The Cauchy problem \eqref{eq:FKPI} is locally well-posed for initial data $\phi \in H^{s_1,s_2}(\D)$, with a real-analytic flow map.
	\end{theorem}

	\medskip
	
	If $u$ solves \eqref{eq:FKPI}, then so does
	\begin{equation}\label{eq:ScaleInvariant}
		u_{\lambda}(x,y,t) := \lambda^{- \frac{2\alpha}{\alpha + 2}} u(\lambda^{-\frac{2}{\alpha + 2}} x, \lambda^{-1} y, \lambda^{-\frac{2(\alpha+1)}{\alpha+2}} t),
	\end{equation}
	with scaled initial data $\phi_{\lambda}(x,y):=\lambda^{- \frac{2\alpha}{\alpha + 2}} \phi(\lambda^{-\frac{2}{\alpha + 2}} x, \lambda^{-1} y)$.
	For the corresponding homogeneous norms, we observe that
	\begin{equation}
		\|\phi_{\lambda}\|_{\dot{H}^{s_1,s_2}(\R^3)} = \lambda^{\frac{3-\alpha}{\alpha+2} - \frac{2 s_1}{\alpha+2} - s_2  } \|\phi\|_{\dot{H}^{s_1,s_2}(\R^3)}.
	\end{equation}
	Notice that for high frequencies, $E^s$ corresponds to $H^{s,0}\cap H^{s-1,1}$
	which indicates that the regularity assumptions both in Theorem \ref{thm:LWPAnisotropic} and in Theorem \ref{thm:SemilinearLWP} are sub-critical.

	Let
	\begin{equation*}
		s(\alpha) = 3 - \frac{\alpha}{2}.
	\end{equation*}
	We shall prove the following set of estimates for $s' \geqslant s >s(\alpha)$:
	\begin{equation*}
		\left.	\begin{array}{cl}
			\| u \|_{F^{s'}(T)} &\lesssim \| u \|_{B^{s'}(T)} + \| \partial_x (u^2) \|_{\mathcal{N}^{s'}(T)}, \\
			\| \partial_x (u^2) \|_{\mathcal{N}^{s'}(T)} &\lesssim \| u \|_{F^s(T)} \| u \|_{F^{s'}(T)}, \\
			\| u \|^2_{B^{s'}(T)} &\lesssim \| u_0 \|^2_{E^{s'}} + \| u \|^2_{F^s(T)} \| u \|_{F^{s'(T)}}.
		\end{array} \right\}
	\end{equation*}
	This yields a priori estimates in $E^{s'}$ for $s' \geq s$ and small initial data. The large data case will be handled by rescaling.
	
	The second set of estimates yields Lipschitz continuous dependence at the regularity $E^0$ depending on the norms of the initial data in $E^s$: Let $v = u_1 - u_2$ be a difference of two solutions. Then, we find the following estimates to hold:
	\begin{equation*}
		\left. \begin{array}{cl}
			\| v \|_{F^0(T)} &\lesssim \| v \|_{B^0(T)} + \| \partial_x (v (u_1+u_2)) \|_{\mathcal{N}^0(T)} \\
			\| \partial_x (v(u_1 + u_2)) \|_{\mathcal{N}^0(T)} &\lesssim \| v \|_{F^0(T)} (\| u_1 \|_{F^s(T)} + \| u_2 \|_{F^s(T)}) \\
			\| v \|^2_{B^0(T)} &\lesssim \| v(0) \|^2_{E^0} + \| v \|^2_{F^0(T)} ( \|u_1 \|_{F^s(T)} + \| u_2 \|_{F^s(T)}).
		\end{array} \right\}
	\end{equation*}
	Finally, we prove continuous dependence in $E^s$ using frequency envelopes (cf. \cite{IfrimTataru2020}).
	
	\medskip
	
	We remark that in the companion paper \cite{SS22} the second and third author address the dispersion-generalized KP-I equation in the case $\D=\R^2$.  Depending on the dispersion, that problem exhibits semi- and quasilinear behaviour.
	
	\medskip
	
	Concerning the proof of Theorem \ref{thm:SemilinearLWP}, we notice that it is a semilinear problem which can be treated by a direct iterative method.
	However, in order to avoid a derivative loss, we need to work with critical norms involving $U^2$ and $V^2$ spaces.
	
	\subsection*{Outline}
	In Section \ref{section:Notations}, we introduce notation and function spaces and prove linear estimates. In Section \ref{section:Resonance}, we study the resonance relation and prove bilinear estimates.
	In Section \ref{section:ShorttimeBilinear} we prove the short-time nonlinear estimates. In Section \ref{section:EnergyEstimates}, the energy estimates are proved. The proof of Theorem \ref{thm:LWPAnisotropic} is then concluded in Section \ref{section:Proof}. In Section \ref{section:lwp5} the fifth-order problem is treated. In Appendix  \ref{sec:anisotropicLeibniz}, we show a fractional anisotropic Leibniz rule on mixed domains, and in Appendix \ref{sec:IllPosed} we treat semilinear ill-posedness issues.

	\section{Notations and function spaces}\label{section:Notations}
	\subsection{Fourier transform}
	To prove local well-posedness for large data, we shall rescale the domain. This requires us to consider the rescaled tori $\T_\lambda = \R / (2\pi \lambda\Z)$ for $\lambda \geq 1$. We have to keep track of the dependence of the estimates with respect to $\lambda$. All quantities are defined so that Plancherel's theorem is valid with $\lambda$-independent constant.
	Let
	\begin{equation*}
		\D_\nu = \K_\nu \times (\R^{d_1} \times \T_\lambda^{d_2}), \text{ where } \K_\lambda \in \{ \R; \T_\lambda \}, \quad  d_1,d_2\in \N_0:= \N \cup \{ 0 \}, \quad \nu = \lambda^{\frac{2}{\alpha+2}}, \quad d_1 + d_2 = 2.
	\end{equation*}
	By symmetry of the equation in $y_1,y_2$
	we can assume this specific order to cover all cases considered in the main results.
	For the dual space, we let $\Z_\lambda = \{ \frac{k}{\lambda} : k \in \Z \}$. The dual domain is defined by
	\begin{equation*}
		\D_\lambda^* := \{(\xi,\eta) \in \K_\nu^* \times (\R^{d_1} \times \Z_\lambda^{d_2})\}, \text{ where } \T_\lambda^*=\Z_\lambda \text{ and }\R^*=\R.
	\end{equation*}
	We define the Fourier transform $\hat{f}: \D_\lambda^*\to \C$ of a function $f \in L^1(\D_\lambda)$ by
	\begin{equation*}
		\hat{f}(\xi,\eta) = \int_{\D_\lambda}  e^{-ix\xi} e^{-iy\cdot\eta} f(x,y)dx dy.
	\end{equation*}
	Its inverse is given by
	\begin{equation*}
		\check{f}(x,y) = \frac{1}{(2 \pi)^3} \int_{\D_\lambda^*} e^{i x\xi} e^{iy\cdot\eta} \hat{f}(\xi,\eta) d\xi d\eta,
	\end{equation*}
	where we use the normalized counting measure in $\Z_\lambda$, i.e.,
	\begin{equation*}
		\int_{\Z_\lambda} f(k) dk := \frac{1}{\lambda} \sum_{k \in \Z_\lambda} f(k),
	\end{equation*}
	and Lebesgue measure in $\R$ coordinates.
	In this setting, Plancherel's theorem gives
	\begin{equation*}
		\|f\|_{L_{x,y}^2(\D_\lambda)} =(2\pi)^{-\frac{3}{2}}\|\hat{f}\|_{L^2_{\xi,\eta}(\D_\lambda^*)},
	\end{equation*}
	Using this notation, the space-time Fourier transform and its inverse are given by
	\begin{equation*}
		\begin{split}
			\mathcal{F}f(\xi,\eta,\tau)&
			= \int_{\R} \widehat{f(\cdot,\cdot,t)}(\xi,\eta) e^{-it\tau} dt\\
			\mathcal{F}^{-1}f(x,y,t)&
			=\frac{1}{2\pi}\int_{\R} \widecheck{f(\cdot,\cdot,\tau)}(x,y) e^{it\tau} d\tau.
		\end{split}
	\end{equation*}
	
	The dispersion relation of the dispersion-generalized KP-I equation is denoted by
	\begin{equation}
		\label{eq:DispersionRelation}
		\omega_\alpha(\xi,\eta) = |\xi|^\alpha \xi + \frac{|\eta|^2}{\xi}, \quad \xi \in \K_\nu^* \backslash \{0\}, \quad \eta \in \R^{d_1} \times \Z_\lambda^{d_2}.
	\end{equation}
	Define the group of unitary operators on $L^2(\D_\lambda)$ (restricted to mean zero if $x \in \T_\lambda$) by
	\begin{equation*}
		\widehat{S_\alpha(t)u_0}(\xi,\eta) = e^{it\omega_\alpha(\xi,\eta)}\widehat{u_0}(\xi,\eta).
	\end{equation*}
	This is the propagator for the linear homogeneous equation
	\begin{equation*}
		\partial_t u - \partial_{x} D_x^\alpha u - \partial_{x}^{-1} \Delta_y u =0.
	\end{equation*}

	\subsection{Function spaces}
	\label{subsection:FunctionSpaces}
	Following  \cite{IKT, GO}, we introduce the short-time $X^{s,b}$ spaces now. We also refer to \cite[Section 2.5]{rsc} for an overview of their properties.
	
	Let $\phi_1 \subseteq C^\infty_c(-2,2)$ be symmetric and decreasing on $[0,\infty)$ with $\phi_1(\xi) = 1$ for $\xi \in [-1,1]$.	 We set for $N \in 2^{\N}$, $\phi_N(\xi) = \phi(\xi/N) - \phi(2\xi/N)$. This yields
	\begin{equation*}
		\sum_{N \in 2^{\N_0}} \phi_N\equiv 1.
	\end{equation*}
	We define inhomogeneous Littlewood-Paley projections: For $f \in \mathcal{S}'(\D_\lambda)$, and $N,K \in 2^{\N_0}$, let
	\begin{equation*}
		\widehat{P_N f} (\xi,\eta) = \phi_N(\xi) \hat{f}(\xi,\eta),
		\text{ and }	\widehat{P_{N,K} f} (\xi,\eta) = \phi_K(\eta) \widehat{P_N  f}(\xi,\eta).		\end{equation*}
	Here we abuse notation by writing $\phi_K(\eta)$ instead of $\phi_K(|\eta|)$.
	
	We define an inhomogeneous decomposition of Fourier space by $(A_N)_{N \in 2^{\N_0}}$:
	\begin{align*}
		A_N &= \{ (\xi,\eta) \in \D_\lambda^* : |\xi| \in [\frac{N}{8},8N] \}, \qquad N \in 2^{\N}, \; \\
		A_{1} &= \{ (\xi,\eta) \in \D_\lambda^* : |\xi| \leq 2 \}.
	\end{align*}
	The corresponding homogeneous decomposition is denoted by
	\begin{equation*}
		\tilde{A}_N = \{ (\xi,\eta) \in \D_\lambda^* : \frac{N}{8} \leq |\xi| \leq 8N \}.
	\end{equation*}
	We also consider the double (inhomogeneous) decomposition $(A_{N,K})_{N \in 2^{\N_0}, K \in 2^{\N_0} }$:
	\begin{align*}
		A_{N,K} &= A_N \cap \{ (\xi,\eta) \in \D_\lambda^* : \frac{K}{8} \leq |\eta| \leq 8 K \}, \qquad K \in 2^{\N_0}, \\
		A_{N,1} &= A_N \cap \{ (\xi,\eta) \in \D_\lambda^* : | \eta| \leq 2 \}.
	\end{align*}
	The corresponding decomposition which is homogeneous in $\xi$ (and inhomogeneous in $\eta$) is denoted by $(\tilde{A}_{N,K})_{N \in 2^{\Z}, K \in 2^{\N_0} }$ with $A_N$ replaced by $\tilde{A}_N$ in the previous display. In the following we write $\rho_J(\tau) = \phi_J(\tau)$ for $J \in 2^{\N_0}$ and let $\rho_{\leq J}(\tau) = \sum_{L \leq J} \rho_L(\tau)$, and for $J>1$, $\rho^\sharp_J=\rho_{J/2}+\rho_J+\rho_{2J} $, whereas $\rho_1^{\sharp} = \rho_1 + \rho_2$ (therefore $\rho^\sharp_J\rho_J=\rho^\sharp_J$). For $N,K,J \in 2^{\N_0}$, we define
	\begin{equation*}
		D_{N,K,J} = \{ (\xi,\eta,\tau) \in \D_\lambda^* \times \R : (\xi,\eta) \in A_{N,K} ,\, |\tau - \omega_\alpha(\xi,\eta)| \in  \mathrm{supp}(\rho^\sharp_J) \}, \;
		D_{N,K,\leq J} = \bigcup_{L \leq J} D_{N,K,L} \}.
	\end{equation*}
	As homogeneous counterpart for $N \in 2^{\Z}$, $K \in 2^{\N_0}$ we let
	\begin{equation*}
		\tilde{D}_{N,K,J} = \{ (\xi,\eta,\tau) \in \D_\lambda^* \times \R : (\xi,\eta) \in \tilde{A}_{N,K} ,\, |\tau - \omega_\alpha(\xi,\eta)| \in \mathrm{supp}(\rho^\sharp_J) \}.
	\end{equation*}
	We let
	\begin{equation*}
		X_N = \{ f \in L^2(\D_\lambda^* \times \R) : f \text{ is supported in } A_N \times \R, \; \| f \|_{X_N} < \infty \},
	\end{equation*}
	and
	\begin{equation*}
		\| f \|_{X_N} = \sum_{L \in 2^{\N_0}} L^{\frac{1}{2}} \| \rho_L(\tau - \omega_{\alpha}(\xi,\eta))f \|_{L^2_\tau L^2_{\xi,\eta}}.
	\end{equation*}
	Note that
	\begin{equation*}
		\Big\| \int_{\R} | f(\xi,\eta,\tau)| d\tau \Big\|_{L^2_{\xi,\eta}} \lesssim \| f \|_{X_{N}},
	\end{equation*}
	and we record the estimate
	\begin{equation}\label{prop}
		\begin{split}
			&\sum_{L \geq M} L^{\frac{1}{2}} \| \rho_L(\tau-\omega_{\alpha}(\xi,\eta)) \int |f(\xi,\eta,\tau')| L^{-1} (1+L^{-1} | \tau - \tau'|)^{-4} d \tau' \|_{L_{\xi,\eta,\tau}^2} \\
			&\quad + M^{\frac{1}{2}} \| \rho_{\leq M}(\tau - \omega_{\alpha}(\xi,\eta)) \int |f(\xi,\eta,\tau')| M^{-1} (1+M^{-1} |\tau - \tau'|)^{-4} d\tau' \|_{L_{\xi,\eta,\tau}^2} \\
			&\lesssim \| f \|_{X_N}.
		\end{split}
	\end{equation}
	For a given Schwartz function $\gamma \in \mathcal{S}(\R)$, the estimate
	\begin{equation*}
		\| \mathcal{F}_{x,y,t} [\gamma(M^{2-\frac{\alpha}{2}} (t - t_0)) \mathcal{F}^{-1}_{x,y,t}(f)] \|_{X_N} \lesssim_\gamma \| f \|_{X_N}.
	\end{equation*}
	holds for all  $M,N \in 2^{\N_0}$, $t_0 \in \R$, $f \in X_N$.
	We put the space-time Fourier transform of the original function into the $X_N$-spaces. Let
	\begin{equation*}
		E_N = \{ \phi : \D_\lambda \to \R \, : \, \hat{\phi} \text{ is supported in } A_N, \; \| \phi \|_{E_N} = \| \phi \|_{L^2} < \infty \}.
	\end{equation*}
	
	Next, define
	\begin{equation*}
		F_N = \{ u \in C(\R;E_N) \, : \, \| u \|_{F_N} = \sup_{t' \in \R} \| \mathcal{F}_{x,y,t}[u \cdot \rho_1(N^{(2-\frac{\alpha}{2})}(t-t'))] \|_{X_N} < \infty \}.
	\end{equation*}
	We place the solution into these short-time function spaces after dyadic frequency localization. For the nonlinearity, we consider correspondingly
	\begin{equation*}
		\mathcal{N}_N = \{ u \in C(\R;E_N) \, : \, \| u \|_{\mathcal{N}_N(T)} = \sup_{t' \in \R} \| (\tau- \omega_{\alpha}(\xi,\eta) + i N^{(2-\frac{\alpha}{2})})^{-1} \mathcal{F}_{x,y,t} [u \cdot \rho_1(N^{(2-\frac{\alpha}{2})}(t-t'))] \|_{X_N} < \infty\}.
	\end{equation*}
	We localize the spaces in time by the usual means: For $T \in (0,1]$, let
	\begin{equation*}
		\begin{split}
			F_N(T) &= \{ u \in C([-T,T];E_N) \, : \, \| u \|_{F_N(T)} = \inf_{\substack{\tilde{u} = u \text{ in } \\
					\D_\lambda \times [-T,T]}} \| \tilde{u} \|_{F_N} < \infty \}, \\
			\mathcal{N}_N(T) &= \{ u \in C([-T,T];E_N) \, : \, \| u \|_{\mathcal{N}_N(T)} = \inf_{\substack{\tilde{u} = u \text{ in } \\
					\D_\lambda \times [-T,T]}} \| \tilde{u} \|_{\mathcal{N}_N} < \infty \}.
		\end{split}
	\end{equation*}
	On the rescaled domains, we consider the weight
	\begin{equation*}
		p_\lambda(\xi,\eta) =\lambda^{-\frac{1}{2}} + \frac{|\eta|}{|\xi|},
	\end{equation*}
	and let
	\begin{equation*}
		E^s(\D_\lambda) = \{ f: \D_\lambda \to \C : \, \| \langle \xi \rangle^s p_\lambda(\xi,\eta) \hat{f}(\xi,\eta) \|_{L^2_{\xi,\eta}} < \infty \}.
	\end{equation*}
	It is easy to see that for $u_0 \in E^s$, $s \geq 0$, we have that $\lambda^{-\frac{2\alpha}{\alpha+2}} u_0(\lambda^{-\frac{2}{\alpha+2}} x, \lambda^{-1} y)$ (cf. \eqref{eq:ScaleInvariant}) converges to zero in $E^s(\D_\lambda)$ polynomially as $\lambda \to \infty$. Indeed, we find that $\lambda^{-\frac{1}{2}} \lambda^{-\frac{2\alpha}{\alpha+2}} \| u_0(\lambda^{-\frac{2}{\alpha+2}} x, \lambda^{-1} y) \|_{L^2}= \lambda^{-\beta} \| u_0 \|_{L^2}$ for some $\beta > 0$. The power $\lambda^{-\frac{1}{2}}$ is chosen to match powers $M^{\frac{1}{2}}$ with $M$ denoting the dyadic localization of the $y$-frequencies in the nonlinear estimates.

	\medskip
	
	Let $E^\infty(\D_\lambda) = \bigcap_{s \geq 0} E^{s}(\D_\lambda)$. We assemble the spaces $F^{s}(T)$, $\mathcal{N}^{s}(T)$, and $B^{s}(T)$ via Littlewood-Paley decomposition:
	\begin{equation*}
		\begin{split}
			F^{s}(T) &= \{ u \in C([-T,T];E^{\infty}) \, : \, \| u \|^2_{F^{s}(T)} = \sum_{N \in 2^{\mathbb{N}_0}} N^{2s} \| P_N p_\lambda(-i\partial_x, - i \nabla_y) u \|^2_{F_N(T)} < \infty \}, \\
			\mathcal{N}^{s}(T) &= \{ u \in C([-T,T];E^{\infty}) \, : \, \| u \|^2_{\mathcal{N}^{s}(T)} = \sum_{N \in 2^{\mathbb{N}_0}} N^{2s} \| P_N p_\lambda(-i\partial_x, - i \nabla_y) u \|^2_{\mathcal{N}_N(T)} < \infty \}, \\
			B^{s}(T) &= \{ u \in C([-T,T];E^{\infty}) \, : \, \| u \|^2_{B^{s}(T)} = \sum_{N \in 2^{\N_0}} N^{2s} \sup_{t\in [-T,T]} \| P_N p_\lambda(-i\partial_x, - i \nabla_y) u(t) \|^2_{E_N} < \infty \}. \\
		\end{split}
	\end{equation*}

	Recall the multiplier properties of admissible time-multiplication: For any $N \in 2^{\N_0}$, we define the set $S_N$ of $N$-acceptable time multiplication factors
	\begin{equation*}
		S_N  = \{ m_N: \R \to \R \, : \, \| m_N \|_{S_N} = \sum_{0 \leq j \leq 30} N^{-(2-\frac{\alpha}{2})j} \| \partial^j m_N \|_{L^\infty} < \infty \}.
	\end{equation*}
	We have for any $s \in \R_{\geq 0}$ and $T \in (0,1]$:
	\begin{equation}\label{TimeMult}
		\left.  \begin{split}
			\| \sum_{N \in 2^{\N_0}} (m_N(t) P_N u ) \|_{F^{s}(T)} &\lesssim \big( \sup\limits_{N \in 2^{\N_0}} \| m_N \|_{S_N} \big) \| u \|_{F^{s}(T)}, \\
			\| \sum_{N \in 2^{\N_0}} (m_N(t) P_N u ) \|_{\mathcal{N}^{s}(T)} &\lesssim \big( \sup\limits_{N \in 2^{\N_0}} \| m_N \|_{S_N} \big) \| u \|_{\mathcal{N}^{s}(T)}, \\
			\| \sum_{N \in 2^{\N_0}} (m_N(t) P_N u ) \|_{B^{s}(T)} &\lesssim \big( \sup\limits_{N \in 2^{\N_0}} \| m_N \|_{S_N} \big) \| u \|_{E^{s}(T)}.
		\end{split} \right\}
	\end{equation}
	
	We recall the embedding $F^{s}(T) \hookrightarrow C([-T,T];E^s)$ for short-time $X^{s,b}$-spaces.
	
	\begin{lemma}\label{lemma:Embed}
		Let $s \in \R_{\geq 0}$. For all $T \in (0,1]$ and $u\in F^{s}(T)$ we have
		\begin{equation}
			\sup_{t\in [-T,T]}\|u(t)\|_{H^{s} }\lesssim \|u\|_{F^{s}(T)}.
		\end{equation}
	\end{lemma}
	For the KP-I equation on the plane, this was proved by Ionescu--Kenig--Tataru in \cite{IKT}. In the periodic case, we refer to Guo--Oh \cite{GO} for a proof in the Sobolev scale, which extends in a straight-forward way to the spaces defined here.
	
	Similarly, we have the linear inhomogeneous estimate (cf. \cite[Proposition~2.5.2]{rsc}):
	\begin{lemma}\label{lemma:LinShortTime}
		Let $s \in \R_{\geq 0}$. For all $ T \in (0,1]$ and (mild) solutions $ u \in C([-T,T],E^\infty(\D_\lambda))$ of
		\begin{equation*}
			\partial_t u - \partial_{x} D_x^\alpha u - \partial_{x}^{-1} \Delta_y u = \partial_x(u^2) \text{ in } \D_\lambda \times (-T,T),
		\end{equation*}
		we have
		\begin{equation}
			\|u\|_{F^{s}(T)} \lesssim \|u\|_{B^{s}(T)} + \|\partial_x (u^2) \|_{\mathcal{N}^{s}(T)}.
		\end{equation}
		
	\end{lemma}
	
	\subsection{Linear Strichartz estimates}
	Next, we prove an $L^4_{x,y,t}$-Strichartz estimate which we will use for estimates of non-resonant interactions. For this purpose, it suffices to take advantage of the evolution in the $\eta$-variables to be approximately of Schr\"odinger type.
	Let $d_1, d_2 \in \N_0$ with $d_1 + d_2 =2$ and $\epsilon>0$ if $d_2=2$ and $\epsilon=0$ if $d_2<2$.
	We recall an $L^4$-Strichartz estimate for the Schr\"odinger equation.
	\begin{lemma}
		\label{lem:LinearStrichartz}
		For all $\lambda \geq 1$, the following estimate holds:
		\begin{equation*}
			\| e^{it \Delta} f \|_{L^4_{y,t}( \mathbb{R}^{d_1} \times \mathbb{T}_\lambda^{d_2} \times [0,1])} \lesssim  \lambda^{\epsilon } \| f \|_{H^\epsilon(\mathbb{R}^{d_1} \times \mathbb{T}_\lambda^{d_2})}.
		\end{equation*}
	\end{lemma}
	On Euclidean space, i.e., $d_2 = 0$ this is standard (cf. \cite[Chapter~2]{Tao2006}). If $\lambda=1$, this is due to Bourgain \cite{Bourgain1993} on the torus, and in the semi-periodic case, this was proved by Takaoka--Tzvetkov \cite{TakaokaTzvetkov2001}. The general case follows by rescaling.

	Next, we prove an $L^4_{x,y,t}$-Strichartz estimate for the linear propagator $S_\alpha$. We will assume certain lower bounds on the frequencies: If the $x$-variable is $\nu$-periodic, due to the mean zero assumption, the case that $x$-frequencies are much smaller than $\nu^{-1}$ is vacuous. Regarding the $y$-frequencies: We never decompose to a scale below $\lambda^{-1}$.
	
	\begin{lemma}
		\label{lem:L4StrichartzLocalized}
		Consider  $\lambda \in 2^{\N_0}$, $\D_\lambda = \K_\nu \times \R^{d_1} \times \T_\lambda^{d_2}$, and $K,N \in 2^{\Z}$, and additionally if $ \D_\lambda = \T_{\nu}\times \R^{d_1} \times \T_{\lambda}$, we suppose that $K,N \geq \nu^{-1}$. Let $I \subseteq \R$ be an interval with $|I| \sim K$ and $|\xi| \sim N$ for $\xi \in I$. Further, let $M \in 2^{\Z} \cap [\lambda^{-1},\infty)$ and $Q_M \subseteq \R^{2}$ be any cube of side-length $M$. For all $f \in L^2_{x,y}(\D_\lambda)$ with $\mathrm{supp}(\mathcal{F}_{x,y}(f))\subseteq I\times Q_M$ we have
		\begin{equation}
			\label{eq:L4StrichartzLocalized}
			\|  S_\alpha(t) u_0 \|_{L^4_{x,y,t}(\D_\lambda \times [0,1])} \lesssim K^{\frac{1}{4}}  C(N,M) \|  u_0 \|_{L^2_{x,y}}
		\end{equation}
		with $\epsilon $ as in Lemma \ref{lem:LinearStrichartz} and
		\begin{equation*}
			C(N,M) =
			\begin{cases}
				M^{\frac{1}{2}}, &\quad M \lesssim 1,\\
				\lambda^{\epsilon} (N^{\frac{1}{4}} \vee 1) M^\epsilon, &\quad M \gg 1.
			\end{cases}
		\end{equation*}
		
	\end{lemma}
	\begin{proof}
		We can use Bernstein's inequality in the $\xi$-frequencies, Plancherel's theorem, and Minkowski's inequality to write
		\begin{equation*}
			\begin{split}
				\| \int e^{i(x.\xi + y.\eta + t \xi |\xi|^\alpha + t \eta^2/\xi)} \hat{f}(\xi,\eta) d\xi d\eta \big\|_{L^4_{x,y,t}(\D_\lambda \times [0,1])} &\lesssim K^{\frac{1}{4}} \big\| \big( \int d\xi \big| \int d\eta e^{i(y.\eta + t \eta^2/\xi)} \hat{f}(\xi,\eta) \big|^2 \big)^{\frac{1}{2}} \big\|_{L^4_{y,t}(\R^{d_1} \times \T_\lambda^{d_2} \times [0,1])} \\
				&\lesssim K^{\frac{1}{4}} \big( \int d\xi \big\| \int d\eta e^{i (y.\eta + t \eta^2/\xi)} \hat{f}(\xi,\eta) \big\|^2_{L^4_{y,t}(\R^{d_1} \times \T_\lambda^{d_2} \times [0,1])} \big)^{1/2}.
			\end{split}
		\end{equation*}
		For $M \lesssim 1$, we use Galilean invariance and Bernstein's inequality to conclude the bound. If $M \geq 1$, we find for $\xi > 0$ (the case $\xi < 0$ is treated by time-inversion)
		\begin{equation*}
			\big\| \int d\eta e^{i(y.\eta + t |\eta|^2 / \xi)} \hat{f}(\xi,\eta) \big\|_{L^4_{y,t}(\R^{d_1} \times \T_\lambda^{d_2} \times [0,1])} \lesssim N^{\frac{1}{4}} \big\| \int d\eta e^{i(y.\eta + t \eta^2)} \hat{f}(\xi,\eta) \big\|_{L^4_{y,t}(\R^{d_1} \times \T_\lambda^{d_2}\times [0,\xi^{-1}] )}.
		\end{equation*}
		We estimate by Galilean invariance and Strichartz estimates (Lemma \ref{lem:LinearStrichartz}) on $N^{-1} \vee 1$ unit time intervals
		\begin{equation*}
			\big\| \int d\eta e^{i(y.\eta + t |\eta|^2 / \xi)} \hat{f}(\xi,\eta) \big\|_{L^4_{y,t}(\R^{d_1} \times \T_\lambda^{d_2} \times [0,1])}	\lesssim (N^{\frac{1}{4}} \vee 1) M^{\epsilon}   \lambda^{\epsilon} \| \hat{f}(\xi,\cdot) \|_{L^2_\eta},
		\end{equation*}
		and finally we take the $L^2$-norm in $\xi$.
	\end{proof}
	The estimate could be further improved taking into account frequency dependent time localization, however, this is not required for the present analysis.
	
	We point out that only the $L^4_{y,t}$-Strichartz estimates on the two-dimensional torus is not scale-invariant and loses a factor $\lambda^{\epsilon}$.
	
	We note the following consequence for small $x$-frequencies.
	\begin{corollary}
		\label{cor:L4StrichartzSmallFrequencies}
		Under the assumptions of Lemma \ref{lem:L4StrichartzLocalized}, if $I\subseteq [-10^4,10^4]$ is fixed,  we have
		\begin{equation}
			\label{eq:L4StrichartzLocalizedSmallFrequencies}
			\|  S_\alpha(t) u_0 \|_{L^4_{x,y,t}(\D_\lambda \times [0,1])} \lesssim \lambda^\epsilon M^\epsilon \|  u_0 \|_{L^2_{x,y}}.
		\end{equation}
	\end{corollary}
	\begin{proof}
		After a Littlewood-Paley decomposition in the $x$-frequencies, this is a consequence of Lemma \ref{lem:L4StrichartzLocalized}.
	\end{proof}
	
	
	We record a consequence for the interaction of two functions with specified frequency and modulation support. We will use this estimate for non-resonant interactions (see Section \ref{section:Resonance}).
	\begin{lemma}
		\label{lem:L4Summary}
		Let $N_1, N_2 \in 2^{\Z}$, $N_2 \leq N_1$, $L_i \in 2^{\N_0}$, $M_i \in 2^{\Z} \cap [\lambda^{-1}, \infty)$ for $i=1,2,3$, and $N_i \geq \nu^{-1}$ if $\D_\lambda = \T_\nu \times \R^{d_1} \times \T^{d_2}_\lambda$. Let $f^{(i)}: \D_\lambda^* \times \R \to \R_+$ be supported in $\tilde{D}_{N_i,M_i,L_i}$ for $i=1,2$. Then the following estimate holds:
		\begin{equation*}
			\| \mathbf{1}_{\tilde{D}_{N_3,M_3,L_3}} (f^{(1)} * f^{(2)} ) \|_{L^2_{\xi,\eta,\tau}} \lesssim \lambda^\epsilon C(N_1,N_2) N_{\min}^{\frac{1}{2}} M_{\min}^{\epsilon} \prod_{i=1}^2 L_i^{\frac{1}{2}} \| f^{(i)} \|_{L^2}
		\end{equation*}
		with
		\begin{equation*}
			C(N_1,N_2) =
			\begin{cases}
				(N_1 N_2)^{\frac{1}{4}}, \; &N_2 \geq 1, \\
				N_1^{\frac{1}{4}}, \; &N_2 \leq 1 \leq N_1, \\
				1, \; &N_1 \leq 1.
			\end{cases}
		\end{equation*}
	\end{lemma}
	\begin{proof}
		By almost orthogonality, we can suppose that $\text{supp} (f^{(i)}) \subseteq I_i \times Q_i \times \R$ with $I_i \subseteq \R$ an interval of length $N_{\min}$ and $Q_i$ a cube of length $M_{\min}$. Then, we can use two $L^4_{x,y,t}$-Strichartz estimates \eqref{eq:L4StrichartzLocalized} to find in case of $N_2 \geq 1$:
		\begin{equation*}
			\begin{split}
				\| \mathbf{1}_{\tilde{D}_{N_3,M_3,J_3}} (f^{(1)} * f^{(2)} ) \|_{L^2_{\xi,\eta,\tau}} &\lesssim
				\| \mathcal{F}^{-1} f^{(1)} \|_{L^4_{x,y,t}} \| \mathcal{F}^{-1} f^{(2)} \|_{L^4_{x,y,t}} \\
				&\lesssim \lambda^\epsilon N_{\min}^{\frac{1}{2}} (N_1 N_2)^{\frac{1}{4}} M_{\min}^{ \epsilon} \prod_{i=1}^2 L_i^{\frac{1}{2}} \| f^{(i)} \|_{L^2}.
			\end{split}
		\end{equation*}
		If $N_2 \leq 1 \leq N_1$, we find by \eqref{eq:L4StrichartzLocalized}
		\begin{equation*}
			\begin{split}
				\| \mathbf{1}_{\tilde{D}_{N_3,M_3,L_3}} (f^{(1)} * f^{(2)}) \|_{L^2_{\xi,\eta,\tau}} &\lesssim
				\| \mathcal{F}^{-1} f^{(1)} \|_{L^4_{x,y,t}} \| \mathcal{F}^{-1} f^{(2)} \|_{L^4_{x,y,t}} \\
				&\lesssim \lambda^\epsilon N_{\min}^{\frac{1}{2}} N_1^{\frac{1}{4}} M_{\min}^{ \epsilon} \prod_{i=1}^2 L_i^{\frac{1}{2}} \| f^{(i)} \|_{L^2}.
			\end{split}
		\end{equation*}
		Lastly, if $N_1 \leq 1$, \eqref{eq:L4StrichartzLocalized} yields
		\begin{equation*}
			\begin{split}
				\| \mathbf{1}_{\tilde{D}_{N_3,M_3,L_3}} (f^{(1)} * f^{(2)} ) \|_{L^2_{\xi,\eta,\tau}} &\lesssim
				\| \mathcal{F}^{-1} f^{(1)} \|_{L^4_{x,y,t}} \| \mathcal{F}^{-1} f^{(2)}  \|_{L^4_{x,y,t}} \\
				&\lesssim \lambda^\epsilon N_{\min}^{\frac{1}{2}} M_{\min}^{ \epsilon} \prod_{i=1}^2 L_i^{\frac{1}{2}} \| f^{(i)} \|_{L^2}.
			\end{split}
		\end{equation*}
		This finishes the proof.
	\end{proof}
	For comparison, we note the following trivial estimate which we employ in certain non-resonant cases:
	\begin{lemma}
		\label{lem:CauchySchwarz}
		Let $N_i \in 2^{\Z}$, $M_i \in 2^{\Z} \cap [\lambda^{-1},\infty)$, $L_i \in 2^{\N_0}$ and $f^{(i)}: \D_\lambda \times \R \rightarrow \R_{+}$ be $L^2$ functions supported in $\tilde{D}_{N_i,M_i,L_i}$, for $i=1,2,3$, respectively. If $\D_\lambda = \T_\nu \times \R^{d_1} \times \T_\lambda^{d_2}$, we suppose that $N_i \geq \nu^{-1}$ in addition. Then
		\begin{equation}\label{eq:NonReso}
			\int_{\D_\lambda^* \times \R}(f^{(1)} \ast f^{(2)})\cdot f^{(3)} \lesssim  (N_{\min} M^2_{\min} L_{\min})^{\frac{1}{2}} \cdot \prod_{i=1}^3 \| f^{(i)} \|_{L^2}.
		\end{equation}
		\begin{proof}
			The estimate follows from applying the Cauchy-Schwarz inequality.
		\end{proof}
	\end{lemma}
	
	\section{Resonance and bilinear estimates}
	\label{section:Resonance}
	We analyze the resonance function to obtain good bilinear estimates in the resonant case. For  $\xi=\xi_1+\xi_2 \in \K^*$ (non-vanishing) and $\eta=\eta_1+\eta_2 \in \R^{2}$, we have
	\begin{equation}\label{resfun}
		\begin{split}
			\Omega_\alpha(\xi_1,\eta_1,\xi_2,\eta_2)&:= \omega_\alpha(\xi_1,\eta_1)+\omega_\alpha(\xi_2,\eta_2)-\omega_\alpha(\xi_1+\xi_2,\eta_1+\eta_2)\\
			&=- \Omega_{KdV}^{(\alpha)} + \frac{\xi_1 \xi_2}{\xi_1 + \xi_2} \Big| \frac{\eta_1}{\xi_1} - \frac{\eta_2}{\xi_2} \Big|^2.
		\end{split}
	\end{equation}
	Above $\Omega^{(\alpha)}_{KdV}(\xi_1,\xi_2)$ denotes the resonance for the dispersion generalized KdV equation:
	\begin{equation*}
		\Omega^{(\alpha)}_{KdV}(\xi_1,\xi_2) = |\xi_1+\xi_2|^\alpha (\xi_1+\xi_2) - |\xi_1|^\alpha \xi_1 - |\xi_2|^\alpha \xi_2.
	\end{equation*}
	We define the resonant case for the higher-dimensional KP-I equations by
	\begin{equation}
		\label{eq:rescon}
		|\Omega_\alpha(\xi_1,\eta_1,\xi_2,\eta_2)| \ll |\Omega^{(\alpha)}_{KdV}(\xi_1,\xi_2)|.
	\end{equation}
	Suppose that $|\xi_1+\xi_2| \sim \max(|\xi_1|,|\xi_2|)$, and let $N_{\max},N_{\min} \in 2^{\Z}$ such that
	\begin{equation*}
		N_{\max} \sim \max (|\xi_1|,|\xi_2|,|\xi_1+\xi_2|), \quad N_{\min} \sim \min(|\xi_1|,|\xi_2|,|\xi_1+\xi_2|).
	\end{equation*}
	The resonance condition requires
	\begin{equation*}
		N_{\max}^\alpha \sim \Big| \frac{\eta_1}{\xi_1} - \frac{\eta_2}{\xi_2} \Big|^2.
	\end{equation*}
	
	The gradient of the dispersion relation is given by
	\begin{equation}\label{eq:grad}
		\nabla \omega_\alpha(\xi,\eta) = \Big((\alpha + 1) |\xi|^\alpha -\frac{|\eta|^2}{\xi^2}, \frac{2\eta}{\xi}\Big).
	\end{equation}
	Consequently, in case of \eqref{eq:rescon}, we find
	\begin{equation*}
		|\nabla\omega_\alpha(\xi_1,\eta_1) - \nabla \omega_\alpha(\xi_2,\eta_2)| \gtrsim \Big| \frac{\eta_1}{\xi_1} -\frac{\eta_2}{\xi_2}\Big| \sim N_{\max}^\frac{\alpha}{2}.
	\end{equation*}
	This means that in case of a resonant interaction with $|\xi_1+\xi_2| \sim \max(|\xi_1|,|\xi_2|)$ the waves are transverse.
	
	\begin{lemma}\label{lemma:elem}
		Let $I, J$ be intervals and $f: J \rightarrow \R$ be a smooth function. Then
		\begin{equation}
			|\{x: f(x) \in I\}| \leqslant \frac{|I|}{\inf_{y}|f'(y)|}.
		\end{equation}
		\begin{proof}
			This is a consequence of the mean value theorem. Let $x_1,x_2 \in J$ be such that $f(x_1),f(x_2) \in I$. Then
			\begin{equation*}
				|x_1-x_2| =\frac{|f(x_1)-f(x_2)|}{|f'(\xi)|} \leqslant \frac{|I|}{\inf_y |f'(y)|}.
			\end{equation*}
		\end{proof}
	\end{lemma}

	\begin{proposition}{(Transverse $L^2$ bilinear estimate)}\label{prop:GeneralBilinear}
		Let $d_1+d_2 =2$, $N \in 2^{\N_0}$, and $u_1,u_2 \in L^2(\R^{d_1} \times \T_\lambda^{d_2} \times \R)$ have their Fourier supports in $\tilde{D}_{N_i,M_i,L_i}$ for $i=1,2$, respectively, with $N_i \in 2^{\Z} $ ,and additionally $N_i \geq \nu^{-1}$ for $d_1 = 0$, and $M_i \in 2^{\Z} \cap [\lambda^{-1},\infty)$, and let $N \sim \max(N_1,N_2)$. Suppose that for frequencies in the Fourier support, the resonance condition \eqref{eq:rescon} holds. Then, we have
		\begin{equation}\label{eq:dyadicbilinear}
			\Vert P_N( u_1 u_2) \Vert_{L^2_{x,y,t}} \lesssim M_{\min}^{\frac{1}{2}} N_{\min}^{\frac{1}{2}} L_{\min}^{\frac{1}{2}} \Big (d_2+\frac{L_{\max}}{N^{\frac{\alpha}{2}}}\Big )^{\frac{1}{2}} \Vert u_1 \Vert_{L^2_{x,y,t}} \| u_2 \|_{L^2_{x,y,t}}.
		\end{equation}
		\begin{proof}
			From Plancherel and Cauchy-Schwarz, we have
			\begin{equation}\label{con}
				\begin{split}
					\| u_1u_2\|_{L^2} &= \Big| \int_{\D_\lambda^* \times \R}\hat{u}_1(\xi_1,\eta_1,\tau_1) \hat{u}_2(\xi-\xi_1,\eta-\eta_1,\tau-\tau_1,) d\xi_1d\eta_1 d\tau_1\Big|   \\
					&\lesssim L_{\min}^{\frac{1}{2}} \sup_{\xi,\eta,\tau: \, |\xi|\sim N, }|E(\tau,\xi,\eta)|^{\frac{1}{2}} \| u_1 \Vert_{L^2} \| u_2 \|_{L^2},
				\end{split}
			\end{equation}
			where $|\cdot|$ denotes the measure on $\D^*_{\lambda}$ and the set $E(\tau,\xi,\eta) \subset \D^*_{\lambda}$ is given by
			\begin{equation*}
				E(\tau,\xi,\eta):= \{(\xi_1,\eta_1)\in  A_{N_1,M_1}: \eqref{eq:rescon} \text{ holds}, \; |\tau-\omega_{\alpha}(\xi_1,\eta_1)-\omega_{\alpha}(\xi-\xi_1,\eta-\eta_1)|\leqslant L_{\max}, (\xi-\xi_1,\eta-\eta_1) \in A_{N_2,M_2}\}.
			\end{equation*}
			It remains to estimate $|E(\tau,\xi,\eta)|$. Define
			\begin{equation*}
				\phi(\xi,\xi_1,\eta,\eta_1,\tau):= \tau-\omega_{\alpha}(\xi_1,\eta_1)-\omega_{\alpha}(\xi-\xi_1,\eta-\eta_1) = \tau-|\xi_1|^\alpha \xi_1 -\frac{|\eta_1|^2}{\xi_1} -|\xi-\xi_1|^\alpha (\xi - \xi_1) -\frac{|\eta-\eta_1|^2}{\xi-\xi_1}  .
			\end{equation*}
			From \eqref{eq:rescon}, we have
			\begin{equation}\label{eq:Gain}
				\big|\frac{\partial \phi}{\partial{\eta_{11}}}\big| \gtrsim N^{\frac{\alpha}{2}} \text{ or }\big|\frac{\partial \phi}{\partial{\eta_{12}}}\big| \gtrsim N^{\frac{\alpha}{2}},
			\end{equation}
			where $\eta_{1i}$ denotes the $i$th component of $\eta_1$.

			Similarly, in the periodic or mixed setting, we have
			\begin{equation*}
				|E(\tau,\xi,\eta)| \lesssim N_{\min}\sup_{|\xi_1|\sim N_1, |\xi|\sim N}|E_1(\xi,\xi_1)|,
			\end{equation*}
			where $E_1(\xi,\xi_1)\subseteq \R^{d_1} \times \Z_\lambda^{d_2}$ is given as
			\begin{equation*}
				E_1(\xi,\xi_1)= I_1(\xi,\xi_1)\cup I_2(\xi,\xi_1),
			\end{equation*}
			where
			\begin{equation*}
				\begin{split}
					I_i(\xi,\xi_1):=  \Big\{|\eta_1|\sim M_1 :\big|\frac{\partial \phi}{\partial{\eta_{1i}}}\big| \gtrsim N^\frac{\alpha}{2}  \Big\}.
				\end{split}
			\end{equation*}
			Using Lemma \ref{lemma:elem}, we have the following bound
			\begin{equation*}
				|E_1(\xi,\xi_1)|\lesssim N_{\min} \Big (1+\frac{L_{\max}}{N^{\frac{\alpha}{2}}}\Big) M_{\min},
			\end{equation*}
			where the summand $1$ could be avoided in case $d_2=0$.
			Substituting this in \eqref{con}, we get the desired estimate.
		\end{proof}
	\end{proposition}

	\section{Short-time bilinear estimates}
	\label{section:ShorttimeBilinear}
	The purpose of this section is to prove short-time bilinear estimates, which we need to propagate the nonlinearity. For the remainder of the section, let $d = 3$, $\alpha \in [2,4)$. Recall that the frequency dependent time localization is given by
	\begin{equation*}
		T(N) = N^{-(2-\frac{\alpha}{2})}.
	\end{equation*}
	We consider the domains (recall $\nu = \lambda^{\frac{2}{\alpha+2}}$)
	\begin{equation*}
		\D_\lambda =
		\begin{cases}
			\K_\nu^{(1)} \times \K_\lambda^{(2)} \times \K_\lambda^{(3)}, \quad &\alpha = 2, \\
			\K_\nu \times \R^2, \quad &\alpha \in (2,4),
		\end{cases}
	\end{equation*}
	where $\K_\lambda^{(i)}$, $\K_\lambda \in \{ \R; \T_\lambda \}$.
	In the following we write for the sake of brevity
	\begin{equation*}
		A \lesssim_\lambda B :\Leftrightarrow
		A \lesssim \lambda^{0+} B
	\end{equation*}
	with implicit constants independent of $\lambda\geq 1$. In the following, we put the factor $\lambda^{0+}$ regardless of the domain since this does not matter for the following analysis and it simplifies the exposition.
	
	\begin{proposition}\label{prop:ShortTimeBilinear}
		Let $s \geqslant r >\max \{\frac{5}{2} - \frac{\alpha}{2},1\}$. For all $T\in(0,1]$, we find the following estimates to hold:
		\begin{align}
			\label{eq:L2ShorttimeBilinearEstimate}
			\| \partial_x(uv)\|_{\mathcal{N}^{0}(T)} &\lesssim_\lambda \|u\|_{F^{0}(T)} \|v\|_{F^{s}(T)}, \\
			\label{eq:ShorttimeBilinearEstimateRegularity}
			\|\partial_x(uv)\|_{\mathcal{N}^{s}(T)} &\lesssim_\lambda \|u\|_{F^{s}(T)} \|v\|_{F^r(T)} + \| u \|_{F^r(T)} \| v \|_{F^s(T)} .
		\end{align}
	\end{proposition}
	
	\subsection{Dyadic estimates} We prove the dyadic estimates which can then be summed up to obtain Proposition \ref{prop:ShortTimeBilinear}.
	We decompose in the $\eta$ variable as follows
	\begin{equation*}
		P_{N}(\partial_{x}(u_{N_1}v_{N_2})) = \sum_{\star} P_{N,M}(\partial_{x}(u_{N_1,M_1} v_{N_2,M_2})),
	\end{equation*}
	where $\star$ denotes a non-trivial relation between the size of the $y$ frequencies $\{ M, M_1, M_2 \} \subseteq 2^{\Z} \cap [\lambda^{-1},\infty)$ and $\{ N, N_1, N_2 \} \subseteq 2^{\Z}$ and additionally $N_i,N  \geq \nu^{-1}$ if $d_1 = 0$, i.e., if the $x$-variable is periodic. Note that for the norm in the LHS of the above equation to be non-zero, we require that the size of at least two $y$ frequencies be comparable. Also, by another almost orthogonal decomposition, we can suppose that the $\eta$-support of $u_{N_1,M_1}$, $v_{N_2,M_2}$ is localized to cubes of length $M_{\min}$. This becomes useful in case $M \leq 2^{-10} M_1$. The key dyadic estimate, which yields \eqref{eq:L2ShorttimeBilinearEstimate} and \eqref{eq:ShorttimeBilinearEstimateRegularity}, reads
	\begin{equation}\label{DyadicShorttimeEstimate}
		\| P_{N,M}(\partial_x(u_{N_1,M_1}v_{N_2,M_2}))\|_{\mathcal{N}_N} \lesssim_\lambda C(N,N_1,N_2) M_{\min}^{\frac{1}{2},\frac{1}{2}+} \|u_{N_1,M_1}\|_{F_{N_1}} \|v_{N_2,M_2}\|_{F_{N_2}}.
	\end{equation}
	Above we denote
	\begin{equation*}
		C(N,N_1,N_2) =
		\begin{cases}
			N_{\min}^{\frac{1}{2}}, \quad &\exists i\in \{1,2\}: N_i \ll N, \\
			N_1^{\frac{1}{2}+}, \quad &\text{else},
		\end{cases}
		\quad M^{s_1,s_2} =
		\begin{cases}
			M^{s_1}, \quad &M \leq 1, \\
			M^{s_2}, \quad &M > 1.
		\end{cases}
	\end{equation*}
	\eqref{eq:L2ShorttimeBilinearEstimate} then follows from trading powers of $|\eta|$ to $|\xi|$ using the weight. For $M_{\min} = \lambda^{-1}$, this is clear. For $|\eta| \gtrsim \lambda^{-1}$, we note
	\begin{equation*}
		|\eta|^{\frac{1}{2}} \lesssim
		\begin{cases}
			(1+|\xi|)^{\frac{1}{2}}, &\text{ if } |\eta| \leq |\xi|, \\
			\frac{|\eta|}{|\xi|} (1+|\xi|)^{\frac{1}{2}}, &\text{ if } |\eta| > |\xi|.
		\end{cases}
	\end{equation*}
	Furthermore, we can decompose the weight
	\begin{equation*}
		p_\lambda(\xi,\eta) = \lambda^{-\frac{1}{2}} + \frac{|\eta|}{|\xi|}.
	\end{equation*}
	The constant term can be perceived as part of the weight of a function on the RHS, for $\frac{|\eta|}{|\xi|}$ we can use dyadic localization in $\xi$ and $\eta$ to conclude.
	For the remainder of the section, we suppose that $M_i \geq \lambda^{-1}$.
	
	\medskip
	
	We consider the $High \times Low \rightarrow High$ case first.
	\begin{lemma}\label{lemma:HLDyadic}
		Let $N \gg 1$, $N_1,N_2 \in 2^{\N_0}$ such that $N_2 \ll N \sim N_1$. Suppose that $u_{N_1,M_1}\in F_{N_1}, \; v_{N_2,M_2} \in F_{N_2}$. Then,
		\begin{equation}\label{eq:HLDyadic}
			\| P_{N,M} ( \partial_x (u_{N_1,M_1} v_{N_2,M_2}))\|_{\mathcal{N}_N} \lesssim_\lambda N_2^{\frac{1}{2}} M_{\min}^{\frac{1}{2},\frac{1}{2}+} \| u_{N_1,M_1}\|_{F_{N_1}} \| v_{N_2, M_2}\|_{F_{N_2}}.
		\end{equation}
	\end{lemma}
	
	\begin{proof}
		
		Using the definition of the $\mathcal{N}_N$ norm, we can bound the left-hand side of \eqref{eq:HLDyadic} by
		\begin{align*}
			&\sup_{t_N \in \R}\| (\tau-\omega(\xi,\eta)+i N^{(2-\frac{\alpha}{2})})^{-1} \cdot \xi \mathbf{1}_{A_{N,M}}(\xi,\eta) \cdot \mathcal{F}[u_{N_1,M_1} \cdot\rho_1(N^{(2-\frac{\alpha}{2})} (t-t_N))] \\
			&\quad \ast \mathcal{F}[v_{N_2,M_2}\cdot \rho_1(N^{(2-\frac{\alpha}{2})}(t-t_N))]\|_{X_N}.
		\end{align*}
		Let
		\begin{equation*}
			f^{(1)} := \mathcal{F}[u_{N_1,M_1}\cdot \rho_1(N^{(2-\frac{\alpha}{2})}(t-t_N))],  \text{ and } f^{(2)}:= \mathcal{F}[v_{N_2,M_2}\cdot \rho_1(N^{(2-\frac{\alpha}{2})}(t-t_N))].
		\end{equation*}
		Using properties \eqref{prop} and \eqref{TimeMult} of the function spaces, it suffices to prove that if $L_1,L_2\geqslant N^{(2-\frac{\alpha}{2})}$ and $f^{(i)}: \D_\lambda^* \times \R \rightarrow \R_{+}$ is supported in $D_{N_i,M_i,L_i}$ for $i=1,2$, then
		\begin{equation}
			\label{eq:HighLowEstimate}
			N \sum_{L \geqslant N^{(2-\frac{\alpha}{2})}} L^{-\frac{1}{2}} \| \mathbf{1}_{D_{N,M,L}} (f^{(1)} \ast f^{(2)} )\|_{L^2} \lesssim_\lambda N_2^{\frac{1}{2}} M^{\frac{1}{2},\frac{1}{2}+}_{\min}  \prod_{i=1}^{2} L_i^{\frac{1}{2}} \| f^{(i)} \|_{L^2}.
		\end{equation}
		This will be proved by a case-by-case analysis. For $L_{\max}=\max(L,L_1,L_2)$, we consider two cases:
		
		\textbf{(i)} \underline{$L_{\max}\leq N_1^\alpha N_2$}:
		In this case, we can use the bilinear Strichartz estimate from Proposition \ref{prop:GeneralBilinear} to find:
		\begin{align}
			\label{eq:HighLowABilinear}
			\| \mathbf{1}_{D_{N,M,L}} (f^{(1)} \ast f^{(2)} )\|_{L^2}
			\lesssim N_2^{\frac{1}{2}} N_1^{-\frac{\alpha}{4}} M_{\min}^{\frac{1}{2}} \prod_{i=1}^2 L_i^{\frac{1}{2}} \| f^{(i)} \|_{L^2}.
		\end{align}
		By summing up \eqref{eq:HighLowABilinear} we obtain
		\begin{equation*}
			\sum_{L \geq N^{(2-\frac{\alpha}{2})}} L^{-\frac{1}{2}} \| \mathbf{1}_{D_{N,M,L}} (f^{(1)} \ast f^{(2)}) \|_{L^2_{x,y,t}} \lesssim N_1^{-1} N_2^{\frac{1}{2}} M_{\min}^{\frac{1}{2}} \prod_{i=1}^2 L_i^{\frac{1}{2}} \| f^{(i)} \|_{L^2}.
		\end{equation*}
		
		\medskip
		
		\textbf{(ii)} \underline{$L_{\max} > N_1^\alpha N_2$}:
		This seemingly easier case requires to distinguish into more subcases. We shall analyze the size of the resonance in case $N_2=1$ more carefully. To this end, we make an additional dyadic decomposition with $N_2 \in 2^{\Z}$, which means a dyadic decomposition of frequencies $\lesssim 1$. Depending on the size of $N_1$ and $N_2$, we consider the following subcases:
		
		\textbf{(a)} $N_2 \geq 1$: In this case the resonance is very favorable, and we shall take advantage of $L_{\max} > N_1^\alpha N_2$. \\
		$\bullet ~L \geq N_1^\alpha N_2$. If $M_{\min} \lesssim 1$, then by Lemma \ref{lem:CauchySchwarz} we find
		\begin{equation*}
			N \sum_{L \geq N_1^\alpha N_2} L^{-\frac{1}{2}} \| \mathbf{1}_{D_{N,M,L}} ( f^{(1)} \ast f^{(2)}) \|_{L^2_{\xi,\eta,\tau}} \lesssim N (N_1^{\alpha}N_2)^{-\frac{1}{2}} N_2^{\frac{1}{2}} M_{\min} \prod_{i=1}^2 L_i^{\frac{1}{2}} \| f^{(i)} \|_{L^2}.
		\end{equation*}
		Hence, we can suppose that $M_{\min} \gtrsim 1$. In this case, we use two $L^4_{x,y,t}$-Strichartz estimates by Lemma \ref{lem:L4Summary}:
		\begin{equation*}
			\begin{split}
				N \sum_{L \geq N_1^\alpha N_2} L^{-\frac{1}{2}} \| \mathbf{1}_{D_{N,M,L}} ( f^{(1)} \ast f^{(2)}) \|_{L^2_{\xi,\eta,\tau}} &\lesssim_\lambda N^{1-\frac{\alpha}{2}} N_2^{-\frac{1}{2}} \| \mathcal{F}^{-1} (f^{(1)}) \|_{L^4_{x,y,t}} \| \mathcal{F}^{-1} (f^{(2)}) \|_{L^4_{x,y,t}} \\
				&\lesssim_\lambda N^{1-\frac{\alpha}{2}} (N_1 N_2)^{\frac{1}{4}} M_{\min}^{ 0+} \prod_{i=1}^2 L_i^{\frac{1}{2}} \| f^{(i)} \|_{L^2}.
			\end{split}
		\end{equation*}
		
		This is acceptable for $N_1 \leq M_{\min}$. If $N_1 \geq M_{\min}$, then an application of Lemma \ref{lem:CauchySchwarz} yields
		\begin{equation*}
			\begin{split}
				N \sum_{L \geq N_1^\alpha N_2} L^{-\frac{1}{2}} \| \mathbf{1}_{D_{N,M,L}} ( f^{(1)} \ast f^{(2)}) \|_{L^2_{\xi,\eta,\tau}} &\lesssim N_1^{1-\frac{\alpha}{2}} M_{\min} \prod_{i=1}^2 L_i^{\frac{1}{4}} \| f^{(i)} \|_{L^2} \\
				&\lesssim N_1^{-\frac{\alpha}{4}} M_{\min} \prod_{i=1}^2 L_i^{\frac{1}{2}} \| f^{(i)}  \|_{L^2} \\
				&\lesssim M_{\min}^{\frac{1}{2}} \prod_{i=1}^2 L_i^{\frac{1}{2}} \| f^{(i)} \|_{L^2}.
			\end{split}
		\end{equation*}
		$\bullet ~N_1^{2-\frac{\alpha}{2}} \leq L \leq N_1^\alpha N_2$: We can use the Cauchy-Schwarz inequality through Lemma \ref{lem:CauchySchwarz} to still find
		\begin{equation*}
			\begin{split}
				N \sum_{N_1^{(2-\frac{\alpha}{2})} \leq L \leq N_1^\alpha N_2} L^{-\frac{1}{2}} \| \mathbf{1}_{D_{N,M,L}} ( f^{(1)} * f^{(2)}) \|_{L^2_{\xi,\eta,\tau}} &\lesssim N N_1^{-1} N_1^{\frac{\alpha}{4}} M_{\min} N_2^{\frac{1}{2}} (N_1^\alpha N_2)^{-\frac{1}{2}} \prod_{i=1}^2 L_i^{\frac{1}{2}} \| f^{(i)} \|_{L^2} \\
				&\lesssim N_1^{-\frac{\alpha}{4}} M_{\min} \prod_{i=1}^2 L_i^{\frac{1}{2}} \| f^{(i)} \|_{L^2}.
			\end{split}
		\end{equation*}
		This is acceptable for $M_{\min} \leq N_1$. For $M_{\min} \geq N_1$, we can in the same range of $L$ consider two $L^4_{x,y,t}$-Strichartz estimates by Lemma \ref{lem:L4Summary} after duality and take out the function with highest modulation in $L_{\xi,\eta,\tau}^2$. In the worst case, this function is at small frequencies $N_2$, i.e., $L_{\max}=L_2$. In this case, we find
		\begin{equation*}
			\begin{split}
				N \sum_{N_1^{(2-\frac{\alpha}{2})} \leq L \leq N_1^\alpha N_2} L^{-\frac{1}{2}} \| \mathbf{1}_{D_{N,M,L}} ( f^{(1)}  * f^{(2)} ) \|_{L^2_{\xi,\eta,\tau}}  &\lesssim_\lambda N_1 \log(N_1^\alpha N_2) (N_1^\alpha N_2)^{- \frac{1}{2}} M_{\min}^{ 0+} (N_1 N_2)^{\frac{1}{2}} \prod_{i=1}^2 L_i^{\frac{1}{2}} \| f^{(i)} \|_{L^2} \\
				&\lesssim_\lambda N_1^{\big( 1 - \frac{\alpha}{2} \big)} M_{\min}^{ \frac{1}{2}+} \prod_{i=1}^2 L_i^{\frac{1}{2}} \| f^{(i)} \|_{L^2}.
			\end{split}
		\end{equation*}
		
		\textbf{(b)} $N_2 \leq N_1^{-\alpha}$. which yields $N_2 N_1^\alpha \leq 1$. Note that this case is possibly vacuous if the $x$-variable is periodic. By two $L^4_{x,y,t}$-Strichartz estimates as in Lemma \ref{lem:L4Summary}, we infer
		\begin{equation}
			\label{eq:L4AuxEstimate}
			\begin{split}
				N \sum_{L \geq N_1^{( 2-\frac{\alpha}{2} )}} L^{-\frac{1}{2}} \| \mathbf{1}_{D_{N,M,L}} ( f^{(1)} \ast f^{(2)}) \|_{L^2_{\xi,\eta,\tau}} &\lesssim_\lambda N_1^{\frac{\alpha}{4}} N_2^{\frac{1}{4}} N_1^{\frac{1}{4}} M_{\min}^{ 0+}  \prod_{i=1}^2 L_i^{\frac{1}{2}} \| f^{(i)} \|_{L^2} \\
				&\lesssim_\lambda M_{\min}^{ 0+} N_1^{\frac{1}{4}} \prod_{i=1}^2 L_i^{\frac{1}{2}} \| f^{(i)} \|_{L^2}.
			\end{split}
		\end{equation}
		
		This is acceptable for $N_1 \leq M_{\min}^2$. If $M_{\min}^2 \leq N_1$, then we can use the Cauchy-Schwarz inequality via Lemma \ref{lem:CauchySchwarz} to find
		\begin{equation*}
			N \sum_{L \geq N_1^{2-\frac{\alpha}{2}}} L^{-\frac{1}{2}} \| \mathbf{1}_{D_{N,M,L}} (f^{(1)} \ast f^{(2)}) \|_{L^2_{\xi,\eta,\tau}} \lesssim N_1^{-(2-\frac{\alpha}{2})} N_1 N_2^{\frac{1}{2}} M_{\min} \prod_{i=1}^2 L_i^{\frac{1}{2}} \| f_i \|_{L^2}.
		\end{equation*}
		This suffices by summation over $N_2 \leq N_1^{-\alpha}$.
		
		\textbf{(c)} $N_1^{-\alpha} \leq N_2 \leq 1$. By the estimate \eqref{eq:L4AuxEstimate}, which is still valid, we can suppose that $N_1 \geq M_{\min}^2$.\\
		$\bullet~ L \geq N_1^\alpha N_2$: By Lemma \ref{lem:CauchySchwarz} we find
		\begin{equation*}
			\begin{split}
				N \sum_{L \geq N_1^\alpha N_2} L^{-\frac{1}{2}} \| \mathbf{1}_{D_{N,M,L}} ( f^{(1)} * f^{(2)}) \|_{L^2_{\xi,\eta,\tau}} &\lesssim N^{1-\frac{\alpha}{2}} N_1^{-1+\frac{\alpha}{4}} M_{\min} \prod_{i=1}^2 L_i^{\frac{1}{2}} \| f^{(i)} \|_{L^2} \\
				&\lesssim M_{\min} N_1^{-\frac{\alpha}{4}} \prod_{i=1}^2 L_i^{\frac{1}{2}} \| f^{(i)} \|_{L^2}.
			\end{split}
		\end{equation*}
		$\bullet~ N_1^{(2-\frac{\alpha}{2})} \leq L \leq N_1^\alpha N_2$: Another application of the Cauchy-Schwarz inequality via Lemma \ref{lem:CauchySchwarz} yields
		\begin{equation*}
			\begin{split}
				N \sum_{L \geq N_1^{(2-\frac{\alpha}{2})}} L^{-\frac{1}{2}} \| \mathbf{1}_{D_{N,M,L}} ( f^{(1)} * f^{(2)}) \|_{L^2_{\xi,\eta,\tau}} &\lesssim N \log(N_1^\alpha N_2) M_{\min} N_1^{-\frac{\alpha}{2}} N_1^{-\frac{2-\alpha/2}{2}} \prod_{i=1}^2 L_i^{\frac{1}{2}} \| f^{(i)} \|_{L^2} \\
				&\lesssim \log(N_1^\alpha N_2) N_1^{-\frac{\alpha}{4}} M_{\min} \prod_{i=1}^2 L_i^{\frac{1}{2}} \| f^{(i)} \|_{L^2} \\
				&\lesssim M_{\min}^{\frac{1}{2}} \prod_{i=1}^2 L_i^{\frac{1}{2}} \| f^{(i)} \|_{L^2}.
			\end{split}
		\end{equation*}
		The proof is complete.
	\end{proof}
	
	Next, we consider the $High \times High \rightarrow Low$ interaction.
	
	\begin{lemma}\label{lemma:HHDyadic}
		Let $N_1,N_2,N \in 2^{\N_0}$, $N_1 \gg 1$ such that $N \ll N_1 \sim N_2$, and $ u_{N_1,M_1} \in F_{N_1}, v_{N_2,M_2} \in F_{N_2}$. Then, the following estimate holds:
		\begin{equation}\label{eq:HHDyadic}
			\| P_{N,M} ( \partial_x (u_{N_1,M_1} v_{N_2,M_2}))\|_{\mathcal{N}_{N}} \lesssim_\lambda N_1^{\frac{1}{2}+} M_{\min}^{\frac{1}{2},\frac{1}{2}+} \| u_{N_1,M_1} \|_{F_{N_1}} \| v_{N_2,M_2} \|_{F_{N_2}}.
		\end{equation}
	\end{lemma}
	\begin{proof}
		We have to add time localization to estimate the functions $u_{N_1,M_1}$, $ v_{N_2,M_2}$ in the short-time norms. This amounts to a factor of $(N_1/N)^{(2-\frac{\alpha}{2})}$ for $N \geq 1$. For $N \geq 1$, we have to show
		\begin{equation}\label{eq:HHLowI}
			\begin{split}
				N \big( \frac{N_1}{N} \big)^{(2-\frac{\alpha}{2})} \sum_{L \geq N^{(2-\frac{\alpha}{2})}} L^{- \frac{1}{2}} \| \mathbf{1}_{D_{N,M,L}}(f^{(1)} \ast f^{(2)})\|_{L^2}
				\lesssim_\lambda N_1^{\frac{1}{2}+} C(M_{\min}) \prod_{i=1}^2 L_i^{\frac1 2} \| f^{(i)} \|_{L^2}
			\end{split}
		\end{equation}
		for $f^{(i)}$ supported in $D_{N_i,M_i,L_i}$, $i=1,2$.
		If $N = 1$, we make an additional dyadic decomposition in the $x$ frequencies such that $N \in 2^{\Z}$ now (which changes the Fourier support to   $\tilde{D}_{N,M,L}$), and additionally suppose that $N \geq \nu^{-1}$, if the $x$-variable is periodic.
		In this case, it suffices to prove
		\begin{equation}\label{eq:HHLowII}
			\begin{split}
				N N_1^{(2-\frac{\alpha}{2})} \sum_{L \geqslant 1} L^{-\frac{1}{2}} \| \mathbf{1}_{\tilde{D}_{N,M,L}}(f^{(1)} \ast f^{(2)})\|_{L^2}
				\lesssim_\lambda N_1^{\frac{1}{2}+} C(M_{\min}) \prod_{i=1}^2 L_i^{\frac{1}{2}} \| f^{(i)} \|_{L^2}.
			\end{split}
		\end{equation}
		We turn to the proof of \eqref{eq:HHLowII} first. 
		We do a case-by-case analysis depending on the size of $N$ and $N_1$:
		
		\textbf{(i)} \underline{$N \leq N_1^{-\alpha}$}: In this case resonance considerations are irrelevant, and for $L \geq N_1^{2-\frac{\alpha}{2}}$, we conclude by two $L^4_{x,y,t}$-Strichartz estimates due to Lemma \ref{lem:L4Summary} on $f^{(i)}$:
		\begin{equation}
			\label{eq:HighHighLowI}
			\begin{split}
				N N_1^{(2-\frac{\alpha}{2})} \sum_{L \geq N_1^{2-\frac{\alpha}{2}}} L^{-\frac{1}{2}} \| \mathbf{1}_{\tilde{D}_{N,M,L}} (f^{(1)} * f^{(2)} ) \|_{L^2} &\lesssim_\lambda N^{\frac{3}{2}} N_1^{\frac{1}{2}} N_1^{1-\frac{\alpha}{4}} M_{\min}^{0+} \prod_{i=1}^2 L_i^{\frac{1}{2}} \| f^{(i)} \|_{L^2} \\
				&\lesssim_\lambda N_1^{\frac{3}{2} - \frac{3 \alpha}{2} - \frac{\alpha}{4}} M_{\min}^{0+} \prod_{i=1}^2 L_i^{\frac{1}{2}} \| f^{(i)} \|_{L^2}.
			\end{split}
		\end{equation}
		This is very favorable for $M_{\min} \geq N_1$. For $1 \leq L \leq N_1^{2- \frac{\alpha}{2}}$, we apply duality and two $L^4_{x,y,t}$-Strichartz estimates to find the above estimate to hold
		(actually, a slightly better estimate holds due to an improved $L^4$-Strichartz estimates for a function with small frequencies).
		
		On the other hand, an application of the Cauchy-Schwarz inequality through Lemma \ref{lem:CauchySchwarz} yields
		\begin{equation}
			\label{eq:HighHighLowII}
			N N_1^{(2-\frac{\alpha}{2})} \sum_{L \geq N_1^{2-\frac{\alpha}{2}}} L^{-\frac{1}{2}} \| \mathbf{1}_{\tilde{D}_{N,M,L}} (f^{(1)} * f^{(2)} ) \|_{L^2} \lesssim N^{\frac{3}{2}} N_1^{1-\frac{\alpha}{4}} M_{\min} \prod_{i=1}^2 L_i^{\frac{1}{2}} \| f^{(i)} \|_{L^2}.
		\end{equation}
		For $M_{\min} \leq N_1$, we find
		\begin{equation*}
			\lesssim N^{\frac{3}{2}} N_1^{\frac{3}{2}} N_1^{-\frac{\alpha}{4}} M_{\min}^{\frac{1}{2}} \prod_{i=1}^2 L_i^{\frac{1}{2}} \| f^{(i)} \|_{L^2}.
		\end{equation*}
		For $1 \leq L \leq N_1^{2- \frac{\alpha}{2}}$, the same estimate holds up to an additional logarithm in $N_1$.
		
		\textbf{(ii)} \underline{$N_1^{-\alpha} \leq N \leq 1$}: In this regime we distinguish between resonant and non-resonant case.
		
		\textbf{(a)} $L_{\max} \leq N_1^\alpha N$. We conclude by duality and Proposition \ref{prop:GeneralBilinear} applied to the dual function and an $f^{(i)}$:
		\begin{equation}
			\label{eq:HHResonantI}
			\begin{split}
				&\quad N_1^{(2-\frac{\alpha}{2})} N \sum_{L \leq N_1^\alpha N} L^{-\frac{1}{2}} \| \mathbf{1}_{\tilde{D}_{N,M,L}} (f^{(1)} * f^{(2)}) \|_{L^2} \\
				&\lesssim N_1^{(2-\frac{\alpha}{2})} N \sum_{L \leq N_1^\alpha N} L^{-\frac{1}{2}} L^{\frac{1}{2}} L_1^{\frac{1}{2}} \frac{N^{\frac{1}{2}}}{N_1^{\frac{\alpha}{4}}} M_{\min}^{\frac{1}{2}} \| f^{(1)} \|_{L^2} N_1^{-\frac{2-\alpha/2}{2}} L_2^{\frac{1}{2}} \| f^{(2)} \|_{L^2} \\
				&\lesssim \log (N_1^\alpha N) N^{\frac{3}{2}} M_{\min}^{\frac{1}{2}} N_1^{(1-\frac{\alpha}{2})} \prod_{i=1}^2 L_i^{\frac{1}{2}} \| f^{(i)} \|_{L^2}.
			\end{split}
		\end{equation}
		This is acceptable.\\
		\textbf{(b)} $L_{\max} \geq N_1^\alpha N$.\\
		$\bullet ~L \geq N_1^\alpha N$. We can use two $L^4_{x,y,t}$-Strichartz estimates by Lemma \ref{lem:L4Summary}
		\begin{equation*}
			N_1^{( 2 - \frac{\alpha}{2} )} N \sum_{L \geq N_1^\alpha N} L^{-\frac{1}{2}} \| \mathbf{1}_{\tilde{D}_{N,M,L}} (f^{(1)} * f^{(2)}) \|_{L^2_{\xi,\eta,\tau}}
			\lesssim_\lambda N_1^{ ( 2 - \alpha )} N^{\frac{1}{2}} (N N_1)^{\frac{1}{2}} M_{\min}^{ 0+} \prod_{i=1}^2 L_i^{\frac{1}{2}} \| f^{(i)} \|_{L^2}.
		\end{equation*}
		For $N_1 \leq M_{\min}$, this gives
		\begin{equation}
			\label{eq:HHResonantA}
			\lesssim_\lambda N_1^{2-\alpha} N M_{\min}^{\frac{1}{2}+} \prod_{i=1}^2 L_i^{\frac{1}{2}} \| f^{(i)} \|_{L^2}.
		\end{equation}
		For $N_1 \geq M_{\min}$, we can use Lemma \ref{lem:CauchySchwarz} to find
		\begin{equation}
			\label{eq:HHResonantB}
			\begin{split}
				N_1^{(2-\frac{\alpha}{2})} N \sum_{L \geq N_1^\alpha N} L^{-\frac{1}{2}} \| \mathbf{1}_{\tilde{D}_{N,M,L}}( f^{(1)} * f^{(2)} )\|_{L^2_{\xi,\eta,\tau}} &\lesssim N_1^{2-\frac{\alpha}{2}} N^{\frac{3}{2}} M_{\min} (N_1^\alpha N)^{-\frac{1}{2}} N_1^{-1 + \frac{\alpha}{4}} \prod_{i=1}^2 L_i^{\frac{1}{2}} \| f^{(i)} \|_{L^2} \\
				&\lesssim N_1^{\frac{3}{2} - \frac{3 \alpha}{4}} N M_{\min}^{\frac{1}{2}} \prod_{i=1}^2 L_i^{\frac{1}{2}} \| f^{(i)} \|_{L^2}.
			\end{split}
		\end{equation}
		$\bullet ~1\leq L \leq N_1^\alpha N$. In this case we can argue like above after applying duality (since $\exists i: L_i \geq N_1^\alpha N$). This yields the same estimates up to a logarithm
		in $N_1$ from summing over $ 1 \leq L \leq N_1^\alpha N$.
		
		\textbf{(iii)} $N \geq 1$. In this case we shall prove \eqref{eq:HHLowI} by considering resonant and non-resonant interactions: \\
		\textbf{(a)} $L_{\max} \leq N_1^\alpha N$. We obtain by the bilinear Strichartz estimate from Proposition \ref{prop:GeneralBilinear} and duality
		\begin{equation}
			\label{eq:HHLowResonantC}
			\begin{split}
				&\quad N \big( \frac{N_1}{N} \big)^{(2-\frac{\alpha}{2})} \sum_{N^{(2-\frac{\alpha}{2})} \leq L_{\max} \leq N_1^\alpha N} L^{-\frac{1}{2}} \| \mathbf{1}{_{D_{N,M,L}}} (f^{(1)} * f^{(2)}) \|_{L^2} \\
				&\lesssim N \big( \frac{N_1}{N} \big)^{(2-\frac{\alpha}{2})} \log( N_1^\alpha N) N_1^{-(1-\frac{\alpha}{4})} M_{\min}^{\frac{1}{2}} \frac{N^{\frac{1}{2}}}{N_1^{\frac{\alpha}{4}}} \prod_{i=1}^2 L_i^{\frac{1}{2}} \| f^{(i)} \|_{L^2}\\
				&= N^{\frac{\alpha}{2}-\frac{1}{2}}N_1^{1-\frac{\alpha}{2}}M_{\min}^{\frac{1}{2}} \log( N_1^\alpha N) \prod_{i=1}^2 L_i^{\frac{1}{2}} \| f^{(i)} \|_{L^2}.
			\end{split}
		\end{equation}
		This is acceptable.\\
		\textbf{(b)} $L_{\max} \geq N_1^\alpha N$. If $L \geq N_1^\alpha N$, we can use two $L^4_{x,y,t}$-Strichartz estimates to find by Lemma \ref{lem:L4Summary}
		\begin{equation*}
			\big( \frac{N_1}{N} \big)^{(2-\frac{\alpha}{2})} N \sum_{L \geq N_1^\alpha N} L^{-\frac{1}{2}} \|\mathbf{1}_{D_{N,M,L}} (f^{(1)} * f^{(2)} ) \|_{L^2} \lesssim_\lambda \big( \frac{N_1}{N} \big)^{(2-\frac{\alpha}{2})} \frac{N^{\frac{1}{2}}}{N_1^\frac{\alpha}{2}} (NN_1)^{\frac{1}{2}} M_{\min}^{0+} \prod_{i=1}^2 L_i^{\frac{1}{2}} \| f^{(i)} \|_{L^2}.
		\end{equation*}
		This is acceptable if $N_1 \leq M_{\min}$:
		\begin{equation}
			\label{eq:HHLowResonantD}
			\lesssim_\lambda \big( \frac{N_1}{N} \big)^{2 - \frac{\alpha}{2}} \frac{N}{N_1^{\frac{\alpha}{2}}} M_{\min}^{\frac{1}{2}+} \prod_{i=1}^2 L_i^{\frac{1}{2}} \| f^{(i)} \|_{L^2} = N^{\frac{\alpha}{2}-1}N_1^{2-\alpha}M_{\min}^{\frac{1}{2}+} \prod_{i=1}^2 L_i^{\frac{1}{2}} \| f^{(i)} \|_{L^2}.
		\end{equation}
		
		If $M_{\min} \leq N_1$, we can apply Lemma \ref{lem:CauchySchwarz} to find
		\begin{equation}
			\label{eq:HHLowResonantE}
			\begin{split}
				N\big( \frac{N_1}{N} \big)^{(2-\frac{\alpha}{2})} \sum_{L \geq N_1^\alpha N} L^{-\frac{1}{2}} \|\mathbf{1}_{D_{N,M,L}} (f^{(1)} * f^{(2)} ) \|_{L^2} &\lesssim \big( \frac{N_1}{N} \big)^{(2-\frac{\alpha}{2})} \frac{N^{\frac{1}{2}}}{N_1^{\alpha/2}} M_{\min} N_1^{-(1-\frac{\alpha}{4})} \prod_{i=1}^2 L_i^{\frac{1}{2}} \| f^{(i)} \|_{L^2} \\
				&\lesssim N_1^{1 - \frac{3\alpha}{4}} N^{-\frac{3}{2} + \frac{\alpha}{2}} M_{\min}^{\frac{1}{2}} \prod_{i=1}^2 L_i^{\frac{1}{2}} \| f^{(i)} \|_{L^2}.
			\end{split}
		\end{equation}
		Summation over $N \leq N_1$ yields
		\begin{equation*}
			\lesssim N_1^{\frac{3}{2} - \frac{\alpha}{2} + } M_{\min}^{\frac{1}{2}} \prod_{i=1}^2 L_i^{\frac{1}{2}} \| f^{(i)} \|_{L^2}.
		\end{equation*}
		
		If $L \leq N_1^\alpha N$, then there is $L_i \geq N_1^\alpha N$ for some $i \in \{1,2 \}$. The argument follows the above lines, estimating the factor with high modulation in $L^2$ and the remaining factors via $L^4$-Strichartz or Cauchy-Schwarz. This gives an additional $\log(N_1^\alpha N)$ from summing over $L$, which can easily be absorbed into $N_1^{\frac{1}{2}+}$.
	\end{proof}
	
	We consider the case when the frequencies are of comparable size and much higher than $1$.
	\begin{lemma}
		Let $N_1,N_2,N \in 2^{\N_0}$, $N_1 \geq 2^{5}$, and $N_1 \sim N_2 \sim N_3$. Then the following estimate holds:
		\begin{equation}
			\label{eq:HHH}
			\| P_{N,M} \partial_x (P_{N_1,M_1} u P_{N_2,M_2} v) \|_{F_N} \lesssim_\lambda M_{\min}^{\frac{1}{2}} N^{\frac{1}{2}+} \| P_{N_1,M_1} u \|_{F_{N_1}} \| P_{N_2,M_2} v \|_{F_{N_2}}.
		\end{equation}
	\end{lemma}
	\begin{proof}
		By the above reductions, we have to prove
		\begin{equation}
			\label{eq:HHHDyadic}
			N \sum_{L \geq N^{(2-\frac{\alpha}{2})}} L^{-\frac{1}{2}} \| \mathbf{1}_{D_{N,M,L} }(f^{(1)} * f^{(2)}) \|_{L^2} \lesssim_\lambda M_{\min}^{\frac{1}{2}} N^{\frac{1}{2}+} \prod_{i=1}^2 L_i^{\frac{1}{2}} \| f^{(i)} \|_{L^2}
		\end{equation}
		for $f^{(i)}$ supported in $D_{N_i,M_i,L_i}$. We consider the resonant and non-resonant interactions:\\
		\textbf{(i)} $L_{\max} \leq N_1^{\alpha+1}$: We apply a bilinear estimate due to Proposition \ref{prop:GeneralBilinear} to $f^{(1)}$ and $f^{(2)}$ to obtain
		\begin{equation*}
			\begin{split}
				N \sum_{L \geq N^{(2-\frac{\alpha}{2})}} L^{-\frac{1}{2}} \| \mathbf{1}_{D_{N,M,L} } (f^{(1)} * f^{(2)} \|_{L^2} &\lesssim \big( \frac{N^{\frac{1}{2}}}{N^{\frac{\alpha}{4}}} N N^{-1} N^{\frac{\alpha}{4}} \big) M_{\min}^{\frac{1}{2}} \prod_{i=1}^2 L_i^{\frac{1}{2}} \| f^{(i)} \|_{L^2} \\
				&\lesssim N^{\frac{1}{2}} M_{\min}^{\frac{1}{2}} \prod_{i=1}^2 L_i^{\frac{1}{2}} \| f^{(i)} \|_{L^2}.
			\end{split}
		\end{equation*}
		\textbf{(ii)} $L_{\max} \geq N_1^{\alpha+1}$: For $L \geq N_1^{\alpha +1}$, we use two $L^4_{x,y,t}$-Strichartz estimates on $f^{(i)}$, $i=1,2$ by Lemma \ref{lem:L4Summary} to find
		\begin{equation*}
			\begin{split}
				N \sum_{L \geq N_1^{\alpha+1}} L^{-\frac{1}{2}} \| \mathbf{1}_{D_{N,M,L}} (f^{(1)} * f^{(2)} ) \|_{L^2} &\lesssim_\lambda N M_{\min}^{ 0 +} N^{-\frac{\alpha+1}{2}} N \prod_{i=1}^2 L_i^{\frac{1}{2}} \| f^{(i)} \|_{L^2} \\
				&\lesssim_\lambda N^{\frac{3}{2}-\frac{\alpha}{2}} M_{\min}^{0+} \prod_{i=1}^2 L_i^{\frac{1}{2}} \| f^{(i)} \|_2.
			\end{split}
		\end{equation*}
		This suffices for $M_{\min} \geq 1$. If $M_{\min} \leq 1$, an application of Lemma \ref{lem:CauchySchwarz} gives
		\begin{equation*}
			\begin{split}
				N \sum_{L \geq N_1^{\alpha + 1}} L^{-\frac{1}{2}} \| \mathbf{1}_{D_{N,M,L}} (f^{(1)} \ast f^{(2)} ) \|_{L^2} &\lesssim \frac{N^{\frac{3}{2}}}{N_1^{\frac{\alpha+1}{2}}} M_{\min} N^{\frac{\alpha}{4}-1} \prod_{i=1}^2 L_i^{\frac{1}{2}} \| f^{(i)} \|_{L^2} \\
				&\lesssim N^{-\frac{\alpha}{4}} \prod_{i=1}^2 L_i^{\frac{1}{2}} \| f^{(i)} \|_{L^2}.
			\end{split}
		\end{equation*}
		This is acceptable.
		If $L \leq N_1^{\alpha+1}$, the arguments from above apply after using duality. The dyadic sum over $N^{2-\frac{\alpha}{2}} \leq L \leq N_1^{\alpha+1}$ gives an additional factor of $\log(N_1)$ which is acceptable.
	\end{proof}
	
	We consider now very small frequencies.
	\begin{lemma}
		Let $\{ N, N_1, N_2 \} \subseteq 2^{\N_0} \cap (-\infty,2^{10}] $. Then the following estimate holds:
		\begin{equation}\label{eq:VLowDyadic}
			\| P_{N,M} ( \partial_x (u_{N_1,M_1} v_{N_2,M_2}))\|_{\mathcal{N}_N} \lesssim_\lambda M_{\min}^{\frac{1}{2}} \| u_{N_1,M_1} \|_{F_{N_1}} \| u_{N_2,M_2} \|_{F_{N_2}}.
		\end{equation}
	\end{lemma}
	\begin{proof}
		We use the same notation as in Lemma \ref{lemma:HLDyadic}. It is then sufficient to prove that if $L_1,L_2 \in 2^{\N_0}$ and $f^{(i)}: \D_\lambda^* \times \R \rightarrow \R_+$ are functions supported in $D_{N_i,M_i,L_i}$, $i=1,2$ and $M_1 \geq M_2$, then
		\begin{equation}\label{eq:ToProve}
			\sum_{ L \geq 1} L^{-\frac{1}{2}} \| \mathbf{1}_{D_{N,M,L}} (f^{(1)} \ast f^{(2)} )\|_{L^2} \lesssim_\lambda
			M_{\min}^{\frac{1}{2}} \prod_{i=1}^2 L_i^{\frac{1}{2}} \| f^{(i)} \|_{L^2}.
		\end{equation}
		This follows from the estimate from Cauchy-Schwarz inequality (Lemma \ref{lem:CauchySchwarz}) for $M_{\min} \leq 1$ and two $L^4_{x,y,t}$-Strichartz estimates for $M_{\min} \geq 1$.
	\end{proof}

	\subsection{Proof~of~Proposition~\ref{prop:ShortTimeBilinear}}
	We estimate the interactions as laid out above separately, i.e.
	\begin{itemize}
		\item High x Low $\to$ High,
		\item High x High $\to$ High,
		\item High x High $\to$ Low,
		\item Low x Low $\to$ Low,
	\end{itemize}
	The key ingredients are the dyadic estimate and the decomposition of the weight $p_{\lambda}(\xi, \eta)$.
	
	We begin with High x Low $\to$ High-interaction. Recall that
	\begin{equation*}
		(1+|\xi|)^s p_\lambda(\xi,\eta) = \lambda^{-\frac{1}{2}} (1+|\xi|)^s + (1+|\xi|)^s \frac{|\eta|}{|\xi|}.
	\end{equation*}
	Let $N \sim N_1 \gg N_2$. For $s \geq r > 1$, we find
	\begin{equation*}
		\lambda^{-\frac{1}{2}} N^s \| P_{N,M} \partial_x (u_{N_1,M_1} v_{N_2,M_2}) \|_{\mathcal{N}_N} \lesssim \lambda^{-\frac{1}{2}} N^s N_{\min}^{\frac{1}{2}} M_{\min}^{\frac{1}{2},\frac{1}{2}+} \| u_{N_1,M_1} \|_{F_{N_1}} \| v_{N_2,M_2} \|_{F_{N_2}}.
	\end{equation*}
	The estimates \eqref{eq:L2ShorttimeBilinearEstimate} and \eqref{eq:ShorttimeBilinearEstimateRegularity} follow by summation: Suppose $M \sim M_1 \gtrsim M_2$. Summation
	over $N_2$ and $M_2$ and using $M_2^{\frac{1}{2},\frac{1}{2}+} \lesssim ( \lambda^{-\frac{1}{2}} + \frac{M_2}{N_2}) (1+N_2)^{\frac{1}{2}}$ as follows:
	\begin{equation*}
		\lambda^{-\frac{1}{2}} N^s \| P_{N,M} \partial_x (u_{N_1,M_1} v_{N_2,M_2}) \|_{\mathcal{N}_N} \lesssim \lambda^{-\frac{1}{2}} N^s \| u_{N_1,M_1} \|_{F_{N_1}} \| v \|_{F^r},
	\end{equation*}
	the claim follows from square summation.
	
	Suppose $M \sim M_2 \gg M_1$. In this case summation over $M_1$ gives
	\begin{equation*}
		\lambda^{-\frac{1}{2}} N^s \| P_{N,M} \partial_x (u_{N_1} v_{N_2,M_2}) \|_{\mathcal{N}_N} \lesssim \lambda^{-\frac{1}{2}} N_1^s \| u_{N_1} \|_{F_{N_1}} N_2^{\frac{1}{2}} M_2^{\frac{1}{2},\frac{1}{2}+} \| v_{N_2,M_2} \|_{F_{N_2}}.
	\end{equation*}
	Similar arguments show that summability is provided for $s \geq r > 1$ for High x High $\to$ High and Low x Low $\to$ Low interaction.
	
	\medskip
	
	For High x High $\to$ Low-interactions we need a different argument. The reason is that the weight $\frac{|\eta|}{|\xi|}$ eliminates the derivative. To estimate $|\eta| \leq 2 \max(|\eta_1|, |\eta_2|)$ in terms of the weight $p_\lambda(\xi_i,\eta_i)$ one has to take into account a high
	frequency $|\xi_i| \sim N_i$. This gives an additional derivative loss in the high frequency.
	
	We revisit the dyadic estimates from Lemma \ref{lemma:HHDyadic}: The estimate \eqref{eq:HighHighLowI} becomes with the additional factor $\frac{N_1}{N}$ and $(N_1 \leq M_{\min}, \; N_1 \leq M_{\min})$:
	\begin{equation*}
		\begin{split}
			\frac{N_1}{N} (N N_1^{2-\frac{\alpha}{2}}) \sum_{L \geq N_1^{2-\frac{\alpha}{2}}} L^{-\frac{1}{2}} \| \mathbf{1}_{\tilde{D}_{N,M,L}} (f^{(1)} \ast f^{(2)}) \|_{L^2} &\lesssim_\lambda N^{\frac{1}{2}} N_1^{\frac{5}{2} - \frac{\alpha}{4}} M_{\min}^{0+} \prod_{i=1}^2 L_i^{\frac{1}{2}} \| f^{(i)} \|_{L^2} \\
			&\lesssim_\lambda N_1^{2-\frac{\alpha}{2} - \frac{\alpha}{4}} M_{\min}^{\frac{1}{2}+} \prod_{i=1}^2 L_i^{\frac{1}{2}} \| f^{(i)} \|_{L^2}.
		\end{split}
	\end{equation*}
	\eqref{eq:HighHighLowII} becomes with $\frac{N_1}{N}$ and recall that $(N \leq N_1^{-\alpha}, \; M_{\min} \leq N_1)$:
	\begin{equation*}
		\begin{split}
			\frac{N_1}{N} N N_1^{2- \frac{\alpha}{2}} \sum_{L \geq N_1^{2-\frac{\alpha}{2}}} L^{-\frac{1}{2}} \| \mathbf{1}_{\tilde{D}_{N,M,L}} (f^{(1)} \ast f^{(2)}) \|_{L^2} &\lesssim N^{\frac{1}{2}} N_1^{2-\frac{\alpha}{4}} M_{\min} \prod_{i=1}^2 L_i^{\frac{1}{2}} \| f^{(i)} \|_{L^2} \\
			&\lesssim N_1^{\frac{5}{2} - \frac{3 \alpha}{4}} M_{\min}^{\frac{1}{2}} \prod_{i=1}^2 L_i^{\frac{1}{2}} \| f^{(i)} \|_{L^2}.
		\end{split}
	\end{equation*}
	\eqref{eq:HHResonantI} becomes with $\frac{N_1}{N}$ and recall $(N_1^{-\alpha} \leq N \leq 1, \; L_{\max} \leq N_1^\alpha N)$:
	\begin{equation*}
		\frac{N_1}{N} N_1^{2-\frac{\alpha}{2}} N \sum_{L \leq N_1^\alpha N} L^{-\frac{1}{2}} \| \mathbf{1}_{\tilde{D}_{N,M,L}} (f^{(1)} \ast f^{(2)}) \|_{L^2} \lesssim \log(N_1^\alpha N) N^{\frac{1}{2}} M_{\min}^{\frac{1}{2}} N_1^{2- \frac{\alpha}{2}} \prod_{i=1}^2 L_i^{\frac{1}{2}} \| f^{(i)} \|_{L^2}.
	\end{equation*}
	\eqref{eq:HHResonantA} becomes with $\frac{N_1}{N}$ and $(N_1 \leq M_{\min}, \; L_{\max} \geq N_1^\alpha N)$:
	\begin{equation*}
		\frac{N_1}{N} N_1^{2-\frac{\alpha}{2}} N \sum_{L \geq N_1^\alpha N} L^{-\frac{1}{2}} \| \mathbf{1}_{\tilde{D}_{N,M,L}} (f^{(1)} \ast f^{(2)} \|_{L^2} \lesssim N_1^{3 - \alpha} M_{\min}^{\frac{1}{2}+} \prod_{i=1}^2 L_i^{\frac{1}{2}} \| f^{(i)} \|_{L^2}.
	\end{equation*}
	For $N_1 \geq M_{\min}$, we find for \eqref{eq:HHResonantB} with $(N_1 \geq M_{\min}, \; L_{\max} \geq N_1^\alpha N)$:
	\begin{equation*}
		\frac{N_1}{N} N_1^{2-\frac{\alpha}{2}} N \sum_{L \geq N_1^\alpha N} L^{-\frac{1}{2}} \| \mathbf{1}_{\tilde{D}_{N,M,L}} (f^{(1)} \ast f^{(2)} \|_{L^2} \lesssim N_1^{\frac{5}{2}-\frac{3 \alpha}{4}} M_{\min}^{\frac{1}{2}} \prod_{i=1}^2 L_i^{\frac{1}{2}} \| f^{(i)} \|_{L^2}.
	\end{equation*}
	
	\eqref{eq:HHLowResonantC} becomes with $\frac{N_1}{N}$ and recall $(N \geq 1, \; L_{\max} \leq N_1^\alpha N)$:
	\begin{equation*}
		\begin{split}
			&\quad \frac{N_1}{N} N \big( \frac{N_1}{N} \big)^{(2-\frac{\alpha}{2})} \sum_{N^{(2-\frac{\alpha}{2})} \leq L_{\max} \leq N_1^\alpha N} L^{-\frac{1}{2}} \| \mathbf{1}{_{D_{N,M,L}}} (f^{(1)} * f^{(2)}) \|_{L^2} \\
			&\lesssim N_1^{2-\frac{\alpha}{2}} N^{-\frac{3}{2}+\frac{\alpha}{2}} M_{\min}^{\frac{1}{2}} \log(N_1^\alpha N) \prod_{i=1}^2 L_i^{\frac{1}{2}} \| f^{(i)} \|_{L^2}.
		\end{split}
	\end{equation*}
	This is the estimate, which requires us to suppose that
	\begin{equation*}
		s> \begin{cases}
			\frac{5}{2}-\frac{\alpha}{2}, &\quad 2 \leq \alpha < 3, \\
			1, &\quad 3 \leq \alpha < 4.
		\end{cases}
	\end{equation*}
	\eqref{eq:HHLowResonantD} becomes for $N_1 \leq M_{\min}$ with $\frac{N_1}{N}$ and $(N \geq 1, \; L_{\max} \geq N_1^\alpha N)$:
	\begin{equation*}
		\lesssim N_1^{3-\alpha} N^{\frac{\alpha}{2}-2} M_{\min}^{\frac{1}{2}} \prod_{i=1}^2 L_i^{\frac{1}{2}} \| f^{(i)} \|_{L^2}.
	\end{equation*}
	\eqref{eq:HHLowResonantE} becomes for $(N \geq 1, \; M_{\min} \leq N_1, \; L_{\max} \geq N_1^\alpha N)$ with the factor $\frac{N_1}{N}$:
	\begin{equation*}
		\lesssim N_1^{2 - \frac{3\alpha}{4}} N^{-\frac{5}{2} + \frac{\alpha}{2}} M_{\min}^{\frac{1}{2}} \prod_{i=1}^2 L_i^{\frac{1}{2}} \| f^{(i)} \|_{L^2}.
	\end{equation*}
	
	This completes the proof of Proposition~\ref{prop:ShortTimeBilinear}.
	
	\section{Energy estimates}\label{section:EnergyEstimates}
	The purpose of this section is to propagate the energy norms in terms of short-time norms. We shall do this for solutions in $F^s$ and for differences of solutions in $F^0$. The key ingredient
	is the dyadic estimate in Proposition \ref{prop:DyadicEnergyEstimate}, to which the estimates are reduced after suitable integration by parts and substitutions.
	\subsection{Energy estimates for the solution}
	We begin with energy estimates for solutions.
	\begin{proposition}\label{prop:EnergyEstimates}
		Let $\alpha \in [2,4)$. For all $T\in (0,1]$ and solutions $u \in C([-T,T]; L^2) \cap L_T^\infty E^s$ to the following
		\begin{equation}\label{eq:SolutionEnergyEstimates}
			\left\{ \begin{array}{cl}
				\partial_t u - \partial_{x} D_x^\alpha u - \partial_{x}^{-1} \Delta_y u &= \partial_x (u^2), \quad (x,y) \in \D_\lambda, ~~t \in
				[-T,T], \\
				u(0) &= \phi,
			\end{array} \right.
		\end{equation}
		we have
		\begin{equation}\label{eq:ShorttimeEnergyEstimate}
			\|u\|_{B^{s}(T)}^2 \lesssim \| \phi\|_{E^{s} }^2 + \lambda^{0+} \|u\|^2_{F^{s}(T)} \| u \|_{F^{r}(T)}
		\end{equation}
		provided that $s \geq r > 3 - \frac{\alpha}{2}$ .
	\end{proposition}

	\subsection{Reductions}\label{subsection:Reductions}	We consider the equation \eqref{eq:SolutionEnergyEstimates} for the Littlewood-Paley pieces $P_{N} u$. Multiplying this equation by $P_{N} u$ and integrating, we get
	\begin{equation}\label{eq:FToCSolution}
		\sup_{|t_N|\leqslant T} \|  P_{N} u(t_N)\|_{L^2}^2 \leqslant \| P_{N} \phi\|_{L^2}^2 + \sup_{|t_N|\leqslant T} \Big| \int_{\D_\lambda \times [0,t_N]} P_{N} u \cdot P_{N} \partial_x(u^2) dxdydt \Big|.
	\end{equation}
	We write the integrand as
	\begin{equation*}
		\begin{split}
			P_N u P_N(\partial_x u^2) =2P_NuP_N (u\partial_x u)  = 2P_N u P_N(P_{\gtrsim N} u \cdot \partial_x u) + 2P_N u P_N(P_{\ll N}u \cdot \partial_x u):=2(I+II).
		\end{split}
	\end{equation*}
	We can further decompose $I$ as
	\begin{equation*}
		\begin{split}
			I = P_N u P_N(P_{\gg N}u \cdot P_{\gg N}\partial_x u) + P_N u P_N (P_{\gtrsim N}u \cdot P_{\lesssim N} \partial_x u )=:a+b,
		\end{split}
	\end{equation*}
	while $II$ can be written as
	\begin{equation*}
		II= P_N u P_{\ll N}u P_N \partial_x u + P_N u [P_N(P_{\ll N}u \partial_x u )-P_{\ll N}u P_N\partial_x u]:=c+d.
	\end{equation*}
	
	We have
	\begin{equation*}
		\begin{split}
			2a+2c &= P_N u \cdot \partial_x(P_{\gg N} u)^2 + \partial_x(P_N u)^2 \cdot P_{\ll N}u.
		\end{split}
	\end{equation*}
	For $b$, we observe that the derivative already hits the low frequency, while for $a+c$, using an integration by parts, we have
	\begin{equation*}
		\int_{\D_\lambda \times [0,T]} 2(a+c) ~dxdydt  = \sum_{N \ll N_1}\int_{\D_\lambda \times [0,T]} \partial_x P_N u \cdot P_{N_1}u \cdot \tilde{P}_{N_1}u ~ dxdydt +\sum_{N_1 \ll N} \int_{\D_\lambda \times [0,T]} \partial_x P_{N_1}u \cdot P_N u\cdot P_N u~dxdydt,
	\end{equation*}
	where $\tilde{P}_{N_1}=\sum_{N'\sim N_1} P_{N_1}$ (and the multiplier is denoted by $\tilde{\phi}_{N_1}$).
	
	Next, we treat $d$ by the same argument as in \cite[eq. 6.10]{IKT} to transfer the derivative to the low frequency factor. We have
	\begin{equation*}
		d= P_N u\sum_{N_1\ll N} [P_N(P_{N_1}u \cdot \partial_x u )-P_{N_1}u\cdot  P_N\partial_x u].
	\end{equation*}
	We fix an extension of $u$ which we still denote by $u$. We have
	\begin{equation*}
		\begin{split}
			&\mathcal{F}[P_N(P_{N_1}u \partial_x u )-P_{N_1}u P_N\partial_x u](\xi,\eta,\tau) \\
			=&\int \frac{(\xi-\xi_1)}{\xi_1} (\tilde{\phi}_N(\xi)-\tilde{\phi}_N(\xi-\xi_1))\widehat{P_{N}u}(\xi-\xi_1,\eta-\eta_1,\tau-\tau_1) \widehat{P_{N_1}\partial_x u} (\xi_1,\eta_1,\tau_1)d\xi_1d\eta_1d\tau_1 =:M(P_N u,P_{N_1} \partial_xu),
		\end{split}
	\end{equation*}
	where the (bilinear) multiplier is
	\[
	m(\xi,\xi_1)=\frac{(\xi-\xi_1)}{\xi_1} (\tilde{\phi}_N(\xi)-\tilde{\phi}_N(\xi-\xi_1)) \tilde{\phi}_{N_1}(\xi_1).\]
	Using the mean value theorem,
	\begin{equation*}
		(\tilde{\phi}_N(\xi)-\tilde{\phi}_N(\xi-\xi_1))  = -\int_0^1 \tilde{\phi}_{N}^{'}(\xi-h\xi_1) \xi_1 ~dh,
	\end{equation*}
	we obtain the uniform boundedness of the multiplier
	\begin{equation*}
		|m(\xi,\xi_1)| \lesssim \int_0^1  \Big|\frac{\xi-\xi_1}{N}\tilde{\phi}^{'}\Big(\frac{\xi-h\xi_1}{N}\Big) \Big|~dh \lesssim 1.
	\end{equation*}
	To conclude, we have shown that by taking the advantage of the form of the nonlinearity, we can transfer the derivative to the low frequency in all the cases. More precisely, we can assume that our integrand is of the form
	\begin{equation}\label{eq:ReducedIntegrand}
		P_{N}u\cdot P_{N_2}u \cdot (P_{N_1}\partial_x u) \text{ or } P_{N}u \cdot  M(P_{N}u, P_{N_1}\partial_x u) \quad \text{ with } N_1\lesssim N\sim N_2
	\end{equation}
	with a bilinear Fourier-multiplier $M$ with bounded symbol $m$. \\
	
	Considering an integrand of the form $P_{N}u\cdot P_{N_2}u \cdot (P_{N_1}\partial_x u)$, we divide the time interval into sub-intervals $I$ of size $N^{-(2-\frac{\alpha}{2})}$ to estimate the functions in short-time norms. Let $\gamma \in C^\infty_c([-1,1];\R_{+})$ such that
	\begin{equation*}
		\sum_{n \in \Z} \gamma^3(t-n) \equiv 1.
	\end{equation*}
	We write
	\begin{equation}
		\label{eq:IntegrationByParts}
		\begin{split}
			&\Big| \int_{\D_\lambda \times [0,T]}  P_{N}u\cdot P_{N_2}u \cdot (P_{N_1}\partial_x u) ~dx dy dt \Big| \\
			&\quad \lesssim \sum_n \Big| \int_{\D_\lambda \times [0,T]} (\gamma (N^{-(2-\frac{\alpha}{2})} t-n)P_N  u)\cdot  ( \gamma(N^{-(2-\frac{\alpha}{2})} t - n)P_{N_2} u )\cdot  ( \gamma(N^{-(2-\frac{\alpha}{2})} t - n) P_{N_1} \partial_x u ) ~ dx dy dt \Big|.
		\end{split}
	\end{equation}
	We consider the sets
	\begin{align*}
		A &= \{ n \in \Z : \text{supp} (\gamma (N^{-(2-\frac{\alpha}{2})} \cdot - n)) \subseteq (0,T) \}, \\
		B &= \{ n \in \Z : 0 \in \text{supp} (\gamma(N^{-(2-\frac{\alpha}{2})} \cdot - n)) \vee T \in \text{supp}(\gamma(N^{-(2-\frac{\alpha}{2})} \cdot - n)) \}.
	\end{align*}
	Note that $|A| \sim T N^{(2-\frac{\alpha}{2})}$ and $|B| \leqslant 4$.
	First, we consider the bulk of the cases given by the set $A$.
	We change for $n \in A$ in \eqref{eq:IntegrationByParts} to Fourier space after an additional dyadic decomposition for the $y$-frequencies: Let
	\begin{equation*}
		\begin{split}
			f^{(1)}_{N_1,M_1} &= \mathcal{F}_{x,y,t} [ \gamma(N^{-(2-\frac{\alpha}{2})} t-n) P_{N_1,M_1} u ], \quad f^{(2)}_{N_2,M_2} = \mathcal{F}_{x,y,t} [\gamma(N^{-(2-\frac{\alpha}{2})} t-n) P_{N_2,M_2} u], \\
			f^{(3)}_{N_3,M_3} &= \mathcal{F}_{x,y,t} [\gamma(N^{-(2-\frac{\alpha}{2})} t-n) P_{N,M_3} u].
		\end{split}
	\end{equation*}
	We make an additional dyadic decomposition in modulation for $L \geq N_3^{(2-\frac{\alpha}{2})}$ according to the time localization:
	\begin{equation*}
		f^{(i)}_{N_i,M_i}= \sum_{L_i \geq N_3^{(2-\frac{\alpha}{2})}} f^{(i)}_{N_i,M_i,L_i}, \qquad f^{(i)}_{N_i,M_i,L_i} = \mathbf{1}_{D_{N_i,M_i,L_i}} f^{(i)}_{N_i,M_i}.
	\end{equation*}
	
	The same can be imitated for an integrand of the type $P_{N}u \cdot  M(P_{N}u, P_{N_1}\partial_x u)$ because we require a bound in terms of the $L^2$ norm on the right hand side. Taking into account the additional derivatives and time localization, the proof of Proposition \ref{prop:EnergyEstimates} reduces to the following estimate:
	\begin{equation}
		\label{eq:DyadicEnergyEstimate}
		\sum_{L_i \geq N_3^{(2-\frac{\alpha}{2})}} \big|  \int f^{(1)}_{N_1,M_1,L_1} \cdot (f^{(2)}_{N_2,M_2,L_2} \ast f^{(3)}_{N_3,M_3,L_3} ) \big| \lesssim_\lambda N_1^{(0+,\frac{1}{2}+)} N_3^{-1} M_{\min}^{\frac{1}{2}} \prod_{i=1}^3 \sum_{L_i \geq N_3^{(2-\frac{\alpha}{2})}}L_i^{\frac{1}{2}}  \| f^{(i)}_{N_i,M_i,L_i} \|_{L^2}
	\end{equation}
	for $N_1 \leq N_2 \sim N_3$, $N_1 \in 2^{\Z}$, $N_2,N_3 \in 2^{\N_0}$. Like in Section \ref{section:ShorttimeBilinear}, we suppose that $N_i \geq \nu^{-1}$ if $\D_\lambda = \T_\nu \times \R^{d_1} \times \T^{d_2}_\lambda$. We always suppose that $M_i \geq \lambda^{-1}$, which is the lowest scale for $y$-frequencies.
	
	Once the above display is proved, we can conclude the proof of \eqref{eq:ShorttimeEnergyEstimate} by \eqref{prop}:
	Recall that
	\begin{equation*}
		p_\lambda(\xi,\eta) = \lambda^{-\frac{1}{2}} + \frac{|\eta|}{|\xi|}.
	\end{equation*}
	The constant term is then estimated by the above argument, and we trade $\eta$-factors into the weight and powers of $\xi$ like in Section \ref{section:ShorttimeBilinear}:
	\begin{equation*}
		|\eta|^{\frac{1}{2}} \lesssim
		\begin{cases}
			(1+|\xi|)^{\frac{1}{2}}, \quad &|\eta| \lesssim |\xi|, \\
			\big( 1+ \frac{|\eta|}{|\xi|} \big) |\xi|^{\frac{1}{2}}, \quad &|\eta| \gtrsim |\xi|.
		\end{cases}
	\end{equation*}
	This way the factor $M_{\min}^{\frac{1}{2}}$ is traded to $N_1^{\frac{1}{2}}$. Together with the factor $N_3^{(2-\frac{\alpha}{2})}$ from the time localization, we see how \eqref{eq:DyadicEnergyEstimate} can be summed to \eqref{eq:ShorttimeEnergyEstimate} by \eqref{prop}. To deal with the second term from the weight,	 we consider the equation for the Littlewood-Paley piece $P_N \partial_x^{-1} \nabla_y u$:
	\begin{equation*}
		\| P_N \partial_x^{-1} \nabla_y u(t_N) \|^2_{L^2} \leq \| P_N \partial_x^{-1} \nabla_y u(0) \|_{L^2}^2 + \big| \int_{\D_\lambda \times [0,t_N]} P_N \partial_{x}^{-1} \nabla_y u P_N (\nabla_y u^2) dx dy dt \big|.
	\end{equation*}
	We write
	\begin{equation*}
		\int_{\D_\lambda \times [0,t_N]} P_N (\partial_x^{-1} \nabla_y u) P_N ((\partial_x \partial_x^{-1} \nabla_y u) \cdot u), \text{ and let } v = \partial_x^{-1} \nabla_y u.
	\end{equation*}
	Then,
	\begin{equation*}
		\begin{split}
			\int_{\D_\lambda \times [0,t_N]} P_N v P_N (\partial_x v \cdot u) dx dy dt &= \int_{\D_\lambda \times [0,t_N]} P_N v (P_N \partial_x v P_{\ll N} u ) dx dy dt + \int_{\D_\lambda \times [0,t_N]} P_N v P_N (P_{\ll N} \partial_x v \cdot u) dx dy dt \\
			&\quad + \int_{\D_\lambda \times [0,t_N]} P_N v P_N (P_{\gtrsim N} \partial_x v P_{\gtrsim N} u ) dx dy dt.
		\end{split}
	\end{equation*}
	The first term can be reduced to the second after integration by parts and the third does not require integration by parts. Then, it suffices to show the estimates for $N' \lesssim N$:
	\begin{equation}
		\label{eq:EnergyEstimateReductionI}
		\sum_{\substack{N \geq 1, \\ N' \lesssim N}} N^{2s } \big| \int_{\D_\lambda \times [0,t_N]} P_N v P_N v (P_{N'} \partial_x u) dx dy dt \big| \lesssim_\lambda \| v \|^2_{\bar{F}^{s}(T)} \| u \|_{F^r(T)},
	\end{equation}
	and for $N_1' \sim N_2' \gtrsim N$:
	\begin{equation}
		\label{eq:EnergyEstimateReductionII}
		\sum_{\substack{N \geq 1, \\  N_1' \sim N_2' \gtrsim N}} N^{2s} \big| \int_{\D_\lambda \times [0,t_N]} P_N v P_N( P_{N_1'} v \cdot P_{N_2'} \partial_x u) dx dy dt \big| \lesssim_\lambda  \| v \|^2_{\bar{F}^{s}(T)} \| u \|_{F^r(T)}.
	\end{equation}
	In the above display we denote with $\bar{F}^s$ the short-time space in the Sobolev scale, i.e., $F^s$ without the weight $p_\lambda$. 	The estimates \eqref{eq:EnergyEstimateReductionI} and \eqref{eq:EnergyEstimateReductionII} can both be reduced to \eqref{eq:DyadicEnergyEstimate}, which we summarize in the following proposition:
	\begin{proposition}
		\label{prop:DyadicEnergyEstimate}
		Let $N_1 \in 2^{\Z}$, $N_2, N_3 \in 2^{\N_0}$, $N_1 \leq N_2 \sim N_3$, $L_i \in 2^{\N_0}$, $M_i \in 2^{\N_0},$  $i=1,2,3$, and $f^{(i)}_{N_i,M_i,L_i} \in L^2(\D_\lambda^* \times \R;\R_{+})$ with $\text{supp}(f^{(1)}) \subseteq \tilde{D}_{N_1,M_1,L_1} $, and $\text{supp}(f^{(k)}) \subseteq D_{N_k,M_k,L_k}$ for $k=2,3$. If $\D_\lambda = \T_\nu \times \R^{d_1} \times \T_\lambda^{d_2}$, we additionally require $N_1 \geq \nu^{-1}$. Then the following estimate holds:
		\begin{equation*}
			\begin{split}
				&\quad	\sum_{L_i \geq N_3^{(2-\frac{\alpha}{2})}} \Big| \int_{\D_\lambda^* \times \R} f^{(1)}_{N_1,M_1,L_1} \cdot (f_{N_2,M_2,L_2}^{(2)} \ast f_{N_3,M_3,L_3}^{(3)}) d\xi d\eta d\tau \Big| \\
				&\lesssim_\lambda N_1^{(0+,\frac{1}{2}+)}  N_3^{-1} M_{\min}^{\frac{1}{2}} \prod_{i=1}^3 \sum_{L_i \geq N_3^{(2-\frac{\alpha}{2})}} L_i^{\frac{1}{2}} \| f^{(i)}_{N_i,M_i,L_i} \|_{L^2}.
			\end{split}
		\end{equation*}
		\begin{proof}
			We consider  the following cases depending on the size of $N_1$ and $N_2$:\\
			\textbf{(i)} \underline{$N_1\leq N_2^{-\alpha}$}: In this case, we do not distinguish resonant and non-resonant interactions. Using Lemma \ref{lem:L4Summary}, we have
			\begin{equation*}
				\begin{split}
					\big| \int_{\D_\lambda^* \times \R} f^{(1)}_{N_1,M_1,L_1} \cdot (f_{N_2,M_2,L_2}^{(2)} \ast f_{N_3,M_3,L_3}^{(3)}) d\xi d\eta d\tau \big| &\lesssim_\lambda N_1^{\frac{1}{2}} M_{\min}^{0+} N_2^{\frac{1}{4}} N_3^{\frac{\big(\frac{\alpha}{2}-2 \big)}{2}} \prod_{i=1}^3 L_i^{\frac{1}{2}}\|f^{(i)}_{N_i,M_i,L_i}\|_{L^2}\\
					&\lesssim_\lambda N_1^{0+} N_2^{\frac{1}{4}} N_3^{-1-\frac{\alpha}{4}+} M_{\min}^{0+}\prod_{i=1}^3L_i^{\frac{1}{2}}\|f^{(i)}_{N_i,M_i,L_i}\|_{L^2},
				\end{split}
			\end{equation*}
			which is sufficient if $N_3\leq M_{\min}^2$. If $M_{\min}^2 \leq N_3$, we use Lemma \ref{lem:CauchySchwarz} to obtain
			\begin{equation*}
				\begin{split}
					\big| \int_{\D_\lambda^* \times \R} f^{(1)}_{N_1,M_1,L_1} \cdot (f_{N_2,M_2,L_2}^{(2)} \ast f_{N_3,M_3,L_3}^{(3)}) d\xi d\eta d\tau \big| &\lesssim M_{\min} N_1^{\frac{1}{2}} N_2^{\big( \frac{\alpha}{2} - 2 \big) }   \prod_{i=1}^3L_i^{\frac{1}{2}}\|f^{(i)}_{N_i,M_i,L_i}\|_{L^2}\\
					&\lesssim M_{\min}^{\frac{1}{2}} N_1^{0+} N_2^{-2+\frac{1}{4}} \prod_{i=1}^3L_i^{\frac{1}{2}}\|f^{(i)}_{N_i,M_i,L_i}\|_{L^2}.
				\end{split}
			\end{equation*}
			\textbf{(ii)} \underline{$N_2^{-\alpha} \leq N_1 \leq 1$}: We consider the resonant and non-resonant cases as follows:\\
			$\bullet ~L_{\max}\leq N_1N_2^\alpha$: From the bilinear Strichartz estimate \eqref{eq:dyadicbilinear} applied to $f^{(1)}_{N_1,M_1,L_1} \ast f^{(2)}_{N_2,M_2,L_2}$ (or $f^{(1)}_{N_1,M_1,L_1} \ast f^{(3)}_{N_3,M_3,L_3}$), we have
			\begin{equation*}
				\begin{split}
					&\quad \big| \int_{\D_\lambda^* \times \R} f^{(1)}_{N_1,M_1,L_1} \cdot (f_{N_2,M_2,L_2}^{(2)} \ast f_{N_3,M_3,L_3}^{(3)}) d\xi d\eta d\tau \big| \\
					&\lesssim   M_{\min}^{\frac{1}{2}} \frac{N_1^{\frac{1}{2}}}{N_3^{\frac{\alpha}{4}}} (L_1L_2)^{\frac{1}{2}} \|f^{(1)}_{N_1,M_1,L_1}\|_{L^2}\|f^{(2)}_{N_2,M_2,L_2}\|_{L^2}  \| f^{(3)}_{N_3,M_3,L_3}\|_{L^2}\\
					&\lesssim N_1^{\frac{1}{2}} N_2^{-1} M_{\min}^{\frac{1}{2}} \prod_{i=1}^3L_i^{\frac{1}{2}}\|f^{(i)}_{N_i,M_i,L_i}\|_{L^2}.
				\end{split}
			\end{equation*}
			$\bullet ~ L_{\max} \geq N_1N_2^\alpha$: We shall use Lemma \ref{lem:L4Summary} by estimating the function with the highest modulation in $L^2$. Here, the worst case occurs when $L_{\max} =L_1$. We have
			\begin{equation*}
				\begin{split}
					\big| \int_{\D_\lambda^* \times \R} f^{(1)}_{N_1,M_1,L_1} \cdot (f_{N_2,M_2,L_2}^{(2)} \ast f_{N_3,M_3,L_3}^{(3)}) d\xi d\eta d\tau \big| &\lesssim_\lambda N_1^{\frac{1}{2}} M_{\min}^{0+} N_2^{\frac{1}{2}} (N_1N_2^\alpha)^{-\frac{1}{2}} \prod_{i=1}^3L_i^{\frac{1}{2}}\|f^{(i)}_{N_i,M_i,L_i}\|_{L^2} \\
					&\lesssim_\lambda N_2^{\frac{1}{2}-\frac{\alpha}{2}} M_{\min}^{0+} \prod_{i=1}^3L_i^{\frac{1}{2}}\|f^{(i)}_{N_i,M_i,L_i}\|_{L^2},
				\end{split}
			\end{equation*}
			which is sufficient if $N_2 \leq M_{\min}$. If $M_{\min} \leq N_2$, we use Lemma \ref{lem:CauchySchwarz}:
			\begin{equation*}
				\begin{split}
					\big| \int_{\D_\lambda^* \times \R} f^{(1)}_{N_1,M_1,L_1} \cdot (f_{N_2,M_2,L_2}^{(2)} \ast f_{N_3,M_3,L_3}^{(3)}) d\xi d\eta d\tau \big| &\lesssim  N_1^{\frac{1}{2}} (N_1N_2^\alpha)^{-\frac{1}{2}} N_2^{ \frac{\big( \frac{\alpha}{2}-2\big)}{2}}  M_{\min} \prod_{i=1}^3L_i^{\frac{1}{2}}\|f^{(i)}_{N_i,M_i,L_i}\|_{L^2}\\
					&\lesssim N_2^{-\frac{1}{2}-\frac{\alpha}{4}} M^{\frac{1}{2}}_{\min} \prod_{i=1}^3L_i^{\frac{1}{2}}\|f^{(i)}_{N_i,M_i,L_i}\|_{L^2}.
				\end{split}
			\end{equation*}
			\textbf{(iii)} \underline{$N_1\gg 1$}: We consider two subcases:\\
			$\bullet ~ L_{\max}\leq N_1 N_2^\alpha$: After using Cauchy-Schwarz inequality, we apply the bilinear Strichartz estimate \eqref{eq:dyadicbilinear} to $f^{(1)}_{N_1,M_1,L_1} * f^{(j)}_{N_j,M_j,L_j}$, $j=2$ or $3$ and obtain
			\begin{equation*}
				\begin{split}
					\big| \int_{\D^* \times \R} f^{(1)}_{N_1,M_1,L_1} \cdot (f_{N_2,M_2,L_2}^{(2)} \ast f_{N_3,M_3,L_3}^{(3)}) d\xi d\eta d\tau \big| \lesssim    M_{\min}^{\frac{1}{2}} N_1^{\frac{1}{2}} N_3^{-\frac{\alpha}{4}} N_3^{(\frac{\alpha}{4}-1)} \prod_{i=1}^3 L_i^{\frac{1}{2}}\|f^{(i)}\|_{L^2},
				\end{split}
			\end{equation*}
			which is sufficient.\\
			$\bullet ~L_{\max}\geq N_1N_2^\alpha$: We assume $L_{\max} = L_1$, since the estimate is better or same in the other cases. For $M_{\min} \leq N_3$, we use Lemma \ref{lem:CauchySchwarz} to obtain
			\begin{equation*}
				\begin{split}
					\big| \int_{\D^* \times \R} f^{(1)}_{N_1,M_1,L_1} \cdot (f_{N_2,M_2,L_2}^{(2)} \ast f_{N_3,M_3,L_3}^{(3)}) d\xi d\eta d\tau \big| &\lesssim N_1^{\frac{1}{2}} M_{\min} L_{\min}^{\frac{1}{2}} \prod_{i=1}^3\|f^{(i)}_{N_i,M_i,L_i}\|_{L^2}\\
					&\lesssim M_{\min} N_3^{-\frac{\alpha}{2}} N_3^{(\frac{\alpha}{4}-1)} \prod_{i=1}^3L_i^{\frac{1}{2}}\|f^{(i)}_{N_i,M_i,L_i}\|_{L^2} \\
					&\lesssim M_{\min}^{\frac{1}{2}} N_3^{-1} \prod_{i=1}^3L_i^{\frac{1}{2}}\|f^{(i)}_{N_i,M_i,L_i}\|_{L^2}.
				\end{split}
			\end{equation*}
			For $M_{\min} \geq N_3$, we use Lemma \ref{lem:L4Summary} on $f^{(2)}_{N_2,M_2,L_2} * f^{(3)}_{N_3,M_3,L_3}$ as follows:
			\begin{equation*}
				\begin{split}
					\big| \int_{\D^* \times \R} f^{(1)}_{N_1,M_1,L_1} \cdot (f_{N_2,M_2,L_2}^{(2)} \ast f_{N_3,M_3,L_3}^{(3)}) d\xi d\eta d\tau \big| &\lesssim_\lambda \|f^{(1)}_{N_1,M_1,L_1}\|_{L^2} ~N_1^{\frac{1}{2}} N_2^{\frac{1}{2}} M_{\min}^{0+} \prod_{i=2}^3 L_i^{\frac{1}{2}} \|f^{(i)}_{N_i,M_i,L_i}\|_{L^2}\\
					&\lesssim_\lambda  (N_1N_2^\alpha)^{-\frac{1}{2}} N_1^{\frac{1}{2}} M_{\min}^{\frac{1}{2}+} \prod_{i=1}^3 L_i^{\frac{1}{2}} \|f^{(i)}_{N_i,M_i,L_i}\|_{L^2}.
				\end{split}
			\end{equation*}
		\end{proof}
	\end{proposition}

	\begin{proof}[Proof of Proposition \ref{prop:EnergyEstimates}]
		As described, we can consider the integrand to be of the form in \eqref{eq:ReducedIntegrand}. Consider the first case, i.e., $P_{N}u\cdot P_{N_2}u \cdot (P_{N_1}\partial_x u)$ with $N_1\lesssim N\sim N_2$. For $N_3=N$, we shall apply Proposition \ref{prop:DyadicEnergyEstimate} to the following:
		\begin{equation*}
			\begin{array}{clcc}
				f^{(1)}_{N_1,M_1} &= \mathcal{F}_{x,y,t} [ \gamma(N^{(2-\frac{\alpha}{2})} t-n) \partial_x P_{N_1,M_1} u]&=&\sum_{L_1 \geq N^{(2-\frac{\alpha}{2})}} f^{(1)}_{N_1,M_1,L_1}, \\
				f^{(2)}_{N_2,M_2} &= \mathcal{F}_{x,y,t} [\gamma(N^{(2-\frac{\alpha}{2})} t - n) P_{N_2,M_2} u ] &=& \sum_{L_2 \geq N^{(2-\frac{\alpha}{2})}} f^{(2)}_{N_2,M_2,L_2}, \\
				f^{(3)}_{N_3,M_3} &= \mathcal{F}_{x,y,t} [\gamma(N^{(2-\frac{\alpha}{2})} t -n) P_{N_3,M_3} u] &=& \sum_{L_3 \geq N^{(2-\frac{\alpha}{2})}} f^{(3)}_{N_3,M_3,L_3}.
			\end{array}
		\end{equation*}
		Secondly, the estimate \eqref{eq:EnergyEstimateReductionI} follows from an application of Proposition \ref{prop:DyadicEnergyEstimate} to
		\begin{equation*}
			\begin{array}{clcc}
				f^{(1)}_{N_1,M_1} &= \mathcal{F}_{x,y,t} [ \gamma(N^{(2-\frac{\alpha}{2})} t-n) \partial_x P_{N_1,M_1} u]&=&\sum_{L_1 \geq N^{(2-\frac{\alpha}{2})}} f^{(1)}_{N_1,M_1,L_1}, \\
				f^{(2)}_{N_2,M_2} &= \mathcal{F}_{x,y,t} [\gamma(N^{(2-\frac{\alpha}{2})} t - n) P_{N_2,M_2} v ] &=& \sum_{L_2 \geq N^{(2-\frac{\alpha}{2})}} f^{(2)}_{N_2,M_2,L_2}, \\
				f^{(3)}_{N_3,M_3} &= \mathcal{F}_{x,y,t} [\gamma(N^{(2-\frac{\alpha}{2})} t -n) P_{N_3,M_3} u] &=& \sum_{L_3 \geq N^{(2-\frac{\alpha}{2})}} f^{(3)}_{N_3,M_3,L_3}
			\end{array}
		\end{equation*}
		with $N_2 = N_3 = N$, $N_1 = N'$.
		
		Lastly, the estimate \eqref{eq:EnergyEstimateReductionII} follows from applying Proposition \ref{prop:DyadicEnergyEstimate} to
		\begin{equation*}
			\begin{array}{clcc}
				f^{(1)}_{N_1,M_1} &= \mathcal{F}_{x,y,t} [ \gamma({N'}^{(2-\frac{\alpha}{2})} t-n) P_{N_1,M_1} v]&=&\sum_{L_1 \geq {N'}^{(2-\frac{\alpha}{2})}} f^{(1)}_{N_1,M_1,L_1}, \\
				f^{(2)}_{N_2,M_2} &= \mathcal{F}_{x,y,t} [\gamma({N'}^{(2-\frac{\alpha}{2})} t - n) P_{N_2,M_2} v ] &=& \sum_{L_2 \geq {N'}^{(2-\frac{\alpha}{2})}} f^{(2)}_{N_2,M_2,L_2}, \\
				f^{(3)}_{N_3,M_3} &= \mathcal{F}_{x,y,t} [\gamma({N'}^{(2-\frac{\alpha}{2})} t -n) \partial_x P_{N_3,M_3} u] &=& \sum_{L_3 \geq {N'}^{(2-\frac{\alpha}{2})}} f^{(3)}_{N_3,M_3,L_3}
			\end{array}
		\end{equation*}
		with $N_1 = N$, $N_2 = N_3 = N'$.
	\end{proof}
	
	\subsection{Energy estimate for the difference of solutions}\label{subsec:diff}
	Let $u_1, u_2$ solve the equation \eqref{eq:SolutionEnergyEstimates} with initial data $\phi_1$ and $\phi_2$, respectively. The difference of the solutions viz. $v:=u_1-u_2$ satisfies the following equation
	\begin{equation}\label{eq:DiffEq}
		\left\{ \begin{array}{cl}
			\partial_t v + \partial_{x}^3 v - \partial_{x}^{-1} \Delta_y v &= \partial_x (v(u_1+u_2)), \quad (x,y) \in \D_{\lambda}, ~~~t\in  \times [-T,T], \\
			v(0) &= \phi_1-\phi_2 =: \phi.
		\end{array} \right.
	\end{equation}
	We have the following result for $v$.
	
	\begin{proposition}\label{prop:EnergyEstDiff}
		Let $\alpha \in [2,4)$, $s > 3 - \frac{\alpha}{2}$. For all $T\in (0,1]$, and with notations from \eqref{eq:DiffEq}, the following estimate holds:
		\begin{equation}
			\label{eq:EnergyEstimateDifferences}
			\| v \|^2_{B^0(T)} \lesssim \| v(0) \|^2_{E^0} + \lambda^{0+} \| v \|^2_{F^0(T)} ( \| u_1 \|_{F^s(T)} + \| u_2 \|_{F^s(T)} ).
		\end{equation}
		\begin{proof}
			As in the proof of Proposition \ref{prop:EnergyEstimates}, we consider the equation \eqref{eq:DiffEq} for Littlewood-Paley pieces  $ P_{N} v$.
			We write
			\begin{equation*}
				\| P_N v(t_N) \|^2_{L^2} = \| P_N v(0) \|^2_{L^2} + \int_{\D_\lambda \times [0,t_N]} P_N v P_N \partial_x (v \cdot (u_1 + u_2)) dx dy dt.
			\end{equation*}
			By an integration by parts argument, similar to the above, it suffices to estimate
			\begin{equation}
				\label{eq:DifferenceEnergyA}
				\big| \int_{\D_\lambda \times [0,t_N]} P_N v P_N v (P_{N_1} \partial_x u_i) dx dy dt \big| \text{ for } N_1 \lesssim N
			\end{equation}
			and
			\begin{equation}
				\label{eq:DifferenceEnergyB}
				\big| \int_{\D_\lambda \times [0,t_N] } P_N v \partial_x P_N (P_{\tilde{N}_1} v P_{\tilde{N}_2} u_i) dx dy dt \big| \text{ for } N \lesssim \tilde{N}_1 \sim \tilde{N}_2.
			\end{equation}
			Adding time localization and frequency localization in $\eta$, we let for \eqref{eq:DifferenceEnergyA} with $N_2 = N_3 = N$:
			\begin{equation*}
				\begin{array}{clcl}
					f^{(1)}_{N_1,M_1} &= \mathcal{F}_{x,y,t}[\gamma(N^{(2-\frac{\alpha}{2})} t - n) \partial_x P_{N_1,M_1} u_i] &=& \sum_{L_1 \geq N^{(2-\frac{\alpha}{2})}} f^{(1)}_{N_1,M_1,L_1}, \\
					f^{(2)}_{N_2,M_2} &= \mathcal{F}_{x,y,t}[\gamma(N^{(2-\frac{\alpha}{2})} t - n) P_{N_2,M_2} v] &=& \sum_{L_2 \geq N^{(2-\frac{\alpha}{2})}} f^{(2)}_{N_2,M_2,L_2}, \\
					f^{(3)}_{N_3,M_3} &= \mathcal{F}_{x,y,t}[\gamma(N^{(2-\frac{\alpha}{2})} t - n) P_{N_3,M_3} v] &=& \sum_{L_3 \geq N^{(2-\frac{\alpha}{2})}} f^{(3)}_{N_3,M_3,L_3}.
				\end{array}
			\end{equation*}
			Then, the claim follows from Proposition \ref{prop:DyadicEnergyEstimate}.
			
			For \eqref{eq:DifferenceEnergyB}, we let with $N_1 = N$ and $N_2 = \tilde{N}_1$, $N_3 = \tilde{N}_3$:
			\begin{equation*}
				\begin{array}{clcl}
					f^{(1)}_{N_1,M_1} &= \mathcal{F}_{x,y,t}[\gamma(\tilde{N}_1^{(2-\frac{\alpha}{2})} t - n) \partial_x P_{N_1,M_1} v] &=& \sum_{L_1 \geq \tilde{N}_1^{(2-\frac{\alpha}{2})}} f^{(1)}_{N_1,M_1,L_1},  \\
					f^{(2)}_{N_2,M_2} &= \mathcal{F}_{x,y,t}[\gamma(\tilde{N}_1^{(2-\frac{\alpha}{2})} t - n) P_{N_2,M_2} v] &=& \sum_{L_2 \geq \tilde{N}_1^{(2-\frac{\alpha}{2})}} f^{(2)}_{N_2,M_2,L_2}, \\
					f^{(3)}_{N_3,M_3} &= \mathcal{F}_{x,y,t}[\gamma(\tilde{N}_1^{(2-\frac{\alpha}{2})} t - n) P_{N_3,M_3} u_i] &=& \sum_{L_3 \geq \tilde{N}_1^{(2-\frac{\alpha}{2})}} f^{(3)}_{N_3,M_3,L_3}.
				\end{array}
			\end{equation*}
			Then, the claim is a consequence of Proposition \ref{prop:DyadicEnergyEstimate}.
			
			We shall also estimate the contribution of the weight by estimating the Littlewood-Paley pieces
			\begin{equation*}
				\| P_N \partial_x^{-1} \nabla_y v(t_N) \|_{L^2}^2 = \| P_N \partial_x^{-1} \nabla_y v(0) \|_{L^2}^2 + 2 \int_{\D_\lambda \times [0,t_N] } P_N (\partial_x^{-1} \nabla_y v) P_N (\nabla_y ( v( u_1 + u_2))) dx dt.
			\end{equation*}
			We let
			\begin{equation*}
				\begin{split}
					\int_{\D_\lambda \times [0,t_N] } P_N (\partial_x^{-1} \nabla_y v) P_N (\nabla_y v \cdot u_i) &= \int_{\D_\lambda \times [0,t_N]} P_N (\partial_x^{-1} \nabla_y v) P_N (\nabla_y v \cdot u_i) dx dt \\
					&\quad + \int_{\D_\lambda \times [0,t_N]} P_N (\partial_x^{-1} \nabla_y v) P_N(v \cdot \nabla_y u_i) dx dt \\
					&= I + II.
				\end{split}
			\end{equation*}
			Let $w = \partial_x^{-1} \nabla_y v$ and rewrite
			\begin{equation*}
				I = \int_{\D_\lambda \times [0,t_N]} P_N w P_N (\partial_x w \cdot u_i).
			\end{equation*}
			We decompose
			\begin{equation*}
				P_N (\partial_x w \cdot u_i) = P_N (\tilde{P}_N \partial_x w \cdot P_{\ll N} u_i) + P_N (P_{\ll N} \partial_x w \cdot P_N u_i) + P_N (P_{\gtrsim N} \partial_x w \cdot P_{\gtrsim N} u_i).
			\end{equation*}
			The first term allows for integration by parts and reduces to the second term. These contributions can be summed like above. Only the third term is a little different
			because the derivative cannot be transferred to the lowest frequency. However, in this case this is acceptable because of $High \times High \to Low$-interaction.
			We have to estimate
			\begin{equation*}
				\sum_{\substack{ N \geq 1, \\ N_1' \sim N_2' \geq N}} \big| \int_{\D_\lambda \times [0,t_N]} P_N w (P_{N_1'} \partial_x w) (P_{N'_2} u_i) dx dt \big|.
			\end{equation*}
			We smoothly localize time to intervals of size $N_1'^{(\frac{\alpha}{2} -2)}$ and let like above
			\begin{equation*}
				\begin{array}{clcl}
					f^{(1)} &= \mathcal{F}_{x,y,t}[\gamma(N_1'^{(2-\frac{\alpha}{2})} t - n) P_{N_1,M_1} w ] &=& \sum_{L_1 \geq N_1'^{(2-\frac{\alpha}{2})}} f^{(1)}_{N_1,M_1,L_1}, \\
					f^{(2)} &= \mathcal{F}_{x,y,t}[\gamma(N_1'^{(2-\frac{\alpha}{2})} t- n) P_{N_2,M_2} w] &=& \sum_{L_2 \geq N_1'^{(2-\frac{\alpha}{2})}} f^{(2)}_{N_2,M_2,L_2}, \\
					f^{(3)} &= \mathcal{F}_{x,y,t}[\gamma(N_1'^{(2-\frac{\alpha}{2})} t-n) P_{N_3,M_3} u_i] &=& \sum_{L_3 \geq N_1'^{(2-\frac{\alpha}{2})}} f^{(3)}_{N_3,M_3,L_3}
				\end{array}
			\end{equation*}
			with $N_1 = N$, $N_2 = N_1'$, $N_3 = N_2'$.
			The claim again follows from applying Proposition \ref{prop:DyadicEnergyEstimate}.
			
			We turn to the estimate of
			\begin{equation*}
				II = \int_{\D_\lambda \times [0,t_N]} P_N (\partial_x^{-1} \nabla_y v) P_N (v \nabla_y u_2).
			\end{equation*}
			As above, we write
			\begin{equation*}
				P_N (v \nabla_y u_2) = P_N (\tilde{P}_N v P_{\ll N} \nabla_y u_2) + P_N (P_{\ll N} v \tilde{P}_N \nabla_y u_2) + P_N (P_{\gtrsim N} v P_{\gtrsim N} \nabla_y u_2).
			\end{equation*}
			After additional decomposition in time and in the $y$-frequencies, the claim follows from applying Proposition \ref{prop:DyadicEnergyEstimate}. For the first term, we have to estimate
			\begin{equation*}
				\int_{\D_\lambda \times [0,1]} P_{N_1,M_1} (\partial_x^{-1} \nabla_y v \cdot \gamma(N^{(2-\frac{\alpha}{2})} t - n)) P_{N_2,M_2} (v \gamma (N^{(2-\frac{\alpha}{2})} t - n)) P_{N_3,M_3} (\partial_x \partial_{x}^{-1} \nabla_y u_i \gamma(N^{(2-\frac{\alpha}{2})} t - n)).
			\end{equation*}
			for $|n| \lesssim T N^{2-\frac{\alpha}{2}}$ and $M_i \in 2^{\N_0}$ with $N_1= N$, $N_2 \sim N$, $N_3 = N' \lesssim N$.
			Letting
			\begin{equation*}
				\begin{array}{clcl}
					f^{(1)} &= \mathcal{F}_{x,y,t}[\gamma(N^{(2-\frac{\alpha}{2})} t - n) P_{N_1,M_1} \partial_x^{-1} \nabla_y v ] &=& \sum_{L_1 \geq N^{(2-\frac{\alpha}{2})}} f^{(1)}_{N_1,M_1,L_1}, \\
					f^{(2)} &= \mathcal{F}_{x,y,t}[\gamma(N^{(2-\frac{\alpha}{2})} t- n) P_{N_2,M_2} v] &=& \sum_{L_2 \geq N^{(2-\frac{\alpha}{2})}} f^{(2)}_{N_2,M_2,L_2}, \\
					f^{(3)} &= \mathcal{F}_{x,y,t}[\gamma(N^{(2-\frac{\alpha}{2})} t-n) P_{N_3,M_3} \partial_x^{-1} \nabla_y u_i] &=& \sum_{L_3 \geq N^{(2-\frac{\alpha}{2})}} f^{(3)}_{N_3,M_3,L_3},
				\end{array}
			\end{equation*}
			the claim becomes a consequence of Proposition \ref{prop:DyadicEnergyEstimate} and carrying out the summations. The other terms are estimated likewise.
		\end{proof}
	\end{proposition}
	
	\section{Proof of Theorem \ref{thm:LWPAnisotropic}}\label{section:Proof}
	This section is devoted to the proof of Theorem \ref{thm:LWPAnisotropic}, which asserts low regularity well-posedness of the fractional KP-I equation:
	\begin{equation}
		\label{eq:fKPISection}
		\left. \begin{split}
			\partial_t u - \partial_x D_x^\alpha u - \partial_{x}^{-1} \Delta_y u &= \partial_x (u^2), \quad (t,x,y) \in \R \times \D, \\
			u(0) &= \phi \in E^s(\D)
		\end{split} \right\}
	\end{equation}
	with $\alpha \in [2,4]$, and $s$ and $\D$ to be specified. The starting point is the local existence on arbitrary domains $\D = \K_1 \times \K_2 \times \K_3$ with $\K_i \in \{ \R; \T \}$:
	\begin{proposition}
	\label{prop:LocalExistence}
		There is a data-to-solution mapping $S_T^\infty:E^\infty \to C_T E^\infty$ with $T=T(\| \phi \|_{E^4})$ of \eqref{eq:fKPISection}.
	\end{proposition}
	In the following we prove that we have data-to-solution mappings $S_T^s : E^s \to L_T^\infty E^s$ with $T=T(\| \phi \|_{E^4})$ for $s \geq 4$ with $u \in C^1_T H^{-1} \cap C_T H^1 \cap C_T E^0$, which is the unique distributional solution to \eqref{eq:fKPISection}. By interpolation, $u \in C_T E^{s'}$ for $0 \leq s' < s$. Hence, we find $S_T^\infty: E^\infty \to C_T E^\infty$. We have reduced to the following:
	\begin{proposition}[Local~existence~at~high~regularity]
\label{prop:LocalExistenceHighAppendix}
Let $s \geq 4$. For every $\phi \in E^s$ there exists $T=T(\| \phi \|_{E^4})$ and a unique solution $u$ to \eqref{eq:fKPISection} in the distributional sense on the time interval $[-T,T]$ satisfying
\begin{equation*}
u \in C_T H^\sigma \; (-1 \leq \sigma < 1), \quad \partial_t u \in L_T^\infty H^{-1}, \quad u \in C_T E^{s'} \cap L_T^\infty E^s \quad (0 \leq s' < s).
\end{equation*}
\end{proposition}

In the first step we construct solutions via Galerkin approximation. Let $\varphi \in C^\infty_c(\R^3)$ be a radially decreasing function with
\begin{equation*}
\varphi(\xi,\eta) \equiv 1 \text{ for } |(\xi,\eta)| \leq 1 \text{ and } \varphi(\xi,\eta) = 0 \text{ for } |(\xi,\eta)| \geq 2.
\end{equation*}
We define $\chi(\xi,\eta) = \varphi((\xi,\eta)/2) - \varphi(\xi,\eta)$, which is supported in $B(0,4) \backslash B(0,1)$. In the first step we consider the Galerkin approximations with low and high frequency cutoff for $M \in 2^{\N}$:
\begin{equation*}
(\tilde{P}_M f) \widehat (\xi,\eta) = \sum_{K=M^{-1}}^M \chi((\xi,\eta)/K) \hat{f}(\xi,\eta).
\end{equation*}
We consider
\begin{equation}
\label{eq:TruncatedEvolution}
\left. \begin{array}{cl}
\partial_t u - \partial_x D_x^\alpha u - \partial_x^{-1} \Delta_y u &= \tilde{P}_M (\partial_x (\tilde{P}_M u)^2), \\
u(0) &= \phi \in E^s(\D).
\end{array} \right\}
\end{equation}
By rewriting \eqref{eq:TruncatedEvolution} as an integral equation,
\begin{equation*}
u^M(t) = S_\alpha(t) \phi + \int_0^t S_\alpha(t-\tau) \tilde{P}_M (\partial_x (\tilde{P}_M u)^2(\tau)) d\tau,
\end{equation*}
we infer local existence in $E^s(\D)$ for $s \geq 0$ by the Cauchy-Picard-Lipschitz theorem. This is based on $S_\alpha(t)$ being bounded on $E^s(\D)$ and the nonlinearity trivially being bounded on $E^s(\D)$ by Sobolev embedding. This however yields a bad constant and the Cauchy-Picard-Lipschitz theorem yields an existence time, which depends on $M$. Denote the emanating solutions by $(u^M) \subseteq C_T E^4$.
To show a bound independent of $M$, we apply the analysis of the previous sections to \eqref{eq:TruncatedEvolution} for $s \geq 4$.

First we suppose that $\| \phi \|_{E^{4}} \leq \varepsilon_0 \ll 1$. From Lemma \ref{lemma:LinShortTime}, Proposition \ref{prop:ShortTimeBilinear}, and Proposition \ref{prop:EnergyEstimates}, we have the following set of estimates for $T= \min(T_{\max},1)$ with $T_{\max}$ denoting the time of existence according to the Cauchy-Picard-Lipschitz theorem in $C_T E^s$:
	\begin{equation}\label{est:first}
		\left. \begin{array}{cl}
			\| u^M \|_{F^{s}(T)} &\lesssim \| u^M \|_{B^{s}(T)} + \| \partial_x ((\tilde{P}_M u^M)^2) \|_{\mathcal{N}^{s}(T)}, \\
			\| \partial_x((\tilde{P}_M u^M)^2) \|_{\mathcal{N}^{s}(T)} &\lesssim \| u^M \|_{F^{s}(T)} \| u^M \|_{F^4(T)}, \\
			\| u^M \|^2_{B^{s}(T)} &\lesssim \| \phi \|^2_{E^{s}} +  \| u^M \|^2_{F^{s}(T)} \| u^M \|_{F^4(T)}.
		\end{array} \right\}
	\end{equation}
	Above we use monotonicity $\| \tilde{P}_M u^M \|_{X^s(T)} \leq \| u^M \|_{X^s(T)}$ with $X \in \{F,\mathcal{N} \}$.  For $s=4$, this gives
	\begin{equation}
		\label{eq:APrioriS}
		\| u^M \|_{F^{4}(T)}^2 \lesssim \| \phi \|_{E^{4}}^2 + \| u^M \|_{F^{4}(T)}^4 + \|u^M \|_{F^{4}(T)}^3.
	\end{equation}
	Hence, by choosing $\varepsilon_0$ small enough, we find by \eqref{eq:APrioriS}
	from the continuity of $\|u^M \|_{B^{4}(T)}$ in $T$,
	\begin{equation*}
		\lim_{T \rightarrow 0} \|u^M \|_{B^{4}(T)} \lesssim \| \phi \|_{E^{4}} , \text{ and } \lim_{T \rightarrow 0} \|\partial_x((\tilde{P}_M u^M)^2)\|_{\mathcal{N}^{4}(T)} =0,
	\end{equation*}
	that
	\begin{equation}
		\label{eq:FSApriori}
		\| u^M \|_{F^{4}(T)} \lesssim \| \phi \|_{E^{4}} \lesssim \varepsilon_0 \ll 1.
	\end{equation}
	Hence, by continuity of $T \mapsto \| u^M \|_{F^4(T)}$, $\sup_{t \in [-T,T]} \| u^M(t) \|_{E^s} \lesssim \| u^M \|_{F^4(T)}$ and iterating the Cauchy-Picard-Lipschitz theorem, we find that the time of existence for solutions in $C_T E^4$ satisfies the bound $T_{\max} \gtrsim 1$ provided that $\| \phi \|_{E^4} \leq \varepsilon_0 \ll 1$.

	Another application of Lemma \ref{lemma:LinShortTime}, Proposition \ref{prop:ShortTimeBilinear}, and Proposition \ref{prop:EnergyEstimates} for $s \geq 4$ yields
	\begin{equation}\label{est:second}
		\left. \begin{array}{cl}
			\| u^M \|_{F^s(T)} &\lesssim \| u^M \|_{B^s(T)} + \| \partial_x ((\tilde{P}_M u^M)^2) \|_{\mathcal{N}^s(T)}, \\
			\| \partial_x((\tilde{P}_M u^M)^2) \|_{\mathcal{N}^{s}(T)} &\lesssim \| u^M \|_{F^{s}(T)} \| u^M \|_{F^{4}(T)}, \\
			\| u^M \|^2_{B^{s}(T)} &\lesssim \| \phi \|^2_{E^{s}} +  \| u^M \|_{F^{s}(T)} \| u^M \|^2_{F^{4}(T)}.
		\end{array} \right\}
	\end{equation}
	This set of estimates gives
	\begin{equation}
		\label{eq:APrioriF3}
		\| u^M \|^2_{F^s(T)} \lesssim \| u_0 \|^2_{E^s} + \| u^M \|^2_{F^4(T)} \| u^M \|^2_{F^{s}(T)} + \| u^M \|_{F^4(T)} \| u^M \|^2_{F^{s}(T)},
	\end{equation}
	and therefore, for $\| u^M \|_{F^{4}(T)} \lesssim \| \phi \|_{E^{4}} \lesssim \varepsilon_0 \ll 1$ and $T=T_{\max}(\| \phi \|_{E^{4}}) \gtrsim 1$, we find
	\begin{equation}
		\label{eq:APriori3}
		\| u^M \|_{F^s(T)} \lesssim \| \phi \|_{E^s}.
	\end{equation}
	We summarize that we have a priori estimates
	\begin{equation*}
	\sup_{t \in [-1,1]} \| u^M(t) \|_{E^s} \lesssim \| \phi \|_{E^s}
	\end{equation*}
	for $s \geq 4$ provided that $\| \phi \|_{E^4} \leq \varepsilon_0 \ll 1$. 
We have now ensured existence of $(u^M)_{M \in 2^{\N}}$ on a common time interval. We observe that for any ball $B = B(0,N) \subseteq \D$\footnote{Restricting to compact domains is only necessary if $\D \neq \T^3$.} we have bounds uniform in $M$:
\begin{equation*}
u^M \in L_T^\infty H^1(B), \quad \partial_t u^M \in L_T^\infty H^{-1}(B).
\end{equation*}
By the compact embedding $H^1(B) \hookrightarrow L^2(B)$ together with the continuous embedding $L^2(B) \hookrightarrow H^{-1}(B)$, we can apply the Aubin--Lions compactness lemma (cf. \cite{Lions1969}) to find that there is a subsequence $u^M \to u \in C_T L^2_{loc}$. We have $u \in L_T^\infty E^s(\D)$, which yields $u \in L_T^\infty H^1(\D)$ and $\partial_{x}^{-1} \nabla_y u \in L_T^\infty L^2(\D)$. By dual pairing in $L^2$ for $\varphi \in C^\infty_c([-T,T],\D)$ and passing to the limit $M \to \infty$, we find that $u$ is a distributional solution to \eqref{eq:fKPISection}. Since $u \in C^1([-T,T],H^{-1}) \cap L_T^\infty H^1(\D)$, we conclude $u \in C([-T,T],H^\sigma(\D))$ for $\sigma \in [-1,1)$.
Next, we show $\partial_{x}^{-1} \nabla_y u \in C_T L^2$. We use Duhamel's formula for $t_1 \leq t_2$:
\begin{equation*}
\partial_x^{-1} \nabla_y u(t_1) - \partial_x^{-1} \nabla_y u(t_2) = (S_\alpha(t_1) - S_\alpha(t_2)) \partial_x^{-1} \nabla_y \phi + 2 \int_{t_1}^{t_2} S_\alpha(t-s) (u \nabla_y u)(s) ds.
\end{equation*}
The linear part converges to zero for $t_1 \to t_2$ since $(S_\alpha(t))_{t \in \R}$ is a $C_0$-semigroup in $L^2$. With $I = [t_1,t_2]$, we have
\begin{equation*}
\| \int_{t_1}^{t_2} S_\alpha(t-s) (u \nabla_y u)(s) ds \|_{L^2(\D)} \lesssim \| u \|_{L_I^1 L_{xy}^\infty(\D)} \| u \|_{L_T^\infty E^1(\D)}.
\end{equation*}
We use Strichartz estimates to improve on Sobolev embedding, which barely fails to control the $L_y^\infty$-contribution. This was previously used in the 2d case to obtain estimates in $H^{2,0}$ in \cite{MolinetSautTzvetkov2007}. We use the following Strichartz estimates, which is straight-forward from Section \ref{section:Notations} by Littlewood-Paley decomposition:
						\begin{proposition}[Linear $L^4$-Strichartz estimates]
						\label{prop:L4Strichartz}
							Let $T \in (0,1]$, $\alpha > 0$, $\phi: \D \to \C$, $u(t) = S_\alpha(t) \phi$, and $\varepsilon > 0$. The following estimate holds:
							\begin{equation}
								\label{eq:StrichartzEstimates}
								\| u \|_{L_T^4 L_{xy}^4(\D)} \lesssim \| \langle \partial_x \rangle^{\frac{1}{2}} \langle \partial_y \rangle^{\varepsilon} \phi \|_{L^2_{xy}}.
							\end{equation}
						\end{proposition}
Secondly, we need the following anisotropic Leibniz rule, possibly on mixed domains, whose proof is postponed to the Appendix:
						\begin{proposition}
							\label{prop:AnisotropicLeibnizRule}
							Let $\D = \K_1 \times \K_2 \times \K_3$ with $\K_i \in \{ \R; \T \}$, $\alpha,\beta \geq 0$, $\delta > 0$, and $\frac{1}{2} = \frac{1}{p} + \frac{1}{q}$. Then the following estimate holds:
							\begin{equation}
								\label{eq:AnisotropicFractionalLeibniz}
								\| \langle \partial_x \rangle^\alpha \langle \partial_y \rangle^\beta (u^2) \|_{L^2_{xy}(\D)} \lesssim \| \langle \partial_x \rangle^\alpha \langle \partial_y \rangle^\beta u \|_{L^2_{xy}(\D)} \| u \|_{L^\infty_{xy}(\D)} + \| \langle \partial_x \rangle^{\alpha + \delta} u \|_{L^2_{x} L_y^p(\D)} \| \langle \partial_y \rangle^\beta u \|_{L^\infty_{x} L_y^q(\D)}.
							\end{equation}
						\end{proposition}

						Based on Propositions \ref{prop:L4Strichartz}, \ref{prop:AnisotropicLeibnizRule}, and Sobolev embedding, we can estimate:
						\begin{equation}
						\label{eq:BoundL1Linf}
						\begin{split}
						\| u^M \|_{L_T^1 L_{xy}^\infty} &\lesssim T^{\frac{3}{4}} \| \langle \partial_x \rangle^{\frac{3}{4}+\varepsilon} \langle \partial_y \rangle^{\frac{1}{2}+\varepsilon} \phi \|_{L^2_{xy}} \\
						&\quad + T^{\frac{3}{4}} \| \langle \partial_x \rangle^{\frac{7}{4}+\varepsilon} \langle \partial_y \rangle^{\frac{1}{2}+\varepsilon} u^M \|_{L_T^\infty L^2_{xy}} \| u^M \|_{L_T^1 L^\infty_{xy}} + T^{\frac{7}{4}} \| \langle \partial_x \rangle^{\frac{7}{4}+\varepsilon} u^M \|_{L_T^\infty L_x^2 L_y^6} \| \langle \partial_y \rangle^{\frac{1}{2}+\varepsilon} u^M \|_{L_T^\infty L_x^\infty L_y^3} \\
						&\lesssim T^{\frac{3}{4}} \| \phi \|_{E^4} + T^{\frac{3}{4}} \| u^M \|_{L_T^\infty E^4} \| u^M \|_{L_T^1 L_{xy}^\infty} + T^{\frac{7}{4}} \| u^M \|^2_{L_T^\infty E^4}.
						\end{split}
						\end{equation}

By a priori estimates in $E^4$ uniform in $M$, we find $\| u^M \|_{L_I^1 L_{xy}^\infty} \to 0$ as $|I| \to 0$ uniformly in $M$. By limiting arguments, this shows continuity of $\partial_{x}^{-1} \nabla_y u \in C_T L^2$, hence $u \in C_T E^0$. By interpolation with $u \in L_T^\infty E^s(\D)$, we infer that $u \in C_T E^{s'}(\D)$ for $0 \leq s' < s$. This concludes the proof of existence for small initial data. 

\medskip

To extend the argument to large initial data, we use rescaling. Recalling \eqref{eq:ScaleInvariant}, we find that
	\begin{equation*}
		u_{\lambda}(x,y,t) = \lambda^{- \frac{2\alpha}{\alpha + 2}} u(\lambda^{-\frac{2}{\alpha + 2}} x, \lambda^{-1} y, \lambda^{-\frac{2(\alpha+1)}{\alpha+2}} t)
	\end{equation*}
	is solution on $\mathbb{D}_\lambda$ with scaled initial data $\phi^\lambda(x,y):=\lambda^{- \frac{2\alpha}{\alpha + 2}} \phi(\lambda^{-\frac{2}{\alpha + 2}} x, \lambda^{-1} y)$ whenever $u$ is a solution on $\mathbb{D}$ with $\phi$.
	
	\medskip
	
	We find $\| \phi^\lambda \|_{E^{4}} \lesssim_{\| \phi \|_{E^4}} \lambda^{-\kappa}$, that is, the norm is polynomially decreasing in $\lambda$. On the rescaled domain, we consider
\begin{equation}
\label{eq:TruncatedEvolutionRescaled}
\left. \begin{array}{cl}
\partial_t u^M_\lambda - \partial_x D_x^\alpha u^M_\lambda - \partial_x^{-1} \Delta_y u^M_\lambda &= \tilde{P}_M (\partial_x (\tilde{P}_M u^M_\lambda)^2), \\
u^M_\lambda(0) &= \phi^\lambda \in E^s(\D_\lambda).
\end{array} \right\}
\end{equation}	
 	Like above, for $s \geq 4$ we have the following set of estimates for $T= \min(T_{\max},1)$ with $T_{\max}$ denoting the time of existence according to the Cauchy-Picard-Lipschitz theorem in $C_T E^s$:
	\begin{equation}\label{est:third}
		\left. \begin{array}{cl}
			\| u^M_{\lambda} \|_{F^{s}(T)} &\lesssim \| u^M_\lambda \|_{B^{s}(T)} + \| \partial_x ((\tilde{P}_M u^M_\lambda)^2) \|_{\mathcal{N}^{s}(T)}, \\
			\| \partial_x((\tilde{P}_M u^M_\lambda)^2) \|_{\mathcal{N}^{s}(T)} &\lesssim \lambda^{0+} \| u^M_\lambda \|_{F^{s}(T)}^2, \\
			\| u^M_\lambda \|^2_{B^{s}(T)} &\lesssim \| \phi^\lambda \|^2_{E^{s}} + \lambda^{0+} \| u^M_\lambda \|_{F^{s}(T)}^3.
		\end{array} \right\}
	\end{equation}
	Applying the set of estimates with $s=4$, we find
	\begin{equation*}
		\| u^M_\lambda \|_{F^4(T)}^2 \lesssim \| \phi^\lambda \|^2_{E^{4}} + \lambda^{0+} \| u^M_\lambda \|_{F^{4}(T)}^4 + \lambda^{0+} \| u^M_\lambda \|_{F^4(T)}^3.
	\end{equation*}
	Since $\| \phi^\lambda \|_{E^4}$ is polynomially decreasing in $\lambda$, we can choose $\lambda = \lambda(\| \phi \|_{E^4})$ large enough such that we obtain like in the case of small initial data the a priori estimate:
	\begin{equation*}
		\sup_{t \in [-1,1]} \| u^M_\lambda(t) \|_{E^4(\D_\lambda)} \lesssim \| \phi^\lambda \|_{E^4}.
	\end{equation*}
	Like above we can infer the existence at higher regularities $s \geq 4$, likewise on $[-1,1]$. By the same compactness arguments, we obtain a distributional solution $u_\lambda$ to 
	\begin{equation*}
		\left. \begin{split}
			\partial_t u_\lambda - \partial_x D_x^\alpha u_\lambda - \partial_{x}^{-1} \Delta_y u_\lambda &= \partial_x (u_\lambda^2), \quad (t,x,y) \in \R \times \D_\lambda, \\
			u_\lambda(0) &= \phi^\lambda \in E^s(\D_\lambda)
		\end{split} \right\}
	\end{equation*}
	
	 Moreover, repeating the arguments from the small-data analysis yields $u_\lambda \in C_T H^\sigma(\D_\lambda)$ for $\sigma \in [-1,1)$ and $u_\lambda \in C_T E^{s'}(\D_\lambda) \cap L_T^\infty E^s(\D_\lambda)$ for $0 \leq s' < s$. We can scale back to infer that for $T=T(\lambda) = T(\| \phi \|_{E^4})$ we have $u \in C_T H^\sigma(\D)$ for $\sigma \in [-1,1)$ and $u \in C_T E^{s'}(\D) \cap L_T^\infty E^s(\D)$ for $0 \leq s' < s$.
	 
\medskip	 
	 
	 This completes the proof of existence of distributional solutions with the claimed regularity properties. We still have to prove uniqueness. To this end, we consider differences of solutions $v = u_1 - u_2$ with initial data $\phi_i \in E^s$, $s \geq 4$, $i=1,2$. Suppose that 
\begin{equation}
\label{eq:FinitenessL1Linf}	 
	 \| \partial_x u \|_{L_T^1 L_{xy}^\infty(\D)} < \infty
\end{equation}	 
	  for the constructed solutions. We compute by integration by parts, H\"older's inequality, and Gr\o nwall's lemma:
	 \begin{equation}
	 \label{eq:DifferenceBound}
	 \partial_t \| v(t) \|_{L^2}^2 \lesssim \| \partial_x(u_1 + u_2) \|_{L^\infty_{xy}} \| v(t) \|^2_{L^2} \Rightarrow \| v(t) \|_{L^2}^2 \lesssim e^{c \int_0^t \| \partial_x (u_1 + u_2) \|_{L^\infty_{xy}}} \| v(0) \|_{L^2}^2.
	 \end{equation}
	Hence, \eqref{eq:FinitenessL1Linf} would clearly imply uniqueness. 
	We apply Strichartz estimates from Proposition \ref{prop:L4Strichartz} together with Sobolev embedding and the fractional Leibniz rule to find
							\begin{equation*}
								\begin{split}
									\| \partial_x u \|_{L_T^1 L_{xy}^\infty} &\lesssim T^{\frac{3}{4}} \| \langle \partial_x \rangle^{\frac{7}{4}+\varepsilon} \langle \partial_y \rangle^{\frac{1}{2}+\varepsilon} \phi \|_{L^2_{xy}} + T^{\frac{3}{4}} \| \langle \partial_x \rangle^{\frac{11}{4}+\varepsilon} \langle \partial_y \rangle^{\frac{1}{2}+\varepsilon} (u^M)^2 \|_{L_T^1 L_{xy}^2} \\
									&\lesssim T^{\frac{3}{4}} \| \phi \|_{E^4} + T^{\frac{3}{4}} \| \langle \partial_x \rangle^{\frac{11}{4}+\varepsilon} \langle \partial_y \rangle^{\frac{1}{2}+\varepsilon} u \|_{L_T^\infty L_{xy}^2} \| u \|_{L_T^1 L_{xy}^\infty} \\
									&\quad + T^{\frac{7}{4}} \| \langle \partial_x \rangle^{\frac{11}{4}+\varepsilon} u \|_{L_T^\infty L_x^2 L_y^{6} } \| \langle \partial_y \rangle^{\frac{1}{2}+\varepsilon} u \|_{L_T^\infty L_x^\infty L_y^{3}} \\
									&\lesssim T^{\frac{3}{4}} \| \phi \|_{E^4} + T^{\frac{3}{4}} \| u \|_{L_T^\infty E^4} \| u \|_{L_T^1 L_{xy}^\infty} + T^{\frac{7}{4}} \| u \|_{L_T^\infty E^4 }^2.
								\end{split}
							\end{equation*}
By the bound of $\| u \|_{L_T^1 L_{xy}^\infty}$ proved in \eqref{eq:BoundL1Linf} after taking limits, and $u \in L_T^\infty E^4$, we find that $ \partial_x u \in L_T^1 L_{xy}^\infty$.

$\hfill \Box$	
	
	With local-in-time solutions $u \in C_T E^\infty$ at hand for smooth initial data $\phi \in E^\infty$, we state a more precise version of Theorem \ref{thm:LWPAnisotropic}, which will be proved subsequently. Let $s(\alpha) = 3 - \frac{\alpha}{2}$.
	\begin{theorem}
		\label{thm:QuasilinearLWP}
		Let $\alpha \in [2,4)$, and $s > s(\alpha)$. For $\alpha = 2$, suppose $\D = \mathbb{K}_1 \times \mathbb{K}_2 \times \mathbb{K}_3$, $\mathbb{K}_i \in \{ \R; \T \}$ and for $\alpha \in (2,4)$, suppose that $\D = \K \times \R^2$ with $\mathbb{K} \in \{  \R; \T \}$.
		
		Then there is a continuous $T=T(\| \phi \|_{E^s(\D)})$ and a continuous data-to-solution mapping $S_T^s : E^s(\D) \to C_T E^s(\D)$ for \eqref{eq:fKPISection}, which extends $S_T^\infty: E^\infty \to C_T E^\infty$.
	\end{theorem}
	The proof of local well-posedness proceeds in three steps:
	\begin{itemize}
		\item We show a priori estimates in $E^s$ for $s > 3 - \frac{\alpha}{2}$ up to times $T=T(\| \phi \|_{E^s})$.
		\item We show Lipschitz continuous dependence in $E^0$ for initial data in $E^s$ with $s > 3 - \frac{\alpha}{2}$ and $T=T( \| u_1(0) \|_{E^s}, \|u_2(0) \|_{E^s})$.
		\item By using frequency envelopes (cf. \cite{Tao2001,IfrimTataru2020}) adjusted to the $E^s$-norms, we conclude the proof of continuous dependence.
	\end{itemize}
	In the following, we work with smooth initial data $\phi \in E^\infty$. It suffices to prove the claims for data in $E^\infty$ because once the continuous dependence on the initial data is established, we obtain extensions $S_T^s$ by density.
	
	\medskip
	
	\emph{A priori estimates:}	Let $\alpha \in [2,4)$ and $\bar{s} > s(\alpha)$. The following is a reprise of the analysis to show existence of distributional solutions; we shall be brief. First we suppose that $\| \phi \|_{E^{\bar{s}}} \leq \varepsilon_0 \ll 1$. Let $u = S_T^{\infty}(\phi)$. Let $T_{\max}(\| \phi \|_{E^4})$ be the time of existence according to Proposition \ref{prop:LocalExistence}. From Lemma \ref{lemma:LinShortTime}, Proposition \ref{prop:ShortTimeBilinear}, and Proposition \ref{prop:EnergyEstimates}, we have the following set of estimates for $T= \min(T_{\max},1)$:
	\begin{equation}\label{est:firstapriori}
		\left. \begin{array}{cl}
			\| u \|_{F^{\bar{s}}(T)} &\lesssim \| u \|_{B^{\bar{s}}(T)} + \| \partial_x (u^2) \|_{\mathcal{N}^{\bar{s}}(T)}, \\
			\| \partial_x(u^2) \|_{\mathcal{N}^{\bar{s}}(T)} &\lesssim \| u \|^2_{F^{\bar{s}}(T)}, \\
			\| u \|^2_{B^{\bar{s}}(T)} &\lesssim \| u_0 \|^2_{E^{\bar{s}}} +  \| u \|^3_{F^{\bar{s}}(T)}.
		\end{array} \right\}
	\end{equation}
	This gives
	\begin{equation}
		\label{eq:APrioriSII}
		\|u\|_{F^{\bar{s}}(T)}^2 \lesssim \|u_0\|_{E^{\bar{s}}}^2 + \|u\|_{F^{\bar{s}}(T)}^4 + \|u\|_{F^{\bar{s}}(T)}^3.
	\end{equation}
	Hence, by choosing $\varepsilon_0$ small enough, we find by \eqref{eq:APrioriSII}
	that
	\begin{equation}
		\label{eq:FSAprioriII}
		\| u \|_{F^{\bar{s}}(T)} \lesssim \| u_0 \|_{E^{\bar{s}}}.
	\end{equation}
	
	Another application of Lemma \ref{lemma:LinShortTime}, Proposition \ref{prop:ShortTimeBilinear}, and Proposition \ref{prop:EnergyEstimates} yields
	\begin{equation}\label{est:secondAPriori}
		\left. \begin{array}{cl}
			\| u \|_{F^4(T)} &\lesssim \| u \|_{B^4(T)} + \| \partial_x (u^2) \|_{\mathcal{N}^4(T)}, \\
			\| \partial_x(u^2) \|_{\mathcal{N}^{4}(T)} &\lesssim \| u \|_{F^{4}(T)} \| u \|_{F^{\bar{s}}(T)}, \\
			\| u \|^2_{B^{4}(T)} &\lesssim \| u_0 \|^2_{E^{4}} +  \| u \|_{F^{\bar{s}}(T)} \| u \|^2_{F^{4}(T)}.
		\end{array} \right\}
	\end{equation}
	This set of estimates gives
	\begin{equation}
		\label{eq:APrioriF3II}
		\| u \|^2_{F^4(T)} \lesssim \| u_0 \|^2_{E^4} + \| u \|^2_{F^4(T)} \| u \|^2_{F^{\bar{s}}(T)} + \| u \|^2_{F^4(T)} \| u \|_{F^{\bar{s}}(T)},
	\end{equation}
	and therefore, for $\| u \|_{F^{\bar{s}}(T)} \lesssim \| u_0 \|_{E^{\bar{s}}} \lesssim \varepsilon_0 \ll 1$ and $T=T_{\max}(\| u_0 \|_{E^{4}})$, we find
	\begin{equation}
		\label{eq:APriori3II}
		\| u \|_{F^4(T)} \lesssim \| u_0 \|_{E^4}.
	\end{equation}
	Combining \eqref{eq:FSAprioriII} and \eqref{eq:APriori3II}, we infer existence and a priori estimates in $E^{\bar{s}}$ for $T=1$ after choosing $\| u_0 \|_{E^{\bar{s}}} \leq \varepsilon_0 \ll 1$.
	Rescaling allows us to extend the argument to large initial data, like in the proof of local existence. The details are omitted to avoid repetition. 

	\medskip
	
	\emph{Lipschitz continuous dependence in $F^0$:} Let $u_1$ and $u_2$ be solutions to \eqref{eq:FKPI} with initial data $u_1(0)$ and $u_2(0)$, respectively. We let $v:= u_1 - u_2$, which solves the equation
	\begin{equation*}
		\partial_t v - \partial_x D_x^\alpha v - \partial_x^{-1} \Delta_y v = \partial_x (v(u_1+u_2)).
	\end{equation*}
	From Lemma \ref{lemma:LinShortTime}, Proposition \ref{prop:ShortTimeBilinear}(b), Proposition \ref{prop:EnergyEstDiff}, we have for $\bar{s} > s(\alpha)$:
	\begin{equation*}
		\left\{ 	\begin{array}{cl}
			\|v\|_{F^{0}(T)} &\lesssim \|v\|_{B^{0}(T)} + \| \partial_x(v(u_1+u_2))\|_{\mathcal{N}^{0}(T)}	, \\
			\|\partial_x(v(u_1+u_2))\|_{\mathcal{N}^{0}(T)} &\lesssim \| v \|_{F^{0}(T)} (\|u_1\|_{F^{\bar{s}}(T)} + \|u_2\|_{F^{\bar{s}}(T)}), \\
			\|v\|_{B^{0}(T)}^2 &\lesssim \|v(0)\|_{E^0}^2 + \|v\|_{F^{0}(T)}^2 (\|u_1\|_{F^{\bar{s}}(T)} + \|u_2\|_{F^{\bar{s}}(T)} ).
		\end{array}\right.
	\end{equation*}
	Taking the estimates together, we find
	\begin{equation*}
		\| v \|^2_{F^0(T)} \lesssim \| v(0) \|^2_{E^0} + \| v \|^2_{F^0(T)} ( \| u_1 \|^2_{F^{\bar{s}}(T)} + \| u_2 \|^2_{F^{\bar{s}}(T)}) + \| v \|^2_{F^0(T)} ( \| u_1 \|_{F^{\bar{s}}(T)} + \| u_2 \|_{F^{\bar{s}}(T)}).
	\end{equation*}
	For $T=T(\|u_1(0) \|_{E^{\bar{s}}}, \|u_2(0) \|_{E^{\bar{s}}})$, we obtain by the previously established a priori estimates
	\begin{equation*}
		\| v \|_{F^0(T)} \lesssim \| v(0) \|_{E^0}.
	\end{equation*}
	
	\medskip
	
	\emph{Continuity of the data-to-solution map:}
	In this paragraph we extend the data-to-solution mapping from smooth data to $E^s$ and make use of frequency envelopes. We follow the exposition of Ifrim--Tataru \cite{IfrimTataru2020}, which we adjust to the present setting
	of the $E^s$-scale of regularity. Let $u_0 \in E^{\bar{s}}$ with size $M$. We define frequency envelopes in the $E^s$-scale:
	\begin{definition}
		We say that $(c_N)_{N \in 2^{\N}} \in \ell^2$ is a frequency envelope for a function $u$ in $E^s$ if we have the following properties:
		\begin{itemize}
			\item[a)] Energy bound:
			\begin{equation*}
				\| P_N u \|_{E^s} \leq c_N,
			\end{equation*}
			\item[b)] Slowly varying:
			\begin{equation*}
				\frac{c_N}{c_J} \lesssim \big[ \frac{N}{J} \big]^\delta.
			\end{equation*}
			Above and in the following let $\big[ \frac{x}{y} \big] = \max( \frac{x}{y}, \frac{y}{x})$ for $x,y > 0$.
		\end{itemize}
	\end{definition}
	An envelope which satisfies
	\begin{equation*}
		\| u \|^2_{E^s} \approx \sum_N c_N^2
	\end{equation*}
	is called sharp. Frequency envelopes for the datum $u_0 \in E^s$ are constructed by mollifying the initial guess $\tilde{c}_J = \| P_J u_0 \|_{E^s}$ to
	\begin{equation*}
		c_N = \sup_J \big( \big[ \frac{N}{J} \big]^{-\delta} \tilde{c}_J \big).
	\end{equation*}
	
	We turn to \emph{regularization} in the setting of frequency envelopes: Let $u_0 \in E^s$ with $\| u_0 \|_{E^s} = C$ and let $(c_N)_{N \in 2^{\N_0}}$ denote a sharp frequency envelope for $u_0$ in $E^s$.
	For $u_0$ we consider a family of regularizations $u_0^M \in E^{\infty}$ defined by truncating the $x$-frequencies at $M$, i.e.\ $u_0^M = P_{\leq M} u_0$. These regularizations satisfy the following:
	\begin{itemize}
		\item[i)] Uniform bounds: $\| P_K u_0^M \|_{E^s} \lesssim c_K$,
		\item[ii)] High frequency bounds: $\| u_0^M \|_{E^{s+j}} \lesssim M^j c_M$,
		\item[iii)] Difference bounds: $\| u_0^{2M} - u_0^M \|_{E^0} \lesssim M^{-s} c_M$,
		\item[iv)] Limit as $M \to \infty$: $u_0 = \lim_{M \to \infty} u_0^M$ in $E^s$.
	\end{itemize}
	Associated with $u_0^M$ we obtain a solution $u^M$ in $E^\infty$ which exists up to time $T=T(C)$, uniformly in $M$. We have the following uniform bounds:
	\begin{itemize}
		\item [i)] High frequency bounds:
		\begin{equation}
			\label{eq:HighFrequencyBoundsRegularization}
			\| u^M \|_{C([0,T],E^{s+j})} \lesssim M^j c_M, \quad j \geq 0,
		\end{equation}
		\item [ii)] Difference bounds:
		\begin{equation}
			\label{eq:DifferenceBoundsRegularization}
			\| u^{2M} - u^M \|_{C([0,T],E^0)} \lesssim M^{-s} c_M.
		\end{equation}
	\end{itemize}
	Interpolation gives
	\begin{equation*}
		\| u^{2M} - u^M \|_{C_T E^m} \lesssim M^{-(s-m)} c_M, \quad m \geq 0.
	\end{equation*}
	Therefore, we obtain the uniform frequency envelope bounds
	\begin{equation*}
		\| P_K u^M \|_{C_T E^s} \lesssim_N c_K \max \{ \frac{K}{M}; 1 \}^{-N}
	\end{equation*}
	for any $N \in \N_0$.
	
	We analyze the convergence of $u^M$ as $M \to \infty$. 	By writing the difference as a telescopic sum
	\begin{equation*}
		u - u^M = \sum_{L \geq M} u^{2L} - u^L,
	\end{equation*}
	the difference bound \eqref{eq:DifferenceBoundsRegularization} implies convergence of $u^M$ in $E^0$ to a limit $u \in C_T E^0$ with
	\begin{equation*}
		\| u - u^M \|_{C_T E^0} \lesssim M^{-s}.
	\end{equation*}
	Now, we prove convergence in $ C_T E^s$:
	For $L' \ll L$ we have from \eqref{eq:DifferenceBoundsRegularization}:
	\begin{equation*}
		\| P_{L'} (u^{2L} - u^L) \|_{C_T E^s} \lesssim (L')^s \| u^{2L} - u^L \|_{C_T E^0} \lesssim \big( \frac{L'}{L} \big)^s c_L.
	\end{equation*}
	For $L' \gg L$, we can use \eqref{eq:HighFrequencyBoundsRegularization} (and the slowly varying property) to argue
	\begin{equation*}
		\| P_{L'} (u^{2L} - u^L) \|_{C_T E^s} \lesssim (L')^{-s} ( \| P_{L'} u^{2L} \|_{C_T E^{2s}} + \| P_{L'} u^L \|_{C_T E^s}) \lesssim \big( \frac{L}{L'} \big)^s c_L.
	\end{equation*}
	Consequently, an application of Young's inequality combined with the slowly varying property gives
	\begin{equation*}
		\| u - u^M \|_{C_T E^s} \lesssim  \big( \sum_{L \geq M} c_L^2 \big)^{\frac{1}{2}}\to 0 \quad (M \to \infty).
	\end{equation*}
	
	Now we turn to the proof of continuous dependence. We shall see that the previously constructed data-to-solution mapping is also continuous. Let $u_{0n} \to u_0$ in $E^s$ for $s > 3 - \frac{\alpha}{2}$ and the
	corresponding solutions $u_n$, $u$, which exist with a uniform lifespan $T=T(\| u_0 \|_{E^s})$. We have to show that $u_n \to u$ in $C([0,T],E^s)$. We have seen
	that $u_n \to u$ in $C_T L^2$. Moreover, uniform boundedness in $C([0,T],E^s)$, which was proved in the previous paragraph, implies convergence in $C_T E^\sigma$ for every $0 \leq \sigma < s$.
	For $\sigma = s$, we shall again use frequency envelopes. To carry out the argument, we consider the approximate solutions $u_n^M$ and $u^M$. We use that the initial data converge in all $E^\sigma$-norms:
	\begin{equation*}
		u_{0n}^M \to u_0^M \text{ in } E^{\sigma} \text{ for } \sigma \geq 0.
	\end{equation*}
	By the above, we have convergence of the regular solutions in all $E^\sigma$-norms:
	\begin{equation*}
		u_n^M \to u^M \text{ in } C_T E^\sigma, \quad \sigma \geq 0.
	\end{equation*}
	We use the triangle inequality to compare $u$ and $u_n$:
	\begin{equation*}
		\| u_n - u \|_{C_T E^s} \lesssim \| u_n^M - u^M \|_{C_T E^s} + \| u^M - u \|_{C_T E^s} + \| u_n^M - u_n \|_{C_T E^s}.
	\end{equation*}
	The first term goes to zero as $n \to \infty$ for fixed $M$; the second term goes to zero as $M \to \infty$. It remains to obtain uniform smallness of the third term for large $n$.
	For this purpose we consider frequency envelopes $c_J^{(n)}$ for the initial data $u_{n0}$. By construction, we can argue that
	\begin{equation*}
		\big( c_J^{(n)} \big)_{J \in 2^{\N_0}} \to (c_J)_{J \in 2^{\N_0}} \text{ in } \ell^2
	\end{equation*}
	with $(c_J)$ denoting a sharp frequency envelope for $u_0$. Hence, in terms of frequency envelopes, we find
	\begin{equation*}
		\| u_n - u \|_{C_T E^s} \lesssim \| u_n^M - u^M \|_{C_T E^s} + \big( \sum_{L \geq M} c_L^2 \big)^{\frac{1}{2}}+\big( \sum_{L \geq M} (c^{(n)}_L)^2 \big)^{\frac{1}{2}} ,
	\end{equation*}
	implying
	\begin{equation*}
		\limsup_{n \to \infty} \| u_n - u \|_{C_T E^s} \lesssim \big( \sum_{L \geq M} c_L^2 \big)^{\frac{1}{2}} \to 0 \quad (M \to \infty)
	\end{equation*}
	and the proof of continuous dependence is complete.
	$\hfill \Box$
	
	\section{Local well-posedness for fifth order KP-I equation}\label{section:lwp5}
	In the following we prove semilinear local well-posedness
	\begin{equation}
		\label{eq:FifthKPI}
		\left. \begin{split}
			\partial_t u - \partial_x^5 u - \partial_{x}^{-1} \Delta_y u &= \partial_x (u^2), \quad ((x,y),t) \in  \mathbb{D} \times \R, \\
			u(0) &= u_0 \in H^{s_1,s_2}(\D)
		\end{split} \right\}
	\end{equation}
	for $\D = \K \times \mathbb{R}^2$, $\mathbb{K} \in \{ \R; \T \}$ in anisotropic Sobolev spaces
	\begin{equation*}
		H^{s_1,s_2}(\D) = \{ \phi \in L^2 (\D) : \| \phi \|^2_{H^{s_1,s_2}} = \int |\hat{\phi}(\xi,\eta)|^2 \langle \xi \rangle^{2s_1} \langle \eta \rangle^{2s_2} d\xi d\eta < \infty \}
	\end{equation*}
	with $s_1,s_2 > \frac{1}{2}$ as stated in Theorem \ref{thm:SemilinearLWP}. We show local well-posedness via the contraction mapping principle in adapted $U^p$-/$V^p$-function spaces. Let $u \in X^{s_1,s_2}_T$ (to be defined below) be a solution to
	\begin{equation*}
		\left. \begin{split}
			\partial_t u - \partial_x^5 u - \partial_x^{-1} \Delta_y u &= f, \\
			u(0) &= u_0 \in H^{s_1,s_2}(\D).
		\end{split} \right\}
	\end{equation*}
	We show the estimates
	\begin{equation*}
		\left. \begin{split}
			& \| u \|_{X_T^{s_1,s_2}} \lesssim \| u_0 \|_{H^{s_1,s_2}} + \| 1_{[0,T]}(t) \int_0^t e^{(t-s)(\partial_x^5 + \partial_{x}^{-1} \Delta_y)} f(s) ds \|_{X_T^{s_1,s_2}}, \\
			&\| 1_{[0,T]}(t) \int_0^t e^{(t-s)(\partial_x^5 + \partial_{x}^{-1} \Delta_y)} \partial_x (u_1 u_2)(s) ds \|_{X_T^{s_1,s_2}} \lesssim T^{\alpha} \| u_1 \|_{X_T^{s_1,s_2}} \| u_2 \|_{X_T^{s_1,s_2}},
		\end{split} \right\}
	\end{equation*}
	for some $\alpha >0$, from which the result follows by standard arguments. The linear estimate is immediate from the definition of the function spaces, and its proof will be omitted. We shall focus on the bilinear estimate.
	\subsection{Function spaces}
	We shall be brief and refer to \cite{HadacHerrKoch2009,HadacHerrKoch2009II} for details.
	
	Let $\mathcal{Z}$ be the set of finite partitions $-\infty = t_0 < t_1 \leq \ldots < t_K = \infty$, and let $\mathcal{Z}_0$ be the set of finite partitions $- \infty < t_0 < t_1 < \ldots < t_K < \infty $.
	\begin{definition}
		Let $1 \leq p < \infty$. For $\{ t_k \}_{k=0}^K \in \mathcal{Z}$ and $\{ \phi_k \}_{k=0}^{K-1} \subseteq L^2(\D)$ with
		\begin{equation*}
			\sum_{k=0}^{K-1} \| \phi_k \|_{L^2}^p = 1 \text{ and } \phi_0 =0.
		\end{equation*}
		We call the function $a: \R \to L^2$ given by
		\begin{equation*}
			a = \sum_{k=1}^K 1_{[t_{k-1},t_k)} \phi_k
		\end{equation*}
		a $U^p$-atom. We define the atomic space
		\begin{equation*}
			U^p = \{ u = \sum_{j=1}^\infty \lambda_j a_j : \, a_j \text{ is a } U^p \text{-atom}, \, \text{ and } \lambda_j \in \C \text{ such that } \sum_{j=1}^\infty |\lambda_j| < \infty \}
		\end{equation*}
		with norm
		\begin{equation*}
			\| u \|_{U^p}  = \inf \{ \sum_{j=1}^\infty |\lambda_j| \, : \, u = \sum_{j=1}^\infty \lambda_j a_j, \, \lambda_j \in \C, \, a_j: U^p\text{-atom} \}.
		\end{equation*}
	\end{definition}
	
	The slightly larger spaces of bounded $p$-variation are defined as follows:
	\begin{definition}
		Let $1 \leq p < \infty$. We define $V^p$ as function spaces $v: \R \to L^2(\D)$ such that $v(\infty) := \lim_{t \to \infty} v(t) = 0$ and $v(-\infty)$ exists, for which the following norm is finite:
		\begin{equation*}
			\| v \|_{V^p} = \sup_{ \{t_k \}_{k=0}^K \in \mathcal{Z}} \big( \sum_{k=1}^K \| v(t_k) - v(t_{k-1}) \|_{L^2}^p \big)^{\frac{1}{p}}.
		\end{equation*}
	\end{definition}
	Let $(S(t) u_0) \widehat (\xi,\eta) = e^{it (\xi^5 + \eta^2/\xi)} \hat{u}_0(\xi,\eta)$ denote the linear propagator in $L^2$.
	\begin{definition}
		We define
		\begin{itemize}
			\item $U^p_S = S(\cdot) U^p$ with norm $\| u \|_{U^p_S} = \| S(-\cdot) u \|_{U^p}$,
			\item $V^p_S = S(\cdot) V^p$ with norm $\| u \|_{V^p_S} = \| S(-\cdot) u \|_{V^p}$.
		\end{itemize}
	\end{definition}
	
	The function space in which we apply the contraction mapping principle is defined as
	\begin{equation*}
		X_T^{s_1,s_2} = \{ f \in C([0,T],L^2) : \| f \|^2_{X_T^{s_1,s_2}} = \sum_{N,M} N^{2s_1} M^{2s_2} \| P_{N,M} f \|_{U^2_S}^2 < \infty \}.
	\end{equation*}
	
	We define the smooth (modulation) projection for $M \in 2^{\N}$:
	\begin{equation*}
		\widehat{Q_M^S u}(\xi,\eta,\tau) = \phi_M(\tau - \xi^5 - \eta^2/\xi) \hat{u}(\xi,\eta,\tau)
	\end{equation*}
	with $\phi \in C^\infty_c$ like in Section \ref{subsection:FunctionSpaces}.
	
	Recall the following bounds for $Q_M^S$ on $U^2_S$ and $V^2_S$:
	\begin{lemma}
		We have
		\begin{equation*}
			\begin{split}
				\| Q_M^S u \|_{L^2} &\lesssim M^{-\frac{1}{2}} \| u \|_{U^2_S}, \\
				\| Q_{\geq M}^S u \|_{L^2} &\lesssim M^{-\frac{1}{2}} \| u \|_{V^2_S}
			\end{split}
		\end{equation*}
		and uniform boundedness of $Q^S_{<M}$ and $Q^S_{\geq M}$ in $M$ on $U^p_S$ and $V^p_S$.
	\end{lemma}
	
	\subsection{Proof of bilinear estimate}
	
	This section is devoted to the proof of
	\begin{equation*}
		\big\| 1_{[0,T]}(t) \int_0^t S(t-s) \partial_x (u_1(s) u_2(s)) ds \big\|_{X_T^{s_1,s_2}} \lesssim T^{\alpha} \| u_1 \|_{X_T^{s_1,s_2}} \| u_2 \|_{X_T^{r_1,r_2}}
	\end{equation*}
	with $s_i \geq r_i > \frac{1}{2}$.
	
	We carry out an inhomogeneous Littlewood-Paley decomposition as well in $\xi$- as $\eta$-frequencies
	\begin{equation*}
		u_i = \sum_{N,M} P_{N,M} u_i.
	\end{equation*}
	By duality $(U^2)^{\ast} = V^{2}$ (cf. \cite[Proposition~2.8]{HadacHerrKoch2009}), it suffices to prove the following for some $\alpha >0$:
	\begin{equation}
		\label{eq:DyadicEstimate}
		\sup_{\| v \|_{V^2_S} = 1} \big| \iint P_{N,M} v \partial_x (P_{N_1,M_1} u_1 P_{N_2,M_2} u_2 ) dx dy dt \big| \lesssim T^\alpha C(\underline{N},\underline{M}) \| P_{N_1,M_1} u_1 \|_{U^2_S} \| P_{N_2,M_2} u_2 \|_{U^2_S}.
	\end{equation}
	The claim then follows from square summation for acceptable bounds of $C(\underline{N},\underline{M})$.
	
	\begin{lemma}[HighxLow-interaction]
		Let $N_1 \gg 1$, and $N_1 \gg N_2$. Then, estimate \eqref{eq:DyadicEstimate}  holds true with $C(\underline{N},\underline{M}) = N_2^{\frac{1}{2}} M_{\min}^{\frac{1}{2}}$.
	\end{lemma}
	\begin{proof}
		If $N_2 = 1$, we make an additional homogeneous frequency decomposition of the low frequencies, so that $N_2 \in 2^{\Z}$.
		We estimate the frequencies with $N_2 \leq N_1^{-4}$ by the $L^4$-Strichartz estimates:
		\begin{equation*}
			\begin{split}
				\big| \iint v_{N,M} \partial_x (P_{N_1,M_1} u_1 P_{N_2,M_2} u_2) \big| &\lesssim N T^{\frac{1}{2}} \| P_{N_1,M_1} u_1 \|_{L^4} \| P_{N_2,M_2} u_2 \|_{L^4} \\
				&\lesssim T^{\frac{1}{2}} N N_1^{\frac{1}{4}} N_2^{\frac{1}{2}} \| P_{N_1,M_1} u_1 \|_{U^2_S} \| P_{N_2,M_2} u_2 \|_{U^2_S}.
			\end{split}
		\end{equation*}
		We suppose in the following that $N_2 \geq N_1^{-4}$. First, we estimate the resonant contribution when all modulations are smaller than $N_1^4 N_2$, i.e.,
		\begin{equation}
			\label{eq:ResonantSemilinear}
			\big| \iint Q_{\ll N_1^4 N_2} P_{N,M} v \partial_x (P_{N_1,M_1} Q_{\ll N_1^4 N_2} u_1 P_{N_2,M_2} Q_{\ll N_1^4 N_2} u_2) \big|.
		\end{equation}
		This is amenable to the bilinear Strichartz estimate in Proposition \ref{prop:GeneralBilinear}, which gives
		\begin{equation*}
			\begin{split}
				\eqref{eq:ResonantSemilinear} &\lesssim \| Q_{\ll N_1^4 N_2} \partial_x (P_{N_1,M_1} Q_{\ll N_1^4 N_2} u_1 P_{N_2,M_2} Q_{\ll N_1^4 N_2} u_2) \|_{L_t^1 L_x^2} \\
				&\lesssim T^{\frac{1}{2}} N_2^{\frac{1}{2}} M_{\min}^{\frac{1}{2}} \| P_{N_1,M_1} u_1 \|_{U^2_S} \| P_{N_2,M_2} u_2 \|_{U^2_S}.
			\end{split}
		\end{equation*}
		In case there is one function in the trilinear expression carrying high modulation we shall rely on the $L^4$-Strichartz estimate. Suppose that $v$ is at high modulation. We use by Lemma \ref{lem:L4Summary}
		\begin{equation}
			\label{eq:BilinearAuxI}
			\begin{split}
				&\quad \big| \iint Q_{\gtrsim N_1^4 N_2} P_{N,M} v \partial_x (P_{N_1,M_1} u_1 P_{N_2,M_2} u_2 ) \big| \\
				&\lesssim (N_1^4 N_2)^{-\frac{1}{2}} N_1 \| P_{N_1,M_1} u_1 \|_{L^4_{x,y,t}} \| P_{N_2,M_2} u_2 \|_{L^4_{x,y,t}} \\
				&\lesssim (N_1^4 N_2)^{-\frac{1}{2}} N_1 N_2^{\frac{1}{2}} N_1^{\frac{1}{2}} \| u_1 \|_{U^2_S} \| u_2 \|_{U^2_S}.
			\end{split}
		\end{equation}
		Note that depending on $N_2 \lesssim 1$, the estimate further improves, but we do not need this. To find an additional factor $T^\alpha$, we can use two $L^4$-Strichartz estimates, but estimate $Q_{\gtrsim N_1^4 N_2} P_{N,M} v$ in $L_t^\infty L_{x}^2$ by H\"older's inequality to find:
		\begin{equation}
			\label{eq:BilinearAuxII}
			\begin{split}
				&\quad \big| \iint Q_{\gtrsim N_1^4 N_2} P_{N,M} v \partial_x (P_{N_1,M_1} u_1 P_{N_2,M_2} u_2 ) \big| \\
				&\lesssim T^{\frac{1}{2}} N_1 \| P_{N_1,M_1} u_1 \|_{L^4_{x,y,t}} \| P_{N_2,M_2} u_2 \|_{L^4_{x,y,t}} \\
				&\lesssim T^{\frac{1}{2}} N_1 N_2^{\frac{1}{2}} N_1^{\frac{1}{2}} \| u_1 \|_{U^2_S} \| u_2 \|_{U^2_S}.
			\end{split}
		\end{equation}
		Interpolation of \eqref{eq:BilinearAuxI} and \eqref{eq:BilinearAuxII} yields a favorable power of $N_1^{-1}$ and a factor of $T^{\alpha}$.
		
		If $u_2$ is at high modulation, then two $L^4_{x,y,t}$-Strichartz estimates applied to $v$ and $u_1$ yield together with
		\begin{equation*}
			\begin{split}
				&\quad \big| \iint P_{N,M} v \partial_x (P_{N_1,M_1} u_1 Q_{\gtrsim N_1^4 N_2} P_{N_2,M_2} u_2 ) \big| \\
				&\lesssim N_1 \| P_{N,M} v \|_{L^4_{x,y,t}} \| P_{N_1,M_1} u_1 \|_{L^4_{x,y,t}} \| Q_{\gtrsim N_1^4 N_2} P_{N_2,M_2} u_2 \|_{L^2} \\
				&\lesssim N_1 N_2^{\frac{1}{2}} N_1^{\frac{1}{2}} ( N_1^4 N_2)^{-\frac{1}{2}} \| P_{N_1,M_1} u_1 \|_{U^2_S} \| P_{N_2,M_2} u_2 \|_{U^2_S}.
			\end{split}
		\end{equation*}
		The case with $u_1$ being estimated at high modulation and $v$ and $u_2$ via $L^4_{x,y,t}$-Strichartz estimates is better behaved because the $L^4$-estimates lose fewer derivatives. The same interpolation argument like in the proof of \eqref{eq:BilinearAuxII} allows us to gain a factor of $T^\alpha$ at dispensing a fraction of $(N_1^4 N_2)$, which is affordable. The proof is complete.
	\end{proof}
	
	\begin{lemma}[HighxHigh-High-interaction]
		Let $N \sim N_1 \sim N_2 \gg 1$. Then, we find \eqref{eq:DyadicEstimate} to hold with $C(\underline{N},\underline{M}) = N_{\min}^{\frac{1}{2}} M_{\min}^{\frac{1}{2}}$.
	\end{lemma}
	\begin{proof}
		In the resonant case, when all modulations are much smaller than $N_1^4 N_2$, we can use a bilinear Strichartz estimate to obtain
		\begin{equation*}
			\begin{split}
				&\quad \sup_{\| v \|_{V^2_S} = 1} \big| \iint Q_{\ll N_1^4 N_2} P_{N,M} v \partial_x (Q_{\ll N_1^4 N_2} P_{N_1,M_1} u_1 Q_{\ll N_1^4 N_2} P_{N_2,M_2} u_2 ) \big| \\
				&\lesssim T^{\frac{1}{2}} N \| Q_{\ll N_1^4 N_2} (Q_{\ll N_1^4 N_2} P_{N_1,M_1} u_1 Q_{\ll N_1^4 N_2} P_{N_2,M_2} u_2) \|_{L^2} \\
				&\lesssim T^{\frac{1}{2}} N \big( \frac{N_2}{N_1^2} \big)^{\frac{1}{2}} M_{\min}^{\frac{1}{2}} \| P_{N_1,M_1} u_1 \|_{U^2_S} \| P_{N_2,M_2} u_2 \|_{U^2_S}.
			\end{split}
		\end{equation*}
		In the non-resonant case, when there is a function with high modulation, we can use two $L^4$-Strichartz estimates to find
		\begin{equation*}
			\begin{split}
				\sup_{\| v \|_{V^2_S} = 1} \big| \iint Q_{\gtrsim N_1^4 N_2} P_{N,M} v \partial_x (P_{N_1,M_1} u_1 P_{N_2,M_2} u_2) \big| &\lesssim N \sup_{\| v \|_{V^2_S} = 1} \| Q_{\gtrsim N_1^4 N_2} P_{N,M} v \|_{L^2_x} \| P_{N_1,M_1} u_1 \|_{L^4} \|P_{N_2,M_2} u_2 \|_{L^4} \\
				&\lesssim N (N_1^4 N_2)^{-\frac{1}{2}} N_1 \| P_{N_1,M_1} u_1 \|_{V^2_S} \| P_{N_2,M_2} u_2 \|_{V^2_S}.
			\end{split}
		\end{equation*}
		Interpolation with
		\begin{equation*}
			\begin{split}
				\sup_{\| v \|_{V^2_S} = 1} \big| \iint Q_{\gtrsim N_1^4 N_2} P_{N,M} v \partial_x (P_{N_1,M_1} u_1 P_{N_2,M_2} u_2) \big| &\lesssim N \sup_{\| v \|_{V^2_S} = 1} \| Q_{\gtrsim N_1^4 N_2} P_{N,M} v \|_{L^2_x} \| P_{N_1,M_1} u_1 \|_{L^4} \|P_{N_2,M_2} u_2 \|_{L^4} \\
				&\lesssim T^{\frac{1}{2}} N N_1 \| P_{N_1,M_1} u_1 \|_{V^2_S} \| P_{N_2,M_2} u_2 \|_{V^2_S}
			\end{split}
		\end{equation*}
		yields the required factor of $T^\alpha$. 
		Since this estimates the functions in $V^2_S$, the argument also applies when $u_i$ is at high modulation.
	\end{proof}
	
	\begin{lemma}[HighxHigh-Low-interaction]
		Let $N_1 \sim N_2 \gg 1$, and $N \ll N_1$. Then, we find \eqref{eq:DyadicEstimate} to hold with
		$C(\underline{N},\underline{M}) = N_1^{\frac{1}{2}+} M_{\min}^{\frac{1}{2}}$.
	\end{lemma}
	\begin{proof}
		If $N = 1$, we carry out an additional homogeneous dyadic decomposition in $N \in 2^{\Z}$.
		Let $N \leq N_1^{-4}$. In this case we simply use two $L^4$-Strichartz estimates to find
		\begin{equation*}
			\begin{split}
				\big| \iint P_{N,M} v \partial_x( P_{N_1,M_1} u_1 P_{N_2,M_2} u_2 ) \big| &\lesssim N T^{\frac{1}{2}} \| P_{N_1,M_1} u_1 \|_{L^4} \| P_{N_2,M_2} u_2 \|_{L^4} \\
				&\lesssim N T^{\frac{1}{2}} (N_1 N)^{\frac{1}{2}} \| P_{N_1,M_1} u_1 \|_{U^2_S} \| P_{N_2,M_2} u_2 \|_{U^2_S}.
			\end{split}
		\end{equation*}
		In the following let $N \geq N_1^{-4}$.
		We begin with the resonant case when all functions have modulation $\ll N_1^4 N$. This allows us to use a bilinear Strichartz estimate
		\begin{equation*}
			\begin{split}
				&\quad \big| \iint Q_{\ll N_1^4 N} P_{N,M} v \partial_x (Q_{\ll N_1^4 N} P_{N_1,M_1} u_1 Q_{\ll N_1^4 N} P_{N_2,M_2} u_2 ) \big| \\
				&\lesssim N \| Q_{\ll N_1^4 N} (Q_{\ll N_1^4 N} P_{N_1,M_1} u_1 Q_{\ll N_1^4 N} P_{N,M} v) \|_{L_t^1 L_x^2} \\
				&\lesssim N T^{\frac{1}{2}} \big( \frac{N}{N_1^2} \big)^{\frac{1}{2}} M_{\min}^{\frac{1}{2}} \| P_{N_1,M_1} u_1 \|_{U^2_S} \| P_{N_2,M_2} v \|_{U^2_S} \| P_{N_2,M_2} u_2 \|_{U^2_S}.
			\end{split}
		\end{equation*}
		But we have to estimate $v$ in $V^2_S$. For this purpose we interpolate (cf. \cite[Proposition~2.20]{HadacHerrKoch2009}) with the Strichartz estimate
		\begin{equation*}
			\| P_{N_1,M_1} u_1 P_{N,M} v \|_{L^2} \lesssim N^{\frac{3}{4}} N_1^{\frac{1}{4}} \| P_{N_1,M_1} u_1 \|_{U^4_S} \| P_{N_2,M_2} v \|_{U^4_S}.
		\end{equation*}
		Hence, we obtain
		\begin{equation*}
			\lesssim T^{\frac{1}{2}} \big( \frac{N}{N_1^2} \big)^{\frac{1}{2}} N \log(N^{\frac{1}{4}} N_1^{\frac{5}{4}}) \| P_{N_1,M_1} u_1 \|_{V^2_S} \| P_{N_2,M_2} u_2 \|_{V^2_S} \| P_{N,M} v \|_{V^2_S}.
		\end{equation*}
		In the non-resonant case we have one modulation comparable to $N_1^4 N$. We obtain
		\begin{equation*}
			\begin{split}
				\big| \iint Q_{\gtrsim N_1^4 N} P_{N,M} v \partial_x (P_{N_1,M_1} u_1 P_{N_2,M_2} u_2) \big| &\lesssim \| Q_{\gtrsim N_1^4 N} \partial_x P_{N,M} v \|_{L^2} \| P_{N_1,M_1} u_1 \|_{L^4} \| P_{N_2,M_2} u_2 \|_{L^4} \\
				&\lesssim N (N_1^4 N)^{-\frac{1}{2}} (N_1 N_2)^{\frac{1}{2}} \| P_{N_1,M_1} u_1 \|_{V^2_S} \| P_{N_2,M_2} u_2 \|_{V^2_S}.
			\end{split}
		\end{equation*}
		An interpolation argument like above yields an additional factor of $T^\alpha$.
		The estimates with $u_i$ at high modulation are better behaved because of improved $L^4$-Strichartz estimates at low frequencies (applied to $P_{N,M} v$).
	\end{proof}
	
	\begin{lemma}[LowxLow-interaction]
		Let $N,N_1,N_2 \lesssim 1$. Then, we find \eqref{eq:DyadicEstimate} to hold with $C(\underline{N},\underline{M}) =1 $.
	\end{lemma}
	\begin{proof}
		By two $L^4$-Strichartz estimates we find
		\begin{equation*}
			\sup_{\| v \|_{V^2_S} =1} \big| \iint P_{N,M} v \partial_x (P_{N_1,M_1} u_1 P_{N_2,M_2} u_2) \big| \lesssim T^{\frac{1}{2}} \| P_{N_1,M_1} u_1 \|_{L^4} \| P_{N_2,M_2} u_2 \|_{L^4} \lesssim T^{\frac{1}{2}} \prod_{i=1}^2 \| u_i \|_{U^2_S},
		\end{equation*}
		which is good enough.
	\end{proof}

	\appendix
	
	\section{Anisotropic Leibniz rule on mixed domains}
\label{sec:anisotropicLeibniz}	
		This section is devoted to the proof of Proposition \ref{prop:AnisotropicLeibnizRule}. We use the following fractional Leibniz rule, which is based on the well-known Kato-Ponce estimate (cf. \cite{Grafakos2012,GrafakosTorres2002}) for $\alpha \geq 0$, $\frac{1}{2} = \frac{1}{p_1} + \frac{1}{q_1} = \frac{1}{p_2} + \frac{1}{q_2}$ for $1 < p_1,p_2,q_1,q_2 \leq \infty$:
	\begin{equation}
		\label{eq:KatoPonceEstimate}
		\| \langle \partial \rangle^{\alpha} (fg) \|_{L^2(\R^d)} \lesssim \| \langle \partial \rangle^\alpha f \|_{L^{p_1}(\R^d)} \| g \|_{L^{q_1}(\R^d)} + \| f \|_{L^{p_2}(\R^d)} \| \langle \partial \rangle^\alpha g \|_{L^{q_2}(\R^d)}.
	\end{equation}

						We also need the above on $\R^{d_1} \times \T^{d_2}$. It turns out we can transfer the above estimate for $1<p_1,q_2 < \infty$ in a straightforward manner by an extension operator. Let $\varphi \in C^\infty_c(\R^d;\R_{\geq 0})$ denote a radially decreasing function with $\varphi(x) = 1$ for $|x| \leq 10$ and $\varphi(x) = 0$ for $|x| \geq 15$. We denote the extension $\tilde{f} = f \varphi \in L^2(\R^d)$ of $f \in L^2(\T^d)$. We have for $k \in \N$ clearly
	\begin{equation*}
		\| \langle \partial \rangle^k f \|_{L^2(\T^d)} \lesssim \| \langle \partial \rangle^k \tilde{f} \|_{L^2(\R^d)} \lesssim \| \langle \partial \rangle^k f \|_{L^2(\T^d)}.
	\end{equation*}
	By interpolation, the equivalence also holds for $\langle \partial \rangle^s$, and by obvious modification of the extension operator also holds on $\R^{d_1} \times \T^{d_2}$. So, let $f,g:\T^d \to \C$, and let $\tilde{f}$, $\tilde{g}:\R^d \to \C$ denote their extensions. Then we obtain
	\begin{equation*}
		\| \langle \partial \rangle^\alpha (fg) \|_{L^2(\T^d)} \lesssim \| \langle \partial \rangle^\alpha (\tilde{f} \tilde{g}) \|_{L^2(\R^d)} \lesssim \| \langle \partial \rangle^\alpha \tilde{f} \|_{L^{p_1}(\R^d)} \| \tilde{g} \|_{L^{p_2}(\R^d)} + \| \tilde{f} \|_{L^{q_1}(\R^d)} \| \langle \partial \rangle^\alpha \tilde{g} \|_{L^{q_2}(\R^d)}.
	\end{equation*}
	Clearly,
	\begin{equation*}
		\| \tilde{g} \|_{L^{p_2}(\R^d)} \lesssim \| g \|_{L^{p_2}(\T^d)}, \quad \| \tilde{f} \|_{L^{q_1}(\R^d)} \lesssim \| f \|_{L^{q_1}(\T^d)}.
	\end{equation*}
	To argue that
	\begin{equation}
		\label{eq:ContinuityExtensionDerivatives}
		\| \langle \partial \rangle^\alpha \tilde{f} \|_{L^p(\R^d)} \lesssim \| \langle \partial \rangle^\alpha f \|_{L^p(\T^d)}
	\end{equation}
	for $1<p<\infty$, we again use interpolation such that it suffices to show \eqref{eq:ContinuityExtensionDerivatives} for $\alpha = k \in \N$. We use the characterization of $L^p$-based Sobolev spaces
	\begin{equation*}
		\| \langle \partial \rangle^k \tilde{f} \|_{L^p(\R^d)} \simeq \| \tilde{f} \|_{L^p(\R^d)} + \sum_{|\alpha| = k} \| \partial^\alpha \tilde{f} \|_{L^p(\R^d)}
	\end{equation*}
	from which \eqref{eq:ContinuityExtensionDerivatives} is immediate by the product rule. We have proved the following:
	\begin{proposition}[Fractional Leibniz rule on cylinders]
		Let $d_1,d_2 \in \N_0$ with $d_1 + d_2 \geq 1$, $\alpha \geq 0$, $\frac{1}{2} = \frac{1}{p_1} + \frac{1}{q_1} = \frac{1}{p_2} + \frac{1}{q_2}$ for $1 < p_1,p_2,q_1,q_2 \leq \infty$ and $1<p_1,q_2 < \infty$:
		\begin{equation}
			\label{eq:KatoPonceEstimateII}
			\| \langle \partial \rangle^{\alpha} (fg) \|_{L^2(\R^{d_1} \times \T^{d_2})} \lesssim \| \langle \partial \rangle^\alpha f \|_{L^{p_1}(\R^{d_1} \times \T^{d_2})} \| g \|_{L^{q_1}(\R^{d_1} \times \T^{d_2})} + \| f \|_{L^{p_2}(\R^{d_1} \times \T^{d_2})} \| \langle \partial \rangle^\alpha g \|_{L^{q_2}(\R^{d_1} \times \T^{d_2})}.
		\end{equation}
	\end{proposition}
	From the above we derive the following anisotropic version. Let $\D = \K_1 \times \K_2 \times \K_3$ and write $(x,y) \in \K_1 \times (\K_2 \times \K_3 ) = \D$. Dual variables will be denoted by $\xi$ and $\eta$. We define Fourier multipliers $\langle \partial_x \rangle^\alpha$ and $\langle \partial_y \rangle^\beta$ on $\D$ by
	\begin{equation*}
		(\langle \partial_x \rangle^\alpha f) \widehat (\xi,\eta) = \langle \xi \rangle^\alpha \hat{f}(\xi), \quad (\langle \partial_y \rangle^\beta f) \widehat(\xi,\eta) = \langle \eta \rangle^\beta \hat{f}(\xi,\eta).
	\end{equation*}
	We are ready for the proof of Proposition \ref{prop:AnisotropicLeibnizRuleAppendix}, whose statement is repeated for convenience:
	\begin{proposition}
		\label{prop:AnisotropicLeibnizRuleAppendix}
		Let $\D = \K_1 \times \K_2 \times \K_3$ with $\K_i \in \{ \R; \T \}$, $\alpha,\beta \geq 0$, $\delta > 0$, and $\frac{1}{2} = \frac{1}{p} + \frac{1}{q}$. Then the following estimate holds:
		\begin{equation}
			\label{eq:AnisotropicFractionalLeibnizAppendix}
			\| \langle \partial_x \rangle^\alpha \langle \partial_y \rangle^\beta (u^2) \|_{L^2_{xy}(\D)} \lesssim \| \langle \partial_x \rangle^\alpha \langle \partial_y \rangle^\beta u \|_{L^2_{xy}(\D)} \| u \|_{L^\infty_{xy}(\D)} + \| \langle \partial_x \rangle^{\alpha + \delta} u \|_{L^2_{x} L_y^p(\D)} \| \langle \partial_y \rangle^\beta u \|_{L^\infty_{x} L_y^q(\D)}.
		\end{equation}
	\end{proposition}
	\begin{proof}
		We use a paraproduct decomposition for the $x$ frequencies. Let $P_N$ denote dyadic Littlewood-Paley projections in the $x$-frequencies. First note that by the usual fractional Leibniz rule \eqref{eq:KatoPonceEstimate} we find for the low frequencies:
		\begin{equation*}
			\| \langle \partial_x \rangle^\alpha \langle \partial_y \rangle^\beta P_{\lesssim 1} (u^2) \|_{L^2_{xy}} \lesssim \| \langle \partial_y \rangle^\beta (u^2) \|_{L^2_{xy}} \lesssim \| \langle \partial_y \rangle^\beta u \|_{L^2_x L^2_{y}} \| u \|_{L^\infty_{xy}}.
		\end{equation*}
		By Littlewood-Paley theory, we write
		\begin{equation*}
			\| P_{\gtrsim 1} u \|^2_{L^2_{xy}} = \sum_{N \in 2^{\N_0}} \| P_N u \|_{L^2_{xy}}^2.
		\end{equation*}

							For $N \in 2^{\N_0}$ we estimate the High-Low interaction as follows:
		\begin{equation*}
			\begin{split}
				\| \langle \partial_x \rangle^\alpha \langle \partial_y \rangle^\beta P_N (P_N u P_{\ll N} u ) \|_{L^2_{xy}}
				&\sim N^\alpha \big\| \| \langle \partial_y \rangle^\beta (P_{\sim N} u P_{\ll N} u) \|_{L^2_y} \big\|_{L^2_x} \\
				&\lesssim N^\alpha \big\| \| \langle \partial_y \rangle^\beta P_{\sim N} u \|_{L^2_y} \| P_{\ll N} u \|_{L^\infty_{y}} + \| P_{\sim N} u \|_{L_y^p} \| \langle \partial_y \rangle^\beta P_{\ll N} u \|_{L_y^q} \big\|_{L^2_x} \\
				&\lesssim \| \langle \partial_x \rangle^\alpha \langle \partial_y \rangle^\beta P_N u \|_{L^2_{xy}} \| P_{\ll N} u \|_{L^\infty_{xy}} \\
				&\quad + N^{- \delta} \| \langle \partial_x \rangle^{\alpha + \delta} P_N u \|_{L_x^2 L^p_{y}} \| P_{\ll N} \langle \partial_y \rangle^\beta u \|_{L_{x}^\infty L_y^q}.
			\end{split}
		\end{equation*}
		The claim follows from square summation. The High-High-Low interaction can be estimated likewise.
	\end{proof}
	\begin{remark}
		The $\delta$-derivative loss can likely be removed by modifying the arguments from \cite{Grafakos2012} to the anisotropic setting, but we do not need this in the following.
	\end{remark}

	\section{Semilinear ill-posedness issues}\label{sec:IllPosed}
	In this section, we prove that the Cauchy problem for \eqref{eq:FKPI} is semilinearly ill-posed for initial data in $H^{s_1,s_2}(\R^3)$ for any $(s_1,s_2) \in \R^2$ and $\alpha < \frac{15}{7}$, namely we prove that the flow-map for \eqref{eq:FKPI} cannot be $C^2$-differentiable at the origin. For similar results for the two-dimensional equation on $\R^2$, we refer to \cite{MST, LPS, SS22}.

	\begin{theorem}\label{thm:illposed}
		Let $\alpha<\frac{15}{7}$, $\bar{s} \in \R^2$ and $\D=\R^3$. Then there exists no $T>0$ such that \eqref{eq:FKPI} admits a unique local solution defined on the interval $[-T,T]$ such that the flow-map for \eqref{eq:FKPI}
		\begin{equation*}
			\Gamma: u_0 \mapsto u,
		\end{equation*}
		is $C^2$-differentiable at zero from $H^{\bar{s}}(\R^3)$ to $C([-T,T],H^{\bar{s}}(\R^3))$.
		\begin{proof}
			We consider the following Cauchy problem for $\gamma \in \R$:
			\begin{equation}\label{eq:ScaledData}
				\left\{ \begin{array}{cl}
					\partial_t u -D_x^{\alpha}\partial_{x} u - \partial_{x}^{-1} \Delta_y u &= \partial_x(u^2), \quad (x,y,t) \in \R \times \R^2 \times \R, \\
					u(0) &= \gamma \phi \in H^{s_1,s_2}(\R^3).
				\end{array} \right.
			\end{equation}
			Suppose that $u(\gamma, x,y,t)$ solves \eqref{eq:ScaledData}. Fix $T>0$ such that the map $\Gamma$ is $C^2$ and let $t\in (0,T)$. Then,
			\begin{equation}
				u(\gamma,x,y,t)= \gamma S_{\alpha}(t)\phi(x,y) + \int_0^t S_{\alpha}(t-t')u(\gamma,x,y,t')u_x(\gamma,x,y,t')dt'.
			\end{equation}
			Then \begin{equation}
				\begin{split}
					\frac{\partial u}{\partial \gamma}(0,x,y,t) &= S_{\alpha}(t)\phi(x,y)=:u_1(x,y,t)\\
					\frac{\partial^2 u}{\partial \gamma^2}(0,x,y,t) &= 2\int_0^t S_{\alpha}(t-t')u_1(x,y,t')\partial_x u_1(x,y,t')=:u_2(x,y,t).
				\end{split}
			\end{equation}
			The $C^2$ assumption enables us to write
			\begin{equation*}
				u(\gamma,x,y,t)=\gamma u_1(x,y,t)+\frac{\gamma^2}{2!} u_2(x,y,t)+ o(\gamma^2),
			\end{equation*}
			and
			\begin{equation}\label{eq:Cont}
				\|u_2 (\cdot,\cdot,t)\|_{H^{\bar{s}}(\R^3)} \lesssim \|\phi\|_{H^{\bar{s}}(\R^3)}^2.
			\end{equation}
			We show that there exists initial data $\phi$ such that \eqref{eq:Cont} fails. For $\xi \in \R, \eta\in \R^2$, define $\phi$ by its Fourier transform as follows:
			\begin{equation}
				\hat{\phi}(\xi,\eta) = |D_1|^{-\frac{1}{2}} \mathbf{1}_{D_1}(\xi,\eta) + |D_2|^{-\frac{1}{2}} N^{-s_1-(1+\frac{\alpha}{2})s_2}\mathbf{1}_{D_2}(\xi,\eta),
			\end{equation}
			where $D_1$ and $D_2$ are defined as follows and $|D_i|$ denotes the measure of the sets $D_i, i=1,2$:
			\begin{equation}\label{eq:DefinitionSets}
				\begin{split}
					D_1&:= [\beta/2,\beta]\times [-\sqrt{\alpha+1}\beta^2,\sqrt{\alpha+1}\beta^2]\times [-\beta^{\frac{1}{2}+2\delta}, \beta^{\frac{1}{2}+2\delta}],\\
					D_2&:= [N,N+\beta] \times [\sqrt{\alpha+1}N^2, \sqrt{\alpha+1}N^2+\beta^2]\times [-N^{\frac{1}{2}-\delta}, N^{\frac{1}{2}-\delta}].
				\end{split}
			\end{equation}
			Here $N,\beta,\delta>0$ are real numbers such that $N\gg 1, \beta,\delta \ll 1$ and will be chosen later.
			A simple computation gives that $\|\phi\|_{H^{\bar{s}}(\R^3)}\sim 1$. Using \cite[Lemma 4]{MST2}, we can write $u_2$ as follows:
			\begin{equation}
				u_2(x,y,t)=c\int_{\R^4}e^{i(x\xi+y.\eta+t(\xi|\xi|^{\alpha}+|\eta|^2/\xi))} \frac{\xi e^{it(\tau-\xi|\xi|^{\alpha}-\frac{|\eta|^2}{\xi})}} {\tau-\xi|\xi|^{\alpha}-\frac{|\eta|^2}{\xi}}\widehat{u_1}\ast \widehat{u_1}(\xi,\eta,\tau) d\xi d\eta d\tau.
			\end{equation}
			Using the definition of $u_1$, the expression for $\widehat{u_1}\ast \widehat{u_1}$ is given by
			\begin{equation}
				\widehat{u_1}\ast \widehat{u_1}(\xi,\eta,\tau) = \int_{\R^3}\delta\Big(\tau-\xi_1|\xi_1|^{\alpha} -\frac{|\eta_1|^2}{\xi_1}-\xi_2|\xi_2|^{\alpha}-\frac{|\eta_2|^2}{\xi_2}\Big)\hat{\phi}(\xi_1,\eta_1)\hat{\phi}(\xi_2,\eta_2)d\xi_1 d\eta_1.
			\end{equation}
			Set
			\begin{equation*}
				\Phi(x,y,t,\xi_1,\eta_1,\xi_2,\eta_2):= \xi e^{i(x\xi+y\cdot \eta +t (\xi|\xi|^{\alpha}+|\eta|^2/\xi))} \frac{e^{-it\Omega_{\alpha}(\xi_1,\eta_1,\xi_2,\eta_2)}-1}{\Omega_{\alpha}(\xi_1,\eta_1,\xi_2,\eta_2)}.
			\end{equation*}
			We split $u_2$ into three parts:
			\begin{equation*}
				u_2(x,y,t) = c(f_1(x,y,t)+f_2(x,y,t)+f_3(x,y,t)),
			\end{equation*}
			where
			\begin{equation*}
				\begin{split}
					f_1(x,y,t) &= \frac{c}{|D_1|}\int_{\substack{{(\xi_1,\eta_1)\in D_1}\\ {(\xi_2,\eta_2)\in D_1}}} \Phi(x,y,t,\xi_1,\eta_1,\xi_2,\eta_2)d\xi_1 d\eta_1 d\xi_2 d \eta_2,\\
					f_2(x,y,t)&= \frac{c}{|D_2|N^{2(s_1+(1+\frac{\alpha}{2})s_2)}}\int_{\substack{{(\xi_1,\eta_1)\in D_2}\\ { (\xi_2,\eta_2)\in D_2}}} \Phi(x,y,t,\xi_1,\eta_1,\xi_2,\eta_2)d\xi_1 d\eta_1 d\xi_2 d \eta_2,\\
					f_3(x,y,t)&=	\frac{c}{|D_1|^{\frac{1}{2}} |D_2|^{\frac{1}{2}}N^{s_1+(1+\frac{\alpha}{2})s_2}} \Big(\int_{\substack{{(\xi_1,\eta_1)\in D_1}\\{(\xi_2,\eta_2)\in D_2}}} +\int_{\substack{{(\xi_1,\eta_1)\in D_2}\\ {(\xi_2,\eta_2)\in D_1}}} \Big)\Phi(x,y,t,\xi_1,\eta_1,\xi_2,\eta_2)d\xi_1 d\eta_1 d\xi_2 d \eta_2.
				\end{split}
			\end{equation*}
			We focus on the high-low interaction viz. $f_3$. The spatial Fourier transform of the same is given by
			\begin{equation*}
				\widehat{f_3}(\xi,\eta,t)=\frac{c\xi e^{it(\xi|\xi|^{\alpha}+|\eta|^2/\xi)}}{|D_1|^{\frac{1}{2}} |D_2|^{\frac{1}{2}}N^{s_1+(1+\frac{\alpha}{2})s_2}}\Big(\int_{\substack{{(\xi_1,\eta_1)\in D_1}\\ {(\xi_2,\eta_2)\in D_2}}} +\int_{\substack{{(\xi_1,\eta_1)\in D_2}\\ {(\xi_2,\eta_2)\in D_1}}} \Big) \frac{e^{-it\Omega_{\alpha}(\xi_1,\eta_1,\xi_2,\eta_2)}-1}{\Omega_{\alpha}(\xi_1,\eta_1,\xi_2,\eta_2)}d\xi_1d\eta_1.
			\end{equation*}

			Employing \cite[Lemma 5]{MST2}, we have the following bound on the size of the resonance function.
			\begin{lemma}[Size of the resonance function]
				Let $(\xi_1,\eta_1)\in D_1, (\xi_2,\eta_2)\in D_2$ or $(\xi_1,\eta_1)\in D_1, (\xi_2,\eta_2)\in D_2$, then
				\begin{equation*}
					|\Omega_{\alpha}(\xi_1,\eta_1,\xi_2\,\eta_2)| \lesssim  N^{\alpha-1}\beta^2.
				\end{equation*}
				\begin{proof}
					We first note that we can relate the resonance functions in the two and three-dimensional cases as follows:
					\begin{equation}\label{eq:2d3dResonance}
						\Omega_{\alpha}^{3d}(\xi_1,\eta_1,\mu_1, \xi_2,\eta_2,\mu_2) = \Omega_{\alpha}^{2d}(\xi_1,\eta_1,\xi_2,\eta_2) -\frac{(\xi_1\mu_2-\xi_2\mu_1)^2}{\xi_1\xi_2(\xi_1+\xi_2)},
					\end{equation}
					where the notation is self explanatory. Consequently,
					\begin{equation*}
						|\Omega_{\alpha}^{3d}(\xi_1,\eta_1,\mu_1,\xi_2,\eta_2,\mu_2) | \lesssim \max \Big(|\Omega_{\alpha}^{2d}(\xi_1,\eta_1,\xi_2,\eta_2)|, \frac{(\xi_1\eta_2-\xi_2\eta_1)^2}{|\xi_1\xi_2(\xi_1+\xi_2)|}\Big).
					\end{equation*}
					From \cite[Theorem 1.2]{SS22}, we have the following bound:
					\begin{equation*}
						|\Omega^{2d}_{\alpha}(\xi_1,\eta_1,\xi_2,\eta_2)| \lesssim N^{\alpha-1}\beta^2.
					\end{equation*}
					We bound the second term on the right-hand side of \eqref{eq:2d3dResonance} as follows:
					\begin{equation*}
						\frac{(\xi_1\eta_2-\xi_2\eta_1)^2}{|\xi_1\xi_2(\xi_1+\xi_2|} \lesssim \max\Big( \frac{|\xi_1\eta_2|^2}{|\xi_1\xi_2(\xi_1+\xi_2)|}, \frac{|\xi_2\eta_1|^2}{|\xi_1\xi_2(\xi_1+\xi_2)|}\Big) \lesssim \max(\beta N^{-1-2\delta}, \beta^{4\delta}).
					\end{equation*}
					
					We  choose $\delta \ll 1$ so that
					\begin{equation}\label{eq:FinalBound}
						|\Omega_{\alpha}^{3d}| \lesssim N^{\alpha-1}\beta^2.
					\end{equation}
				\end{proof}
			\end{lemma}

			We continue with the proof of Theorem \ref{thm:illposed} and denote $\Omega_{\alpha}^{3d}$ by $\Omega_{\alpha}$ as there is no confusion.\\
			
			\emph{Proof (ctd.)}
			We choose $N, \beta,\delta$ such that $N^{\alpha-1}\beta^2 \sim N^{-\theta},$ i.e., $\beta \sim N^{\frac{1}{2}-\frac{\alpha}{2}-\frac{\theta}{2}}$ and  $\delta=\frac{\theta}{2}$ for $0 <\theta \ll 1$. Then,
			\begin{equation*}
				\Big | \frac{e^{it\Omega_{
							\alpha}(\xi_1,\eta_1,\xi_1,\eta_2)} - 1}{\Omega_{\alpha
					}(\xi_1,\eta_1,\xi_2,\eta_2)}\Big| = |t| + O(N^{-\epsilon}).
			\end{equation*}
			We calculate the $H^{\bar{s}}(\R^3)$ norm of $f_3(t)$:
			\begin{equation}
				\|f_3(\cdot, \cdot,t)\|_{H^{\bar{s}}(\R^3)} \gtrsim N \frac{|D_1|^{\frac{1}{2}}} {|D_2|^{\frac{1}{2}}} \beta^{\frac{3}{2}} N^{\frac{1}{4}-\frac{\delta}{2}} = N|D_1|^{\frac{1}{2}}.
			\end{equation}
			
			From \eqref{eq:Cont}, we have
			\begin{equation*}
				1\sim \|\phi\|_{H^{\bar{s}}(\R^3)} ^2 \gtrsim N\beta^{\frac{7}{4}+\delta} =N^{(\frac{15}{8}-\frac{7\alpha}{8})-},
			\end{equation*}
			which is true only if $\alpha\geqslant \frac{15}{7}$.
		\end{proof}
	\end{theorem}
	
	\subsection*{Acknowledgment}
	Funded by the Deutsche Forschungsgemeinschaft (DFG, German Research Foundation) -- Project-ID 317210226 -- SFB 1283 (S.H.) and Project-ID -- 258734477 -- SFB 1173 (R.S.).


	\bibliographystyle{abbrv}

\end{document}